\documentclass[10pt]{amsart}
\usepackage{amsmath,amsfonts,amssymb,amsthm}
\usepackage{graphics,graphicx}
\usepackage[colorlinks,hyperindex,linkcolor=blue,
citecolor=purple]{hyperref}
\hypersetup{
pdfauthor = {B1},
 pdftitle= {Generalized Pseudo-Holomorphic Curves},
 pdfsubject = {Stuff}}
\usepackage{tikz-cd}
\usepackage{tikz}
\usepackage[all]{xy}
\usepackage{subfiles}
\usetikzlibrary{matrix,arrows,decorations.pathmorphing}
\usepackage{xspace}

\newcommand{\Wedge}{\scalebox{1.2}{$\wedge$}}

\newcommand{\Ccal}{{\mathcal{C}}}
\newcommand{\Dcal}{{\mathcal{D}}}

\newcommand{\Jcal}{{\mathcal{J}}}

\newcommand{\Ocal}{{\mathcal{O}}}
\newcommand{\Pcal}{{\mathcal{P}}}

\newcommand{\CC}{\mathbb{C}}

\newcommand{\RR}{\mathbb{R}}

\newcommand{\TT}{\mathbb{T}}

\newcommand{\VV}{\mathbb{V}}

\newcommand{\ZZ}{\mathbb{Z}}

\newcommand{\Z}{\mathbb{Z}}

\newcommand{\Hom}{\operatorname{Hom}}

\newcommand{\id}{\textbf{\textit{I}}}

\newcommand{\End}{\operatorname{End}}

\newcommand{\codim}{{\operatorname{codim}}}

\newcommand{\set}[1]{\lbrace #1\rbrace}

\newcommand{\rr}{\mathbb{R}}

\newcommand{\cc}{\mathbb{C}}

\newcommand{\C}{\mathcal{C}}

\DeclareMathOperator{\Sym}{Sym}

\DeclareMathOperator{\Ho}{Ho}

\renewcommand{\hat}{\widehat}

\newcommand{\Abf}{\mathbf{A}}
\DeclareMathOperator{\type}{type}

\newcommand{\Fuk}{{\operatorname{Fuk}}}
\newcommand{\Dbcoh}{{\operatorname{D^b_{coh}}}}
\newcommand{\ii}{\mathbf{i}}
\newcommand{\D}{{\mathcal{D}}}
\usepackage{thmtools}
\declaretheorem[style=definition,qed=$\diamondsuit$]{definition}
\declaretheorem[style=definition,qed=$\triangle$,sibling=definition]{example}
\declaretheorem[style=plain,sibling=definition]{theorem}
\declaretheorem[style=plain,sibling=definition]{lemma}
\declaretheorem[style=plain,sibling=definition]{proposition}
\declaretheorem[style=plain,sibling=definition]{corollary}
\declaretheorem[style=definition,qed=$\diamondsuit$,sibling=example]{claim}

\declaretheorem[style=definition,qed=$\diamondsuit$,sibling=claim]{remark}

\declaretheorem[style=definition,sibling=example]{conjecture}
\newtheorem*{claim*}{Claim}
\numberwithin{theoremalpha}{section}

\numberwithin{equation}{section}
\numberwithin{definition}{section}
\numberwithin{theorem}{section}
\numberwithin{proposition}{section}
\numberwithin{lemma}{section}
\numberwithin{example}{section}
\numberwithin{remark}{section}
\numberwithin{corollary}{section}
\numberwithin{question}{section}
\numberwithin{problem}{section}
\numberwithin{assumption}{section}

\newtheorem*{theorem*}{Theorem}

\newcommand{\btrans}[1]{ 
e^{#1}
}
\newcommand{\betatrans}[1]{
e^{#1}
}

\newcommand{\tbeta}{\tilde{\beta}}
\newcommand{\tB}{\tilde{B}}

\newtheoremstyle{named}{}{}{\itshape}{}{\bfseries}{.}{.5em}{\thmnote{#3}#1}
\theoremstyle{named}
\newtheorem*{namedtheorem}{}

\newtheoremstyle{nameddefinition}{}{}{\normalfont}{}{\bfseries}{.}{.5em}{\thmnote{#3}#1}
\theoremstyle{nameddefinition}
\newtheorem*{nameddefinition}{}

\numberwithin{equation}{section}
\title{The derived category of a generalized complex torus}
\author{Patrick Antweiler}
\address{Fachbereich Mathematik, Universität Hamburg, Germany}
\email{patrick.antweiler@uni-hamburg.de}

\author{Vicente Cortés}
\email{vicente.cortes@uni-hamburg.de}

\author{Julian Holstein}
\email{julian.holstein@uni-hamburg.de}

\author{Markus Röser}
\email{markus.roeser@uni-hamburg.de}

\author{Aldo Witte}
\email{aldo.witte@uni-hamburg.de}

\begin{document}

\begin{abstract}
    For any translation invariant generalized complex structure $\Jcal$ on a torus we construct the derived category $\D(\Jcal)$ generalizing the derived category of coherent sheaves on a complex torus and the affine/convergent Fukaya category of a symplectic torus.
  The category is defined in terms of a non-commutative curved deformation of the Dolbeault algebra of a complex torus.
  We show that this category is independent of the chosen deformation path, which we prove using non-commutative Fourier-Mukai duality. 
  We also prove that the action of $O(2n,2n; \mathbb Z)$ on the space of generalized complex structures leaves the category $\D(\Jcal)$ invariant. In particular, T-duality acts trivially on $\D(\Jcal)$.
  We construct concrete objects associated to affine Lagrangians in the derived category associated to a symplectic torus and show that their spaces of morphisms are given by Floer cohomology. 
  We show that our construction is consistent with the homological mirror symmetry conjecture (including its non-commutative version).
\end{abstract}
\maketitle
\tableofcontents
\setlength{\parskip}{3mm plus0.4mm minus0.4mm}
\section{Introduction}
\emph{Homological mirror symmetry} (HMS), as conjectured by Kontsevich \cite{HAMS}, is a fundamental duality between symplectic and complex geometry. It is described as an appropriate equivalence between the \emph{Fukaya category} of a symplectic manifold (`The A-model') and the \emph{derived category of coherent sheaves} of a complex manifold (`The B-model'). HMS provides a precise mathematical statement in which way type IIA and type IIB string theories can be considered equivalent.

Although a priori a duality between complex and symplectic manifolds, already in \cite{Fukaya02} it is postulated that the mirror of a generic complex torus should be described as a non-commutative space. To approach these spaces in a geometric manner, we will use the framework of \emph{generalized complex (GC) geometry}.

 As introduced in \cite{Hitchin,Gualtieri:2003dx} GC geometry arose from the search for the mathematical structure of the $\mathcal N=(2,2)$ supersymmetric sigma model (\cite{Witten}): the physical theory describing the aforementioned string theories. GC geometry provides a framework in which symplectic and complex geometry are unified into a common structure, but it further describes interpolations between these structures as well as non-commutative deformations of complex structures.

 There is strong evidence that there should be a derived category associated to a GC structure unifying the A and B-model, which goes by the name of `category of GC branes'. For instance, Gualtieri shows that infinitesimal deformations of the B-model are equivalent to the deformations of a complex structure within the space of GC structures. Therefore, GC geometry is a natural place to study mirror symmetry and over the last 25 years there has been significant work in this direction, see e.g.\ \cite{CG11}, \cite{BenBassat2006}, \cite{kobayashi2021remarks} and \cite{BischoffGualtieri2022}. However, no satisfactory definition of the category of GC branes has been achieved.

In this paper we propose a new approach to constructing the derived category $\D(\Jcal)$ of a GC manifold $(M, \Jcal)$ and carry it out in the case of tori. 

The desiderata of such a category are as follows. (We refer to Section \ref{sec:GC} for the relevant background on GC geometry):
\begin{enumerate}
    \item If $X = (M,J)$ is a \emph{complex manifold} and $\mathcal J_J$ the associated generalized complex structure then $\D(\Jcal_J)$ should be equivalent to $\Dbcoh(X)$, the bounded derived category of complexes of sheaves with coherent cohomology on $X$.
    \item If $(M, \omega)$ is a \emph{symplectic manifold} and $\Jcal_\omega$ the associated generalized complex structure then $\D(\Jcal_\omega)$ should be equivalent to (some version of) the Fukaya category.
    \item If $(M, \Jcal)$ and $(\tilde M, \tilde \Jcal)$ are integrally \emph{T-dual} their derived categories should be equivalent.  
    In particular, in view of \cite{Strominger}, if $(M, J)$ and $(\check M, \omega)$ are \emph{mirror} then their derived categories should be equivalent. 
    \item \emph{Deformations} of generalized complex manifolds should induce deformations of derived categories. Starting with a complex manifold $X = (M, J)$ this should give an isomorphism on the deformation spaces which are both described in terms of \eqref{eq:defspaces} below.
    \item An \emph{integral $B$-transform} of a generalized complex structure
    should not change the derived category as it is expected to act as a gauge transformation, see \cite[Section 3.1]{kapustin2005branes} and \cite[Section 3.8]{Gualtieri:2003dx}.
\end{enumerate}
T-duality is physically described as an equivalence between quantum field theories (\cite{Bus87}), and mathematically as a duality between manifolds admitting free torus actions \cite{BEM04}. It can be applied very effectively in the setting of generalized geometry, where it can be used to encompass both Fourier--Mukai duality between complex manifolds, as well as the SYZ type mirror symmetry dualities between A and B models (\cite{Strominger}).

Gualtieri shows that the space of first-order deformations of a complex structure viewed as a GC structure is described by 
\begin{equation}\label{eq:defspaces}
    H^2(X, \mathcal O) \oplus H^1(X, \mathcal T_X)\oplus H^0(X, \Wedge^2 \mathcal T_X).
\end{equation}
This space is considered as the ``extended deformation space'' by Barannikov and Kontsevich in \cite{barannikov1998frobenius}; GC geometry provides a clear geometric interpretation of this space. The component $H^1(X, \mathcal T_X)$ provides deformations as a complex manifold, as studied by Kodaira. The component $H^2(X, \mathcal O)$ describes deformations coming from \emph{B-field/gerby transformations}:
Given a $B$-field $B \in \Omega^2(X)$ and a GC structure $\Jcal$, the $B$-transformed GC structure is given by $e^{B}\Jcal e^{-B}$. 
The component $H^0(X, \Wedge^2 \mathcal T_X)$ describes \emph{non-commutative deformations}:
Given a bivector $\beta \in \mathfrak{X}^2(M)$ and a GC structure $\Jcal$, the $\beta$-transformed GC structure is given by $e^{\beta}\Jcal e^{-\beta}$.
The deformation space \eqref{eq:defspaces} also controls first order deformations of the derived category $\Dbcoh(X)$, induced by deformations of the complex structure, twisting sheaves by a gerbe and non-commutative deformations of the structure sheaf, respectively.

In this paper we propose a definition of $\D(\Jcal)$ for translation invariant GC structures on the torus using deformation theory: We write a GC structure as a certain type of deformation of a purely complex structure and then perform an analogous deformation on the level of derived categories.

The starting point is the following theorem by Block, expressing the derived category of a compact complex manifold in terms of its Dolbeault algebra:
\begin{theorem*}[4.1.3 in \cite{BlockNC1}]
Let $X$ be a compact complex manifold, and consider its Dolbeault algebra $\mathbf{A} = (\Omega^{0,\bullet}(X),\bar{\partial},0)$. Then the homotopy category of cohesive modules $\Pcal_{\mathbf{A}}$ is equivalent to $\Dbcoh(X)$. 
\end{theorem*}
Here $\Pcal_{\mathbf{A}}$ denotes the dg-category of \emph{cohesive modules} over the Dolbeault algebra: those modules whose underlying graded modules are bounded, finitely generated and projective. Cohesive modules give a refinement of the classical derived category of a dg algebra. They will play an important role in this paper and we refer to Section \ref{sec:cohmod} for the relevant background. 

This set-up allows us to deform the derived category by deforming the Dolbeault algebra. 
Consider first the case of a $B$-field transformation of a complex structure.
Given a complex manifold $(X,I)$ and 2-form $B$ we aim to construct the derived category of the GC structure $e^{B}\Jcal_I e^{-B}$.
We take the $(0,2)$-part of $B$ with respect to the given complex structure $I$ and use it to deform the Dolbeault algebra to a \emph{curved differential graded algebra} 
\begin{equation*}
    \mathbf{A}(I,0,B) = (\Omega^{0,\bullet}(X),\bar{\partial},2\pi\ii B^{0,2}).
\end{equation*}
The differential $d$ of a curved differential graded algebra no longer needs to square to zero. Rather, there is a prescribed curvature element $h$ in degree two such that $d^2 = [h,-]$.
Note that for $\mathbf{A}(I,0,B)$ the differential $\bar\partial$ does satisfy  $\bar{\partial}^2 =0 = [2\pi\ii B^{0,2},-]$ as the wedge-product is graded commutative. 
Even in the commutative case $B^{0,2}$ changes the notion of module over an algebra. 
In this case: 
\begin{theorem*}[{\cite{BlockNC1}, \cite{Antweiler24}}]
$\Ho \Pcal_{\mathbf{A}(I,0,B)}$ is equivalent to the bounded derived category of complexes of \emph{$B$-twisted sheaves} of $\Ocal_X$-modules with coherent cohomology.
\end{theorem*}
We view the derived category of twisted sheaves as the derived category of $e^{B}\Jcal_I e^{-B}$.

In this paper we generalize this approach. Corresponding to $\beta$-transforms in GC geometry are non-commutative deformations of derived categories.
These deformations are much harder to study than derived categories of twisted sheaves and while non-commutative deformations always exist formally we are interested in strict (or geometric) deformations. 

Hence we restrict our study in this paper to tori, whose Dolbeault algebras have a geometric theory of noncommutative deformations.
Given a translation-invariant GC structure on the torus $V/\Gamma$ which has \emph{even cotype} (a parity condition which can always be achieved by performing T-duality, see Definition \ref{def:cotype}), we may express it as a $B$-transform followed by a $\beta$-transform of a complex structure $I$:
\begin{namedtheorem}[Proposition \ref{proposition:B_then_beta_eventype}]
    Let $\Jcal$ be a translation invariant GC structure on $V$ of even cotype. 
        There exist $B\in\Wedge^2 V^*, \beta\in \Wedge^2 V$ and $J\in\mathcal C_V$, such that 
        \[\Jcal = \betatrans{\beta }\btrans{B }\Jcal_J\btrans{-B }\betatrans{-\beta }.\] 
\end{namedtheorem}
We associate to this data a \emph{generalized Dolbeault algebra} which will be a curved noncommutative deformation of the Dolbeault algebra defined by
$$\mathbf{A}(I, \beta, B) = (C^\infty(V/\Gamma) \otimes \Wedge^\bullet (V_{\mathbb C}^*)^{0,1}, \star^\beta, \bar \partial_{I}, 2 \pi \ii B^{0,2}).$$ 

Here we use $\beta$ to deform the commutative product on smooth functions on the torus, tensor with an exterior algebra over $\bar V$, take the usual Dolbeault differential induced by $I$ and consider $2\pi \ii B^{0,2}$ as a curvature term. As the curvature is central the differential still squares to zero.

The algebra $\mathbf{A}(I, \beta, B)$ depends on the choices we have made, but it leads to a well-defined derived category:

\begin{nameddefinition}[Definition \ref{definition:gencategory}]
Let $\Jcal$ be a translation-invariant GC structure on $V/\Gamma$ of even cotype. Choose $I$, $\beta$ and $B$ such that 
$\Jcal =\betatrans{\beta }\btrans{B } \Jcal_I\btrans{-B}\betatrans{-\beta}$.
The \emph{derived category} of $\Jcal$ is defined as 
\[ \Dcal(\Jcal) = \Pcal_{\mathbf{A}(I,\beta,B)}. \] 
If $\Jcal$ is of odd cotype we let $\Jcal'$ be the GC structure obtained by performing T-duality relative to any circle.
We define $\D(\Jcal) := \Dcal(\Jcal')$.
\end{nameddefinition}

We then have our main result:
\begin{namedtheorem}[{Theorem \ref{theorem:welldefined}}]
Let $\mathcal J$ be a translation-invariant GC structure of even cotype on a torus. Then the generalized derived category    $\D(\Jcal)$ is well-defined up to $A_\infty$-quasi-equivalence.
\end{namedtheorem}
We may extend the definition to the odd case and obtain a key invariance result: 
\begin{namedtheorem}[{Corollary \ref{corollary:welldefinedandinvariant}}]
The generalized derived category of a $2n$-dimensional torus is well-defined also for translation invariant GC structures of odd cotype and satisfies  \emph{integral invariance}, i.e.\ it is invariant under the natural action of $SO(2n,2n; \Z)$ on GC structures.
\end{namedtheorem}

This satisfies our list of desiderata (restricting to affine T-duality and assuming homological mirror symmetry): 
\begin{enumerate}
    \item If $\Jcal$ is a purely complex structure this reduces to Block's theorem \cite{BlockNC1}.
    \item For a symplectic torus we obtain the convergent Fukaya category assuming  homological mirror symmetry, see Corollary \ref{cor:fukaya}. 
    Note that for generic symplectic structures this depends on a non-commutative homological mirror symmetry conjecture that we formulate below.
    \item Invariance under affine integral T-duality is a special case of integral invariance, see Lemma \ref{lem:Tdualintegral}. Mirror symmetry appears as a special case of affine integral T-duality, see Corollary \ref{cor:integralconsequences}.
    \item We identify deformation spaces of a complex structure in terms of GC structures, curved noncommutative Dolbeault algebras and derived categories in  Section \ref{sect:deformations}.
    \item Invariance under integral $B$-transforms is another special case of integral invariance, see Corollary \ref{cor:integralconsequences}.
\end{enumerate}

We elaborate on the relation to symplectic topology and mirror symmetry.

Given a translation invariant GC structure $\Jcal$ on the torus we define its mirror $\check \Jcal$ and prove in Theorem \ref{thm:nchms} that their derived categories agree.
The mirror itself depends on choices but its derived category does not.
In the special case that $\Jcal$ comes from a $B$-symplectic structure admitting the structure of Lagrangian fibration its mirror is exactly the GC structure associated to the complex torus that is mirror to the symplectic torus.

Thus Kontsevich's homological mirror symmetry conjecture for the torus is equivalent to the fact that our generalized derived category agrees with the Fukaya category.
Note that we consider here the convergent Fukaya category over $\CC$, not the more general construction over the Novikov ring, see Remark \ref{rem:novikov}.

It has been stipulated since \cite{Fukaya02} that the mirror of a generic translation invariant $B$-symplectic torus should be a non-commutative torus.
We can make precise the noncommutative version of the homological mirror symmetry conjecture:
\begin{namedtheorem}[Conjecture \ref{conj:nchms}]
    Given a torus with translation invariant $B$-symplectic form $\omega + \ii B$ and associated GC structure $\Jcal_{\omega+\ii B}= \btrans{B}\Jcal_\omega\btrans{-B}$ its Fukaya category is given by the derived category of the mirror of $\Jcal_{\omega + \ii B}$.  
\end{namedtheorem}
If $X$ admits an affine Lagrangian its mirror may be chosen to be a noncommutative torus or a gerby torus, see Corollary \ref{cor:nchms}.

We also exhibit explicit noncommutative curved modules that correspond to affine Lagrangians and verify that their derived hom spaces agree with Floer cohomology in agreement with the computations in \cite{Fukaya02}.

Given any symplectic torus $(V/\Gamma, \omega)$ we may choose a Hyperkähler triple $g,I,J,K$ with $\omega_J = \omega$. 
We find
$$\D(\Jcal_\omega) = \mathcal P_{\mathbf A(I,-\omega_K^{-1}, \omega_K / 2)}.$$
Let $(L,\alpha), (\tilde L,\tilde \alpha)$ be two affine Lagrangian submanifolds
of $V/\Gamma$ with local systems $\alpha$ and $\tilde \alpha$. 
To the Lagrangians we associate cohesive modules $\tilde N(L,\alpha)$ and $\tilde N(\tilde L, \tilde \alpha)$ over $\mathbf A(I,-\omega_K^{-1}, \omega_K / 2)$ and prove:
\begin{namedtheorem}[Theorem \ref{thm:FloerCohomology}]
    With notation as above there is an isomorphism:
    \[ H^\bullet(\Hom^\bullet_{\mathbf A(I,-\omega_K^{-1},\omega_K/2)}(\tilde N(L,\alpha),\tilde N(\tilde L, \tilde \alpha))) \cong HF^\bullet( ( L,\alpha), (\tilde L, \tilde \alpha) ).
    \]
\end{namedtheorem}

In the special case that a symplectic torus admits a space filling brane we construct a certain module corresponding to the space filling brane, see Example \ref{eg:sfb}. 
We verify its derived endomorphisms correspond to a non-commutative deformation of the Dolbeault algebra, as predicted in the mathematical physics literature, e.g.\  \cite[Section 11.1]{kapustin2007electric}.

We conclude the introduction with an indication of the proof of well-definedness (Theorem \ref{theorem:welldefined}).
To compare cohesive modules over different generalized Dolbeault algebras we have to construct suitable bimodules inducing an equivalence of categories.
The key insight is that the existence of such bimodules is closely tied to T-duality relative to a point, whose algebraic incarnation is the non-commutative Fourier-Mukai transform constructed in \cite{NCGrauert}.

Indeed if  $V^*/\Gamma^*$ is the dual of $V/\Gamma$ we give a precise condition under which the GC structures
\[\Jcal = \betatrans{\beta }\btrans{B }\Jcal_I\btrans{-B}\betatrans{-\beta}\quad \text{ and }\quad  
\tilde{\Jcal} = \betatrans{\tilde{\beta }}\btrans{\tilde{B }}\Jcal_{\tilde{I}}\btrans{-\tB}\betatrans{-\tbeta}\] 
on $V/\Gamma$ and $V^*/\Gamma^*$ are identified  by T-duality (Lemma \ref{lemma:Tdualflip}), and this is the same condition ensuring that 
a bimodule exists identifying $\Pcal_{\mathbf{A}(I,\beta,B)}$ and $\Pcal_{\mathbf{A}(\tilde I,\tilde\beta,\tilde{B})}$  (Theorem \ref{theorem:inducedequivalences} which is proven in \cite{NCGrauert}). 
Thus, we may compare the derived categories associated to two different representations of the same GC structure by comparing them to the derived categories of their T-duals. 

\subsection{Relation to other work}
This paper is closely related to \cite{NCGrauert} where the general form of non-commutative Fourier Mukai duality, a key ingredient to this paper, is proven.
This work in turn builds on \cite{BlockNC1}, note also the formal version \cite{ben2007non}.

The most substantial progress on categories of GC branes in the literature up to this point is obtained by  Bischoff-Gualteri in \cite{BischoffGualtieri2022}: Their approach to the non-commutative deformations is to integrate the corresponding Poisson manifold to a symplectic groupoid. This reduces the problem to defining the category of GC branes to another hard open problem in symplectic geometry: finding an extension of the Fukaya category which also allows for coisotropic submanifolds as objects. The paper then carries out this construction in detail for toric Poisson structures.

In the mathematical physics literature a description of $A$-branes in terms of non-commutative geometry has been suggested by Kapustin \cite{kapustin2004topological, kapustin2005branes} and we make his idea precise in the case of the torus.

There has since been work on using generalized complex geometry to consider mirror symmetry for non-commutative tori, e.g.\ \cite{kajiura2008noncommutative, kim2009mirror, kobayashi2021remarks}.
For example, in \cite{kajiura2008noncommutative} Kajiura considers noncommutative Dolbeault algebras with a view to homological mirror symmetry and speculates in the concluding remarks that their two-sided twisted complexes (roughly cohesive modules) would form a category of interest, but this has not been made precise.

Finally, there has been significant work on understanding the Fukaya category in terms of deformation quantization, e.g.\ \cite{bressler2002mirror, nest2004remarks, tsygan2013microlocal, kontsevich2024holomorphicfloertheoryi}. 
We will not explore the precise relations to our work here but our result relating derived hom spaces of cohesive modules and Floer cohomology can be viewed as an instance of this program.

\subsection{Organization of the paper}
In Section \ref{sec:GC} we recall the required background on generalized complex geometry, and further obtain the description of invariant GC structures on the torus as a $B$ and $\beta$-transformation of a complex structure in Proposition \ref{proposition:B_then_beta_eventype}. 
In Section \ref{sec:Tdual} we recall the basics of T-duality and its interpretation within generalized geometry. 
In Section \ref{sec:cohmod} we recall the definition of curved differential graded algebras and the cohesive modules over them.
In Section \ref{sect:deformations} we explain how to deform the Dolbeault algebra of a complex  torus  in curved and noncommutative directions to obtain the generalized Dolbeault algebra $\mathbf{A}(I,\beta,B)$.
We also compare the deformation theory of the complex structure as a GC structure, the deformation theory of its Dolbeault algebra and of its derived category. 
In Section \ref{section:Heisenberg} we introduce the Heisenberg modules which will provide the appropriate Morita equivalences between different Generalized Dolbeault algebras and show the relation between T-duality and the existence of non-commutative Poincar\'e line bundles. 
In Section \ref{sec:gencat} we finally define the derived category of a generalized complex torus, and show that it is well-defined and satisfies integral invariance. 
In Section \ref{sect:fukaya} we discuss homological mirror symmetry and the relationship of our generalized derived category and the Fukaya category. 
We also construct cohesive modules associated to affine Lagrangians and compute their hom spaces to compare them with Floer cohomology.

\subsection{Conventions}
Throughout $V$ is an $\RR$-vector space.
The linear dual space of  $V$ is denoted by $V^*$.
Lattices will be denoted by $\Gamma \subset V$ and the dual lattice $\Gamma^* \subset V^*$ is defined as \[\Gamma^* := \{ \alpha \in V^* | \alpha(\Gamma) \subset \mathbb Z \}.\]

Functions will usually be considered with values in $\CC$ unless specified otherwise. 

Generalized complex structures will by denoted by $\Jcal$, and their $+\ii$-eigenbundles as $L_{\Jcal}$.

Given a two form $B \in \Wedge^2 V^*$ on $V$, we also denote by $B  \colon V \to V^*$ the linear map $B (v) := B(v,-)$. Similarly, a bivector $\beta \in \Wedge^2 V$ gives rise to $\beta  \colon V^* \to (V^*)^* = V$, $\beta ( \alpha ) := \beta(\alpha,-)$.

We use cohomological grading convention and complexes may be denoted as in $P^\bullet$.

We denote the $n\times n$ identity matrix by $\id_n$ (or just $\id$).

\subsection{Statement on AI use}
ChatGPT 5.6 was used for help with algebra in the proofs of Lemma \ref{lemma:pert} and Theorem \ref{thm:zinvariance}.
ChatGPT 6 and Claude Opus 5.5 were used to help compare the Hyperkähler Maslov index we introduce in Definition \ref{Def:HK2} to the Maslov index from \cite{Fukaya02}.

\subsection{Acknowledgements}
We are grateful to Jonathan Block for freely sharing his ideas on the approach we take in this paper.
We also thank Marco Gualtieri for several insightful discussions and Francis Bischoff, Ailsa Keating, Tobias Dyckerhoff, Janko Latschev, Tony Pantev, Nick Sheridan, Jörg Teschner for useful conversations.

All authors are supported by the Deutsche Forschungsgemeinschaft (DFG, German Research Foundation) -- SFB 1624 -- ``Higher structures, moduli spaces and integrability'' -- 506632645. 
The second and third author are also supported by the DFG through EXC 2121 ``Quantum Universe'' -- project number 390833306

\section{Review of generalized complex geometry}\label{sec:GC}

In this section we set up our notation and review a number of standard facts from the theory of GC geometry. In \S 2.1. we give the relevant general definitions in the context of an arbitrary smooth base manifold. In \S 2.2. we focus on the linear GC structures on a real vector space $V$. This case will be particularly relevant for our constructions in later sections, where we will work with translation invariant GC structures on tori. In particular we study certain symmetries associated to a $2$-form $B\in\Wedge^2V^*$ (so-called $B$-transforms) and a bivector $\beta\in\Wedge^2V$ (so-called $\beta$-transforms). 
The main technical result is Proposition \ref{proposition:B_then_beta_eventype} which allows us to write any linear GC structure of even cotype on $V$ as a $B$-transform followed by a $\beta$-transform applied to an (ordinary) complex structure on $V$.

\subsection{The generalized tangent bundle}

\begin{definition}
    Let $M$ be a smooth manifold. A \emph{Courant algebroid} over $M$ is a vector bundle $E\to M$ together with a non-degenerate scalar product $\langle\cdot,\cdot\rangle\in\Gamma(\Sym^2E^*)$, a bilinear bracket $[\cdot,\cdot]\colon \Gamma(E)\times \Gamma(E)\to\Gamma(E)$ and a bundle map $\rho\colon E\to TM$, called the \emph{anchor}, such that for all $u,v,w\in\Gamma(E)$ the following holds:
    \begin{itemize}
        \item $[u,[v,w]]  = [[u,v],w]+[v,[u,w]]$,
        \item $\rho(u)\langle v,w\rangle = \langle [u,v],w\rangle+\langle v,[u,w]\rangle$,
        \item $\langle [u,v]+[v,u],w\rangle = \rho(w)\langle u,v\rangle$. \hfill \qedhere
    \end{itemize}
\end{definition}

For us the most important example of a Courant algebroid  is the following:
\begin{definition}
The \emph{generalized tangent bundle} is the Courant algebroid with vector bundle
    \[
        E = \TT M = TM\oplus T^*M
    \]
with anchor given by the natural projection $\rho \colon\mathbb{T}M \rightarrow TM$, scalar product given by the natural pairing 
    \[
        \langle v+\xi, v+\xi\rangle = \xi(v)
    \]
and bracket given by the \emph{Dorfman bracket}
    \[ [X+\xi, Y+\eta] = \mathcal L_X(Y+\eta)-\iota_Yd\xi, \]
where $X,Y\in \Gamma(TM), \xi,\eta\in\Gamma(T^*M) = \Omega^1(M)$.
\end{definition}
Like the usual Lie bracket on vector fields, the Dorfman bracket is invariant under the natural action of the group $\mathrm{Diff}(M)$ on $\Gamma(\mathbb{T}M)$. An important new feature of the generalized tangent bundle is the presence of additional symmetries given by \emph{$B$-field transformations}. Let $B\in\Omega^2(M)$ be a closed $2$-form with associated homomorphism $B \colon TM\to T^*M$. This gives rise to the  automorphism
    \[
        \btrans{B } = \begin{pmatrix}
            1 & 0 \\ B & 1
        \end{pmatrix}\in\Gamma(\End(TM\oplus T^*M))
    \]
of the generalized tangent bundle. 
Any automorphism of $(\TT M, \langle\cdot,\cdot\rangle, [\cdot,\cdot],\rho)$ is given by a combination of a $B$-field transformation and a diffeomorphism. 

We are interested in the study of GC structures, which we define next. 

\begin{definition}
    Let $M$ be a smooth manifold. A \emph{generalized complex (GC) structure} on $M$ is an element $\Jcal\in \Gamma(\End(\TT  M))$ such that 
    \begin{enumerate}
        \item $\Jcal^2 = -\id_{\TT M}$,
        \item $\Jcal$ is skew-symmetric with respect to $\langle\cdot,\cdot\rangle$,
        \item $\Jcal$ is integrable in the sense that the $+\ii$-eigenbundle $L_{\Jcal} = \mathrm{Eig}(\Jcal,\ii) \subset \TT M _\CC$ is involutive, i.e. 
        \[
        [\Gamma(L_{\Jcal}),\Gamma(L_{\Jcal})]\subset \Gamma(L_{\Jcal}). \hfill \qedhere
        \]
    \end{enumerate}
\end{definition}

A fundamental fact is that a GC structure $\Jcal$ is characterized by its associated eigenbundle $L_{\Jcal}\subset \TT M_{\CC}$ which is an isotropic  subbundle and satisfies $\TT M_\C = L_{\Jcal}\oplus \overline{L_{\Jcal}}$. 

\begin{remark}\label{remark:signature}
If the manifold  $ M$ admits a generalized (almost) complex structure the dimension of $M$ must be even. 
\end{remark}

The general form of a GC structure is  
\[
\Jcal = \begin{pmatrix}
    A & \pi  \\ \omega  & -A^*
\end{pmatrix}
\]
with $A\in\Gamma(\End(TM)), \omega\in\Omega^2(M), \pi\in\Gamma(\Wedge^2TM)$ satisfying
\[
A^2+\pi\omega =-\id_{TM}, \quad \omega  A - A^*\omega  = 0 = A\pi  -\pi  A^*.  
\]

The integrability of $\Jcal$ may be translated into a set of differential equations satisfied by $A,\omega,\pi$, see \cite{Crainic2011}. 

The associated maximal isotropic subspace is then 
\[
L_{\Jcal} = (\id_{\mathbb TM}-\ii \Jcal)(\mathbb TM) \subset \mathbb TM_\CC.
\]

Since $B$-field transformations preserve the inner product and the Dorfman bracket, they act on the set of GC structures by conjugation: Given a closed $2$-form $B\in \Omega^2(M)$ and a GC structure $\Jcal$, we obtain a new GC structure 
    \[
        \btrans{B }\Jcal\btrans{-B }. 
    \]
In terms of corresponding maximal isotropic subbundles this action reads 
    \[
        L_{\btrans{B }\Jcal\btrans{-B }} = \btrans{B }L_{\Jcal}.
    \]

\begin{definition}\label{def:cotype}
    Let $\Jcal\in\Gamma(\End(\mathbb TM))$ be a GC structure. The \emph{type} of $\Jcal$ at a point $p\in M$ is the number 
    \[
        \mathrm{type}(\Jcal)_p = \codim_{TM_\cc} (\pi(L_{\Jcal}))\in \{0,1,\dots,\frac{1}{2}\dim_{\rr} M\}
    \]
    We say $\Jcal$ \emph{is of type $k\in \{0,\dots,\frac{1}{2}\dim_{\rr} M\}$} if $\mathrm{type}(\Jcal)_p = k$ for all $p\in M$.

    We also define the \emph{cotype} of $\mathcal J$ at $p$ as
    \[ \mathrm{cotype}(\Jcal)_p = \dim_{\mathbb C}(\pi(L_{\mathcal J})) - \frac{1}{2}\dim_{\rr} M = \frac{1}{2}\dim_{\rr} M - \mathrm{type}(\Jcal)_p \in \{0,1,\dots,\frac{1}{2}\dim_{\rr} M\} . \hfill \qedhere\]
\end{definition}

\begin{remark}
    The type of a GC structure is invariant under $B$-field transformations, since these are automorphisms of the Courant algebroid structure on $\TT M$. 
\end{remark}

\begin{example}\label{example:complexsymplecticGC}
    There are two important examples of GC structures of minimal, respectively maximal, possible type: 
    \begin{enumerate}
        \item Let $J\in \Gamma(\End(TM))$ be a complex structure on $M$ (in the ordinary sense). Then we obtain a GC structure via 
        \[
            \Jcal_J = \begin{pmatrix}
                J & 0 \\ 0 & -J^*
            \end{pmatrix}.
        \]
        The corresponding maximal isotropic subbundle of $\mathbb{T}M_\CC$ is given by 
        \[
            L_{\Jcal_J} = \left(\id_{\mathbb TM}-\ii\Jcal_J\right)(\mathbb TM) = TM^{1,0} + T^*M^{0,1}\subset \TT M_\CC.
        \]
        We have 
        \[
            \pi(L_{\Jcal_J}) = TM^{1,0},
        \] 
        which has codimension $n= \dim_{\CC}(M) = \frac{1}{2}\dim_{\RR}M$ in $TM_\C$. Thus, $\Jcal_J$ is of type $n = \dim_\CC M$ and hence of cotype $0$. One can show that any GC structure $\Jcal$ of this type is a $B$-field transformation of a complex structure, i.e. of the form 
        \[
            \Jcal = \btrans{B }\Jcal_J\btrans{-B }
        \]
        with corresponding isotropic subbundle 
        \[
            \btrans{B }L_{\Jcal_J}.
        \]
        \item Let $\omega\in\Omega^2(M)$ be a symplectic form, i.e. $d\omega =0$ and $\omega$ is non-degenerate. Then 
        \[
            \Jcal_\omega = \begin{pmatrix}
                0 & -\omega^{-1}\\ \omega  & 0 
            \end{pmatrix}
        \]     
        defines a GC structure with associated subbundle 
        \[
            L_{\Jcal_\omega} =\mathrm{graph}(-\ii\omega ) = \{X-\ii\omega (X)~:~ X\in TM_\CC\}\subset \TT M_\CC.
        \]
        We have 
        \[
        \pi(L_{\Jcal_\omega}) = TM_{\CC},
        \]
        thus $\Jcal_\omega$ is of type $0$ (and thus of cotype $n= \frac{1}{2}\dim_{\RR}M$). Any GC structure $\Jcal$ of type $0$  is a $B$-field transform of a symplectic structure $\Jcal_\omega$, i.e. is of the form 
        \[
            \Jcal = \btrans{B  }\Jcal_\omega\btrans{-B  }
        \]
        with corresponding maximal isotropic subbundle 
        \[
            L_{\Jcal} = \btrans{B } L_{\Jcal_\omega} = \mathrm{graph}(-\ii(\omega+\ii B)).
        \]
        Due to the last equality, we will sometimes use the notation $\Jcal_{\omega + \ii B} =  \btrans{B  }\Jcal_\omega\btrans{-B  }$.  \hfill \qedhere
    \end{enumerate}
    \end{example}

\subsection{The space of generalized complex structures on a vector space}
In this paper we are mainly concerned with translation-invariant GC structures on tori. So from now on our discussion will focus on the set of constant GC structures on a real vector space. 

Let $n\in \mathbb N$ and let  $V$ be a real $2n$-dimensional vector space.   Consider the set $\C_V$ of complex structures on $V$, i.e. 
\[
    \C_V = \{J\in\End(V)~:~ J^2 = -1\}.
\]
The group $\mathrm{GL}(V)\cong \mathrm{GL}(2n,\rr)$ acts transitively on $\C_V$ by conjugation. If we fix an element $J\in \C_V$, then the stabilizer of $J$ is given by $\mathrm{GL}(V,J) \cong \mathrm{GL}(n,\CC)$. 

Thus, we may view $\C_V$ as the homogeneous space
\[\C_V = \mathrm{GL}(V)/\mathrm{GL}(V,J)\cong \mathrm{GL}(2n,\RR)/\mathrm{GL}(n,\CC).\]
On the direct sum
\[\VV = V\oplus V^*\] 
the quadratic form 
\[\langle X+\xi,X+\xi\rangle = \iota_X\xi = \xi(X)\]
induces on $\VV$ an inner product of signature $(2n,2n)$. A GC structure on $V$ is an endomorphism $\Jcal \in\End(\VV)$ which is skew-symmetric with respect to $\langle\cdot,\cdot\rangle$ and satisfies $\Jcal^2 = -\id_\VV$. We consider the set of GC structures
\[
\mathcal{GC}_V = \{\Jcal\in\mathfrak{so}(\VV)~:~ \Jcal^2 = -\id_\VV\}.
\]
The orthogonal group $\mathrm{O}(\VV)$ acts transitively on $\mathcal{GC}_V$ by conjugation and contains $\mathrm{GL}(V)$ as a subgroup via the embedding
\[\mathrm{GL}(V)\to\mathrm{O}(\VV),\quad A\mapsto \begin{pmatrix}
    A & 0 \\ 0 & (A^{-1})^*
\end{pmatrix} \]
which restricts to an embedding
\[   \C_V\to\mathcal{GC}_V,\quad J\to \Jcal_J =\begin{pmatrix}
    J & 0 \\ 0 & - J^* \end{pmatrix}.
\]
Given $J\in\C_V$, we may take $\Jcal_J\in\mathcal{GC}_V$ as a base point, the stabilizer of which is 
\[
    \mathrm{U}(\VV,\Jcal_J) = \mathrm{O}(\VV)\cap \mathrm{GL}(\VV,\mathcal{J}_J) \cong\mathrm{U}(n,n).
\]
The set of all GC structures is thus given by the homogeneous space
\[
    \mathcal{GC}_V = \mathrm{O}(\VV)/\mathrm{U}(\VV,\mathcal{J}_J) \cong \mathrm{O}(2n,2n)/\mathrm{U}(n,n).
\]
and the map $\C_V\to\mathcal{GC}_V$ is $\mathrm{GL}(V)$-equivariant. 

Note that 
\[\mathfrak{so}(\VV) \cong  \Wedge^2(V\oplus V^*) = \Wedge^2V\oplus\Wedge^2V^* \oplus V\otimes V^*= \left\{\begin{pmatrix}
    A & \beta \\ B  & -A^*
\end{pmatrix}~:~ A\in\End(V), B\in \Wedge^2V^*, \beta\in\Wedge^2V\right\}.\]

\begin{remark}\label{remark:AdBbeta}
We now explicitly describe certain kinds of symmetry transformations that act on $\mathcal{GC}_V$. Let 
\[
    \mathcal{J} = \begin{pmatrix} A & \pi \\ \omega  & -A^* \end{pmatrix}\in \mathcal{GC}_V,
\]
where $A\in\End(V), \omega  \in \Wedge^2V^*, \pi \in\Wedge^2V$. 
\begin{enumerate}
    \item We have already encountered $B$-field transformations. To a $2$-form $B\in\Wedge^2V^*$ we associate the element 
    \[
    \btrans{B } = \begin{pmatrix}
        1 & 0 \\ B & 1
    \end{pmatrix}\in\mathrm{O}(\VV),
    \]
    which acts on $\Jcal$ by 
    \[
        \btrans{B }\Jcal \btrans{-B }= \begin{pmatrix}
        A-\pi  B  & \pi  \\ B  A+A^*B +\omega -B \pi  B  & B \pi -A^*
        \end{pmatrix}.
    \]
    \item We shall also encounter $\beta$-transformations. To a bivector $\beta\in\Wedge^2V$ we associate the element 
    \[
    \betatrans{\beta } = \begin{pmatrix}
        1 & \beta \\ 0 & 1
    \end{pmatrix}\in\mathrm{O}(\VV),
    \]
    which acts on  $\Jcal$ by 
    \[
       \betatrans{\beta }\Jcal\betatrans{-\beta }= \begin{pmatrix}
        A+\beta  \omega   & \pi -A\beta  -\beta  A^*-\beta  \omega \beta   \\ \omega & -A^*-\omega \beta  \end{pmatrix}. \hfill\qedhere
    \]
\end{enumerate}
\end{remark}

A central classification result is the following Darboux theorem for GC structures.
\begin{theorem}[Generalized Darboux theorem, Section 4.7 \cite{Gualtieri:2003dx}]\label{theorem:generalizedDarboux}
    Any   GC structure $\Jcal$ on $\VV = V \oplus V^*$ is a 
    $B$-field transformation of a direct sum $\Jcal_1 \oplus \Jcal_2$, 
    where $\Jcal_1 = \Jcal_J \in \mathrm{O}(V_1 \oplus V_1^*)$ 
    and $\Jcal_2 = \Jcal_\omega \in \mathrm{O}(V_2 \oplus V_2^*)$ are 
    associated with a complex structure $J$ on  $V_1$ and a symplectic structure $\omega$ on $V_2$ for some splitting $V = V_1 \oplus V_2$. The type of $\Jcal$ is $\dim_\CC (V_1,J)$. 
\end{theorem}
We now investigate the problem of transforming a complex structure $\Jcal_J$ into an arbitrary GC structure $\Jcal$ using a product of a $B$-field transformation and a $\beta$-transformation (or the other way round). The rest of this section will be dedicated to proving the following parametrization result for the space of linear GC structures on $V$:

\begin{proposition}\label{proposition:B_then_beta_eventype}
    Let $\Jcal\in\mathcal{GC}_V$ be a GC structure of even cotype. 
    \begin{enumerate}
        \item There exist $B\in\Wedge^2V^*, \beta\in\Wedge^2 V$ and $J\in\mathcal C_V$, such that 
        \[
        \Jcal = \betatrans{\beta }\btrans{B }\Jcal_J\btrans{-B }\betatrans{-\beta }.
        \]
        \item There exist $B\in\Wedge^2V^*, \beta\in\Wedge^2 V$ and $J\in\mathcal C_V$, such that 
        \[
        \Jcal = \btrans{B }\betatrans{\beta }\Jcal_J\betatrans{-\beta }\btrans{-B }.
        \]
    \end{enumerate} 
\end{proposition}

 The Darboux theorem will allow us to reduce to the case when $\Jcal = \Jcal_\omega$ for a symplectic form $\omega\in\Wedge^2V^*$. Concerning the question of uniqueness, the explicit formulas for $B$-field transformations and $\beta$-transformations given in Remark \ref{remark:AdBbeta} immediately imply: 
\begin{proposition} 
    Suppose $J,J'\in \C_V$ are complex structures and $\omega,\omega'\in\Wedge^2V^*$ are symplectic forms. Then for any $B\in \Wedge^2V^*,\beta\in\Wedge^2V$ we have
    \begin{enumerate}
        \item $\btrans{B }\Jcal_\omega\btrans{-B } = \Jcal_{\omega'}$ if and only if $B=0$ and $\omega=\omega'$.
        \item $\btrans{B }\Jcal_J\btrans{-B } = \Jcal_{J'}$ if and only if  $J=J'$ and $B\in\Wedge^{1,1}_JV^*$.
        \item $\betatrans{\beta }\Jcal_J \betatrans{-\beta }= \Jcal_{J'}$ if and only if $J=J'$ and $\beta\in\Wedge^{1,1}_JV$.
        \item $\betatrans{\beta }\Jcal_\omega\betatrans{-\beta } = \Jcal_{\omega'}$ if and only if $\beta=0$ and $\omega=\omega'$.
    \end{enumerate}
    In particular: If the equation 
    \[
        \betatrans{\beta }\btrans{B }\Jcal_\omega\btrans{-B }\betatrans{-\beta }  = \Jcal_J
    \]
    admits a solution $(B,\beta)$, then $B$ is unique and $\beta$ is unique up to addition of a bivector of type $(1,1)$ and if the equation 
    \[
        \betatrans{\beta }\btrans{B }\Jcal_J \btrans{-B }\betatrans{-\beta } = \Jcal_\omega
    \]
    admits a solution $(B,\beta)$, then $\beta$ is unique and $B$ is unique up to addition of a form of type $(1,1)$.
\end{proposition}

\begin{proposition}\label{proposition:betaBJomega}
Let $\omega\in\Wedge^2V^*$ be a symplectic form and let $J\in \C_V$ be a complex structure on $V$.
\begin{enumerate}
\item The equation
\[\betatrans{\beta }\btrans{B }\Jcal_\omega\btrans{-B }\betatrans{-\beta }= \Jcal_J\]
admits a solution $(B,\beta)\in \Wedge^2V^*\times\Wedge^2V$ if and only if $\omega\in (\Wedge^{2,0}_JV^*+\Wedge^{0,2}_JV^*)\cap\Wedge^2V^*$. In this case $B$ lies in the subspace $(\Wedge^{2,0}_JV^*+\Wedge^{0,2}_JV^*)\cap\Wedge^2V^*$ and is uniquely determined by 
\[B  = \omega  J.\] 
The general form of $\beta$ is 
\[\beta   = -\frac{1}{2} B^{-1} + \tau\] 
for an arbitrary $\tau\in\Wedge^{1,1}_JV$.
\item The equation
\[\betatrans{\beta }\btrans{B }\Jcal_J\btrans{-B }\betatrans{-\beta }= \Jcal_\omega\]
admits a solution $(B,\beta)\in \Wedge^2V^*\times\Wedge^2V$ if and only if $\omega\in (\Wedge^{2,0}_JV^*+\Wedge^{0,2}_JV^*)\cap\Wedge^2V^*$. In this case $\beta$ lies in the subspace $(\Wedge^{2,0}_JV+\Wedge^{0,2}_JV)\cap\Wedge^2V$ and is uniquely determined by 
\[\beta  = -J \omega^{-1} = (\omega  J)^{-1}.\] 
The general form of $B\in\Wedge^2V^*$ is 
\[B   = -\frac{1}{2} \beta^{-1}+ B \] 
for an arbitrary $b\in\Wedge^{1,1}_JV^*$.
\end{enumerate}
\end{proposition}
\begin{proof}
\begin{enumerate}
\item Remark \ref{remark:AdBbeta} implies in particular for a symplectic form $\omega\in \Wedge^2V^*$: 
\begin{align*}
    & \betatrans{\beta }\btrans{B }\Jcal_\omega \btrans{-B }\betatrans{-\beta } = \betatrans{\beta }\begin{pmatrix}
    \omega^{-1} B  & -\omega^{-1} \\ \omega +B \omega^{-1} B  & -B \omega^{-1}\end{pmatrix}\betatrans{-\beta }\\
    &=  \begin{pmatrix}
    \omega^{-1} B  +\beta (\omega +B  \omega^{-1} B  ) & -\omega^{-1}-\omega^{-1} B  \beta  -\beta  B  \omega^{-1}-\beta  (\omega +B  \omega^{-1} B  )\beta   \\ \omega +B  \omega^{-1} B  & -B  \omega^{-1}-(\omega +B  \omega^{-1} B  )\beta  \end{pmatrix}
\end{align*}
It follows that the equation
\[ \betatrans{\beta }\btrans{B }\Jcal_\omega \btrans{-B }\betatrans{-\beta }  = \Jcal_{J}\quad \text{for some } J\in \C_V\]
is equivalent to the following conditions: 
\begin{align}
    0&= \omega +B \omega^{-1}B  \label{eq:1}\\
    0&= \omega^{-1}+\omega^{-1} B\beta +\beta B\omega^{-1}+\beta (\omega+B\omega^{-1} B)\beta \label{eq:2}. 
\end{align}
If these are satisfied, then 
\[\betatrans{\beta}\btrans{B}\Jcal_\omega \btrans{-B}\betatrans{-\beta} = \begin{pmatrix}
    \omega^{-1} B & 0 \\ 0 & -B\omega^{-1} \end{pmatrix}.\]
In particular
\[J := \omega^{-1} B,\] 
is a complex structure on $V$ and satisfies
\[J^2 = -1,\quad J^* = B\omega^{-1}\]
and
\[B = \omega J.\]
Note that for this to be a two-form, i.e. $B^* = - B$, we must require 
\[
-\omega  J = (\omega  J)^* = J^*\omega^* = -J^*\omega ,
\]
i.e. 
\[
J^*\omega  J = -\omega , 
\]
which implies that also 
\[J^*B  J=-B .\] 
In particular, 
\[B,\omega \in (\Wedge^{2,0}_JV^*+\Wedge^{0,2}_JV^*)\cap \Wedge^2V^*.\]
Note that if \eqref{eq:1} holds, then \eqref{eq:2} simplifies to 
\[0 = \omega^{-1}+\omega^{-1}B \beta +\beta  B \omega^{-1}.\]
With the above notation, \eqref{eq:2} becomes
\[\omega^{-1} = -J\beta  -\beta  J^*,\]
which we may rewrite in the form 
\[J B^{-1} = -J\beta  -\beta  J^*,\]
or, after multiplying both sides by $-J = J^{-1}$ from the left,
\[ B^{-1} = -\beta +J\beta  J^*.\]
In other words, we see that $B\in (\Wedge^{2,0}_JV^*+\Wedge^{0,2}_JV^*)\cap \Wedge^2V^*$ is non-degenerate and $- B^{-1}$ is given by (two times) the component of $\beta$ in $(\Wedge^{2,0}_JV+\Wedge^{0,2}_JV)\cap \Wedge^2V$ , which is given by
\[\beta^{(2,0)}+\beta^{(0,2)} = \frac{1}{2}(\beta -J\beta  J^*).\]
\item Remark \ref{remark:AdBbeta} implies in particular
\begin{align*}
    &\betatrans{\beta }\btrans{B }\Jcal_J\btrans{-B }\betatrans{-\beta }= \betatrans{\beta }\begin{pmatrix}
    J & 0 \\ B  J+J^*B  & -J^*\end{pmatrix}\betatrans{-\beta }\\
    &=  \begin{pmatrix}
    J + \beta(B  J+J^*B  ) & -J\beta-\beta J^*-\beta(B  J+J^*B  )\beta \\ B  J+J^*B   & -J^*-(B  J+J^*B  )\beta\end{pmatrix}
\end{align*}
It follows that 
\[\betatrans{\beta }\btrans{B }  \Jcal_{J}\btrans{-B  }\betatrans{-\beta } = \Jcal_\omega\]
for some nondegenerate $\omega\in\Wedge^2V^*$ is equivalent to the equation
\begin{align}
    0&= J + \beta  (B  J+J^*B  ) \label{eq:3}
\end{align}
If this equation is satisfied, then 
\[\betatrans{\beta }\btrans{B }  \Jcal_{J}\btrans{-B  }\betatrans{-\beta } = \begin{pmatrix}
   0& -J\beta -\beta  J^*-\beta (B  J+J^*B  )\beta  \\ B  J+J^*B   & 0 \end{pmatrix}.\]
Equating this with $\Jcal_\omega$ and looking at the lower left entry, we find
\begin{equation}\label{eq:5}
\omega  = B  J+J^*B   \in(\Wedge^{2,0}_JV^*+\Wedge^{0,2}_JV^*)\cap \Wedge^2V^*.
\end{equation}
It follows that 
\[-B  +J^*B  J = \omega  J.\]
The left hand side of this equation is (minus two times) the projection of $B$ onto the subspace $(\Wedge^{2,0}_JV^*+\Wedge^{0,2}_JV^*)\cap \Wedge^2V^*$. And we see that we are free to add to $B$ any two-form $b\in\Wedge^{1,1}_JV^*$. 

Inserting \eqref{eq:5}  into \eqref{eq:3}, we get
\[0 = J+\beta \omega ,\]
i.e.
\[\beta  = -J\omega^{-1}. \hfill \qedhere\]
\end{enumerate}
\end{proof}

\begin{proposition}\label{proposition:Bbetaomeganeven}
    Let $V$ be a real vector space of dimension $2n$.
    \begin{enumerate}
        \item Let $\omega\in\Wedge^2V^*$ be a symplectic form. Then there exists a complex structure $J\in \C_V$, a $2$-form $B\in\Wedge^2V^*$ and a bivector $\beta\in\Wedge^2V$ such that 
        
        \[\betatrans{\beta }\btrans{B }  \Jcal_{\omega}\btrans{-B  }\betatrans{-\beta }= \Jcal_J\]
        if and only if $n$ is even, i.e. $\dim V$ is divisible by $4$.
        \item Let $J\in\mathcal C_V$ be a complex structure. Then there exists a symplectic form $\omega\in\Wedge^2V^*$, a $2$-form $B\in\Wedge^2V^*$ and a bivector $\beta\in\Wedge^2V$ such that 
    \[\betatrans{\beta }\btrans{B }  \Jcal_{J}\btrans{-B  }\betatrans{-\beta }= \Jcal_\omega\]
    if and only if $n$ is even, i.e. $\dim V$ is divisible by $4$.
    \end{enumerate}
\end{proposition}
\begin{proof}
    The discussion above shows that if $J\in\mathcal C_V$ and a solution $(B,\beta)$ exist, then $B   = \omega  J$ and this is a two-form iff $\omega\in (\Wedge^{2,0}_JV^*+\Wedge^{0,2}_JV^*)\cap V$. Extend $B,\omega$ complex bilinearly and consider 
    \[\Omega = B+\ii\omega \in \Wedge^2V_\cc^*.\]
    Let $a,b\in V$ such that $v = a+\ii b\in V_{\cc}$ lies in the kernel of $\Omega$, i.e. $\omega (v) = 0$. Then
    \[0 = \omega (v) = B (a+\ii b)+\ii\omega (a+\ii b) = B (a)-\omega (b) + \ii(B  (b)+\omega (a)).\]
    This means that 
    \[\omega (b) = B (a) = \omega (Ja)\]
    and
    \[\omega (a) = -B (b) = -\omega (Jb).\]
    By the nondegeneracy of $\omega$ this is equivalent to 
    \[b = Ja.\]
    Thus, we see that $v = a+\ii b\in \ker\Omega$  is equivalent to $v\in V^{0,1} = \mathrm{Eig}(J,-\ii)\subset V_\cc$. It follows that $\Omega\in\Wedge^{2,0}V^*$ is a \emph{complex symplectic form}, and thus the complex dimension of $V$ must be even. 

    Conversely, if the  dimension of $V$ is divisible by $4$ and $\omega\in\Wedge^2V^*$ is nondegenerate,  we may pick a complex structure $J$ and a complex symplectic form $\Omega\in\Wedge^{2,0}V^*$, and decompose it into real and imaginary parts $\Omega = \tB+\ii\tilde\omega$. Then the calculation in the first part of the proof shows that $\tilde \omega  J = \tB $. Since any two symplectic forms on $V$ are in the same $\mathrm{GL}(V)$-orbit, we may use the $\mathrm{GL}(V)$ action to achieve $\tilde\omega = \omega$. 
\end{proof}

\begin{proof}[Proof of Proposition \ref{proposition:B_then_beta_eventype}]
\begin{enumerate}
    \item By the generalized Darboux theorem, we can find $\tilde{B}\in\Wedge^2V^*$ and a splitting $ V= V_1\oplus V_2$, a complex structure $J_1\in\mathcal C_{V_1}$ and a symplectic form $\omega_2\in\Wedge^2 V_2^*$ such that 
        \[
            \btrans{\tB }\Jcal\btrans{-\tB } = \Jcal_{J_1}\oplus \Jcal_{\omega_2}.
        \]
    Since the cotype of $\Jcal$ is even, it follows that the complex dimension of $V_2$ is even and therefore
    Proposition \ref{proposition:Bbetaomeganeven}  allows us to find $B_2\in\Wedge^2 V_2^*, \beta_2\in \Wedge^2 V_2$ and $J_2\in\mathcal C_{V_2}$ such that  
    \[
        \Jcal_{\omega_2}= \btrans{B _2}\betatrans{\beta_2 }\Jcal_{J_2}\betatrans{-\beta_2 }\btrans{-B _2}\in\mathcal{GC}_{V_2}. 
    \]
    Using the splitting $V = V_1\oplus V_2$ and extending by zero on $V_2$, respectively $V_2^*$, we may view $B _2,\beta_2 $ as elements of $\mathfrak{so}(\VV)$. With this interpretation we can write 
    \[
    \Jcal = \btrans{-\tB  +B _2}\betatrans{\beta_2 }\Jcal_{J_1\oplus J_2}\betatrans{-\beta_2 }\btrans{\tB -B_2 }.
    \] 
    \item Exchanging the role of $V$ and $V^*$ in the generalized Darboux theorem, we see that there also exists $\tbeta \in\Wedge^2V$ such that 
     \[
            \betatrans{\tbeta }\Jcal\betatrans{-\tbeta } = \Jcal_{J_1}\oplus \Jcal_{\omega_2}.
    \]
    Note that although $\beta$-transforms do not preserve the cotype in general, they still preserve the \emph{parity} of the cotype (see \cite{Gualtieri:2003dx}).\hfill\qedhere
\end{enumerate}
\end{proof}

We will use Hyperkähler manifolds as running examples throughout the paper:
\begin{example}\label{example:Hyperkähler}
 Suppose $\dim V = 4n$ for some $n\in\mathbb N$. Consider a (pseudo-)Hyperk\"ahler structure $(g,I,J,K)$ on $V$, i.e. $g$ is a (possibly indefinite) inner product on $V$ and $I,J,K\in\C_V$ are skew-symmetric with respect to $g$  and satisfy the quaternionic relations. Denote by 
\[
    \omega _I = g  I,\quad \omega_J  = g  J,\quad \omega_K  = g  K
\]
the associated symplectic forms. The quaternionic relations imply $\omega_J,\omega_K \in \Wedge^{2,0}_IV^*+\Wedge^{0,2}_IV^*$ and moreover 
\[
    \omega_K  I= g  KI = g  J = \omega_J .
\]
It then follows from Proposition   \ref{proposition:betaBJomega} that
        \begin{equation}
            \Jcal_{\omega_J} = \betatrans{-\omega_K^{-1}}\btrans{\frac{\omega _K}{2}}\Jcal_I\btrans{\frac{-\omega _K}{2}}\betatrans{\omega_K^{-1}}.\hfill\qedhere
        \end{equation}
\end{example}

\section{T-duality}\label{sec:Tdual}
T-duality will be indispensable in our description of the symmetries of the generalized Dolbeault algebra $\Abf(\Jcal)$. In this section we will recall the formulation of T-duality as a symmetry of principal torus bundles by \cite{BEM04} and how it induces a symmetry of Courant algebroids, as was described in \cite{CG11}. 
We further specialize to T-duality for tori, and provide explicit equations describing when two invariant GC structures on dual tori are T-dual in Lemma \ref{lemma:Tdualflip}, which will re-appear in later sections when describing the symmetries of $\Abf(\Jcal)$.
We also review the notion of integrality for T-duality.

\begin{definition}
Let $M$ and $\tilde{M}$ be principal $T^k$-bundles with base $B$. Consider the fiber-product $M\times_B \tilde{M}$, and the following diagram: 
\begin{center}
\begin{tikzcd}
& M\times_B \tilde{M} \ar[ld,"p"],\arrow{rd}[swap]{\tilde{p}} &\\
M \ar[rd,"\pi"] & & \tilde{M} \arrow{ld}[swap]{\tilde{\pi}}\\
& B & 
\end{tikzcd}
\end{center}
We say that $M$ and $\tilde{M}$ are \emph{T-dual} if there exists a closed $T^{2k}$-invariant two-form $F \in \Omega^2_{T^{2k}}(M\times_B\tilde{M})$ such that
the restriction
\begin{equation*}
F \colon \mathfrak{t}_M \times \mathfrak{t}_{\tilde{M}} \rightarrow \rr,
\end{equation*}
where $\mathfrak{t}_M,\mathfrak{t}_{\tilde{M}}$ denotes the tangent space to the fibres of $p$, $\tilde{p}$ respectively, is non-degenerate. We call $M \times_B \tilde{M}$ the \emph{correspondence space}. 
\end{definition}
We further say that $F$ is \emph{unimodular} if $F((\partial_{\theta_i},\partial_{\theta_j^{\vee}})) \in {\rm GL}(k;\mathbb{Z})$ where $(\partial_{\theta_i})_i$ and $(\partial_{\theta_j^{\vee}})_j$ are bases of invariant period 1 vector fields of $\mathfrak{t}_M$ respectively $\mathfrak{t}_{\tilde{M}}$. If further $F$ defines an integral cohomology class we call the T-duality \emph{integral}.

Let $(TM \oplus T^*M)/T^k$ denote the Courant algebroid over $B$ whose sections consists of the invariant sections $\Gamma_{T^k}(\mathbb{T}M)$ of $\mathbb{T}M$.

\begin{theorem}[\cite{CG11}]\label{theorem:courantiso}
If $M$ and $\tilde{M}$ are T-dual, then there is an isomorphism of Courant algebroids
\begin{equation*}
\varphi \colon (TM \oplus T^*M)/T^k \rightarrow (T\tilde{M} \oplus T^*\tilde{M})/T^k.
\end{equation*}
Moreover, if $\Jcal$ is a $T^k$-invariant GC structure on $M$ it gets transformed to a $T^k$-invariant GC structure on $\tilde{M}$ denoted $\tilde{\Jcal}$.
\end{theorem}

\begin{remark}[Type change]\label{remark:typechange}
Given a GC structure $\Jcal$ on $M$ there are several ways in which one my predict the type of the T-dual structure $\tilde{\Jcal}$. This is carefully described in Example 4.3 in \cite{CG11}. For us it's important to remark that when $k=1$, one has the relation that
\begin{equation*}
    \type(\tilde{\Jcal}) \pm 1
\end{equation*}
implying that the parity will always change when performing T-duality along the circle.
\end{remark}

\subsection{T-duality over a point}
\label{sec:Tdualitypoint}
Of particular importance in this paper is the case where $M = V/\Gamma$ is an $n$-dimensional torus, and T-duality is performed with the induced $T^n$-action. In this case, necessarily $\tilde{M} = V^*/\Gamma^{*}$, where $\Gamma^{*} = \set{f \in V^* ~:~ f(\Gamma) \subset \mathbb{Z}}$. The choice of $F\in \Omega^2(M \times \tilde{M})$ then corresponds to the choice of isomorphism $TM \simeq T^*\tilde{M}$, which we will take to be induced by the canonical isomorphism $\phi \colon V \simeq V^{**}$. If we then suppress this $\phi$ from the notation, we conclude that the isomorphism of Theorem \ref{theorem:courantiso} is given by 
\begin{equation}
\varphi =  \begin{pmatrix}
     0 & \id_{V^*}\\
     \id_V & 0
 \end{pmatrix}: 
 V \oplus V^* \rightarrow V^* \oplus V 
\end{equation}
Consequently, given a $T^n$-invariant GC structure $\Jcal$ on $M$, identified with $\Jcal \colon V\oplus V^* \rightarrow V \oplus V^*$, we obtain the following formula for the T-dual GC structure:
\begin{equation}\label{equation:Tdual}
\tilde{\Jcal} = 
\begin{pmatrix} 0 & \id_{V^*} \\ \id_V & 0 \end{pmatrix}
\Jcal
\begin{pmatrix} 0 & \id_{V^*} \\ \id_V & 0 \end{pmatrix}.
\end{equation}

Our next course of action is to describe the T-dual of a $(\beta,B)$-transformation of a complex structure: Given a complex structure $I \colon V \rightarrow V$, the T-dual of $\mathcal{J}_I$ is given by $\mathcal{J}_{-I^*}$. Given $\beta \in \Wedge^2V,B\in \Wedge^2V^*$, and the GC structure
\begin{equation*}
L_{\Jcal} =  \betatrans{\beta }\btrans{B }L_{\Jcal_I}    
\end{equation*}
a direct computation shows that T-dual structure is given by
\begin{equation}\label{eq:Tdual1}
    L_{\tilde{\Jcal}}= \btrans{\beta}\betatrans{B}L_{\Jcal_{-I^*}},
\end{equation}
where $B$ in this formula is viewed as a bivector on $V^*$ and $\beta$ as a 2-form on $V^*$.

However, it will be essential to again write this as a $(\tbeta,\tilde{B})$-transformation for, for $\tbeta \in \Wedge^2V^*$ and $\tilde{B} \in \Wedge^2V$. The equations relating $(\beta,B)$ and $(\tbeta,\tB)$ will be given in the following lemma, and will be of fundamental importance in this paper:

\begin{lemma}
\label{lemma:Tdualflip}
    Let 
    \[
    L_{\Jcal} = \betatrans{\beta }\btrans{B } L_{\Jcal_I}
    \] 
    be a GC structure on $V/\Gamma$ and 
    \[
    L_{\Jcal'} = \betatrans{\tbeta }\btrans{\tB }L_{\Jcal_{\tilde I}}
    \]
    a GC structure on $V^* / \Gamma^*$.
    Then $\tilde{\Jcal} = \Jcal'$ if and only if \begin{align*}
        &\tilde I(1-\tbeta  \beta ) = -(1-\tbeta  \beta ) I^* \\
        &(B  I + I^* B ) (1 - \beta  \tbeta ) = I^* \tbeta  - \tbeta  \tilde I^* \\
        & (\tB  \tilde I + \tilde I^* \tB ) (1-\tbeta  \beta ) = \tilde I^* \beta  - \beta  I^*
    \end{align*}
    In particular, if $\alpha^{-1} := (1-\tbeta  \beta )$ is invertible, then $\tilde{\Jcal} = \Jcal'$ if and only if
    \begin{align}
        &\tilde I = -\alpha^{-1} I^* \alpha \nonumber \\
        & B^{(0,2)_I} = (\tbeta  \alpha^*)^{(0,2)_I} \label{equation:Poissongauge1} \\
        & \tB ^{(0,2)_{\tilde I}} = ( \beta  \alpha)^{(0,2)_{\tilde I}}.\label{equation:Poissongauge2}
    \end{align}
    
    \begin{proof}
        We compute $\Jcal$:
        \begin{align*}
            \Jcal &=\betatrans{\beta }\btrans{B } \begin{pmatrix}
                I & 0 \\ 0 &-I^* 
            \end{pmatrix} \btrans{-B }\betatrans{-\beta }
            \\&= \begin{pmatrix}
    I + \beta (B  I+I^*B  ) & -I\beta -\beta  I^*-\beta (B  I+I^*B  )\beta  \\ B  I+I^*B   & -I^*-(B  I+I^*B  )\beta \end{pmatrix}
    \end{align*}
    Hence
    \begin{align*}
      \tilde{\Jcal} = \begin{pmatrix}
        -I^*-(B  I+I^*B  )\beta  & B  I+I^*B  \\
        -I\beta -\beta  I^*-\beta (B  I+I^*B  )\beta  & I + \beta (B  I+I^*B  )
    \end{pmatrix}
    \end{align*}
    needs to be equal to 
    \begin{align*}
        \Jcal' = 
        \begin{pmatrix}
            \tilde I + \tbeta (\tB  \tilde I+\tilde I^* \tB  ) & -\tilde I \tbeta -\tbeta  \tilde I^*-\tbeta (\tB  \tilde I+\tilde I^*\tB  )\tbeta  \\ \tB  \tilde I+\tilde I^* \tB   & -\tilde I^*-(\tB  \tilde I+\tilde I^* \tB  )\tbeta \end{pmatrix}.
    \end{align*}
    Let us assume that $\Jcal' = \tilde{\Jcal}$. We define 
    \[
    B'  := B  I + I^* B 
    \]
    and 
    \[
    \tB'  := \tB  \tilde I + \tilde I^* \tB .
    \]
    By comparing top-left and top-right entries, we obtain
    \begin{align*}
        B'  &= -(\tilde I \tbeta +\tbeta  \tilde I^* + \tbeta  \tB'  \tbeta ) \\ 
        &= -(\tilde I + \tbeta  \tB' ) \tbeta  - \tbeta  \tilde I^* \\
        &= (I^* + B'  \beta ) \tbeta - \tbeta  \tilde I^*, 
    \end{align*}
    which implies 
    \begin{equation}\label{equation:t1}
        B'  (1 - \beta  \tbeta ) = I^* \tbeta  - \tbeta  \tilde I^*.
    \end{equation}
    Similarly, we find 
    \begin{align}
        \label{equation:t2}
        \tB'  (1-\tbeta  \beta ) = \tilde I^* \beta  - \beta  I^*.
    \end{align}
    Now compose top-left entries from the right with $(1-\tbeta  \beta )$ to get 
    \[
        (\tilde I + \tbeta  \tB' ) (1-\tbeta  \beta ) = -(I^* +  B' \beta ) (1-\tbeta  \beta )\]
     and use the identities \eqref{equation:t1} and \eqref{equation:t2} to obtain
    \[   \tilde I(1-\tbeta  \beta ) = -(1-\tbeta  \beta ) I^*.\]
    The remaining statements follow similarly.
    \end{proof}

\end{lemma}

\begin{remark}
In the case where $(1-\tbeta \beta )$ is invertible, there is a Poisson geometric interpretation of the equations \eqref{equation:Poissongauge1} and \eqref{equation:Poissongauge2}.

Given two Poisson structures $\pi,\pi'$ on $M$, they are said to be \emph{gauge equivalent} if there exists a closed two-form $B$ such that $(\mathrm{id}+\pi B )$ is invertible and satisfies the equation:
\begin{equation*}
\pi'  = \pi (1+\pi B )^{-1}.    
\end{equation*}
Given a Poisson structure $\pi$ it's graph $L_{\pi} = {\rm graph }(\pi) \subset TM \oplus T^*M$ defines a Dirac structure. The term gauge equivalence comes from the fact the associated Dirac structures are related via $L_{\pi'} = \btrans{B } L_{\pi}$.

If we view $B$ as a bivector on $V^*$ and $\beta$ as a two-form on $V^*$, then the equation \eqref{equation:Poissongauge1} tells us precisely that $\tbeta$ is the $\beta$-gauge transformation of $B$. Similarly, \eqref{equation:Poissongauge2} tells us that $\beta$ is the $B$-gauge of $\tbeta$.
\end{remark}

The condition that $1 - \tilde{\beta}\beta $ is invertible is not a property of the GC structure, rather of the particular way in which it is presented as a $(\beta,B)$-transformation. In fact, for any invariant GC structure there exists choices of $(\beta,B)$ for which this condition holds, as we will see in Lemma \ref{lemma:pert}.

\subsection{T-duality over a torus}
When given a torus $V/\Gamma$ it is guaranteed that the T-dual over a point is again a torus. When performing T-duals with respect to lower dimensional tori this needn't be the case, and even when it is imposed the T-duality operations might not be compatible with the structures present. We thus single out those T-dualities which are compatible: 
\begin{definition}
Let $V / \Gamma$ and $\tilde V / \tilde \Gamma$ be tori. We say that $V/\Gamma$ and $\tilde V / \tilde \Gamma$ are \emph{affine integral T-dual} if we have a diagram
\[\begin{tikzcd}
                        & \Gamma \times_{\Gamma_B} \tilde \Gamma \arrow[ld] \arrow[rd] &                                      \\
\Gamma \arrow[rd, "p"'] &                                                              & \tilde \Gamma \arrow[ld, "\tilde p"] \\
                        & \Gamma_B                                                     &                                     
\end{tikzcd}\]
together with a 2-form $F \in \Wedge^2 (\Gamma \times_{\Gamma_B} \tilde \Gamma)^*$, such that
\begin{enumerate}
    \item $p$ and $\tilde p$ are $\mathbb Z$-linear and surjective,
    \item $F$ is unimodular, i.e., letting $\Gamma_f := \text{ker}(p)$ and $\tilde \Gamma_f := \text{ker}(\tilde p)$, we have that $F|_{\Gamma_f \times \tilde \Gamma_f} : \Gamma_f \to (\tilde \Gamma_f)^*$ is invertible. \qedhere
\end{enumerate}
\end{definition}

\begin{lemma}\label{lem:Tdualintegral}
    Let $V/\Gamma$ and $\tilde V/ \tilde \Gamma$ be affine integral T-dual. Then the induced equivalence of linear Courant algebroids $\phi : V \oplus V^* \to \tilde V \oplus \tilde V^*$ is integral, i.e., $\phi(\Gamma \times \Gamma^*) = \tilde \Gamma \times \tilde \Gamma^*$. In particular, if we choose an isomorphism $\Gamma \cong \tilde \Gamma$, then the corresponding automorphism of $V \oplus V^*$ is contained in $O(\Gamma, \Gamma^*) \cong O(n,n;\mathbb Z)$.
    \begin{proof}
        By Theorem \ref{theorem:courantiso}, we obtain an isomorphism of Courant algebroids over the base $B/\Gamma_B$.    
        In the case of affine T-duality, the Courant isomorphism is constant along the base $B / \Gamma_B$ (it also follows from the inspection below), hence we can describe it by an isomorphism 
        \[ \phi : V \oplus V^* \to \tilde V \oplus \tilde V^* \]
        preserving the inner products. To show that $\phi$ preserves lattices, we inspect the definition of $\phi$ given in \cite[Section 3]{CG11}. Let $X + \xi \in \Gamma \oplus \Gamma^*$. Choose an auxiliary lift $\hat X_{\mathrm{aux}}$ of $X$ along $\Gamma \times_{\Gamma_B} \tilde \Gamma \to \Gamma$. Any other lift in $V \times_{B} \tilde V$ will differ by an element in $\tilde V_f$. Hence, according to \cite[Section 3]{CG11}, we need to find the unique element $\Delta X \in \tilde V_f$ such that
        \[ (p^* \xi - F(\hat X_{\mathrm{aux}})) |_{\Gamma_f} = F(\Delta X)|_{\Gamma_f}. \]
        By construction and as $F$ is integral, the left hand side is contained in $\Gamma_f^*$. 
        Moreover, by unimodularity, $F|_{\Gamma_f \times \tilde \Gamma_f} : \Gamma_f \to (\tilde \Gamma_f)^*$ is invertible, so $\Delta X \in \tilde \Gamma_f \subset \tilde V_f$. Thus the correct lift $\hat X = \hat X_{\mathrm{aux}} + \Delta X$ lies in $\Gamma \times_{\Gamma_B} \tilde \Gamma$. By construction, $p^* \xi - F(\hat X)$ is basic, and
        \[ \phi(X + \xi) = \tilde p_* \hat X + (p^* \xi - F(\hat X)) \in \tilde \Gamma \times \tilde \Gamma^*. \qedhere \]
    \end{proof}
\end{lemma}

\begin{remark}\label{remark:ts1matrix}
    T-duality relative to a circle as in Remark \ref{remark:typechange} corresponds to a matrix $\begin{pmatrix} 0 & 1 \\ 1 & 0    \end{pmatrix}
    \oplus \id_{2n-2}$ when we unravel definitions.
\end{remark}

\section{Cohesive modules}\label{sec:cohmod}
In this section we will recall the definition of curved differential graded algebra, and the cohesive modules over them. We further recall how the category of cohesive modules over the Dolbeault algebra of a complex manifold gives rise to the derived category of coherent sheaves in Theorem \ref{th:block}.
\begin{definition}\label{definition:cdga}
    A \emph{curved dg-algebra} or \emph{cdg-algebra} $A = (A^\bullet,\cdot,d,B)$ consists of 
    \begin{enumerate}
        \item a $\mathbb Z$-graded algebra $A^\# = (A^{\bullet}, \cdot)$,
        \item a map, called the \emph{differential}, $d \colon A^{\bullet} \to A^{\bullet+1}$ satisfying \[d(ab) = d(a)b + (-1)^{|a|}a d(b),\]
        \item and an element $B \in A^2$, called the \emph{curvature element}, satisfying
        \begin{align*}
            &d(B) = 0,\\
            &d^2(a) = [B,a] = Ba - aB. 
        \end{align*}
    \end{enumerate}
    A \emph{(left) cdg-module} $P = (P^\bullet, ., d)$ over $A$ is a graded left module $P^\# = (P^{\bullet}, .)$ over the graded algebra $A^\#$ together with a map $d_P$ of degree $1$ satisfying the graded Leibniz rule
    \[ d_P(a.p) = d(a).p + (-1)^{|a|} a.d_P (p)\]
    and
    \[ d_P^2(p) = B.p.\]
    Given two cdg-modules $P = (P^\bullet,.,d_P)$ and $Q = (Q^\bullet,.,d_Q)$ we define 
    \[ \Hom_{A}^\bullet(P,Q) \]
    in degree $n$ to consist of maps $f \colon P^\bullet \to Q^{\bullet+n}$ commuting with the $A^\#$-action and endowed with the differential
    \[ d(f) = d_Q \circ f - (-1)^{|f|} f \circ d_P\]
    where $|f| = n$; one checks that indeed $d^2 = 0$ even when $B \neq 0$. With these definitions, we obtain the dg-category $A\mathrm{-Mod}$ of \emph{left cdg-modules over the cdg-algebra $A$}.
\end{definition}
\begin{definition}
    Let $A$ be a cdg-algebra. A \emph{cohesive module} over $A$ is a left cdg-module $P$ over $A$ such that its underlying graded module $P^\#$ over $A^\#$ is finitely generated and projective. The full dg subcategory of $A\mathrm{-Mod}$ consisting of cohesive modules over $A$ is denoted by $\mathcal P_{A}$.
\end{definition}

\begin{definition}
Let $\mathcal{C}$ be a differential graded category, the associated \emph{homotopy category} $\Ho \mathcal{C}$ is the category with the same objects as $\Ccal$ and whose morphisms are the 0th cohomology
\begin{equation*}
    \Ho \Ccal(x,y) = H^0(\Ccal(x,y)). \hfill \qedhere
\end{equation*}
\end{definition}
The dg-category $\Ho \Pcal_A$ is canonically triangulated for every curved dga $A$.

\begin{remark}
    The original definition of cohesive modules is due to Block \cite{BlockNC1} and is easily seen to be equivalent to Definition \ref{definition:cdga} in the case that $A$ is concentrated in non-negative degrees. Let us thus assume that $A$ is concentrated in non-negative degrees. Then a cohesive module over $A$ is equivalently specified by a bounded, graded $A^0$-module $E^\#$ consisting of finitely generated projective $A^0$-modules together with an $A$-superconnection, i.e.\ a map
    \[ \mathbb E \colon  A^\bullet \otimes_{A^0} E^\bullet \to A^\bullet \otimes_{A^0} E^\bullet \]
    which fulfills the Leibniz rule and satisfies the curvature equation \[\mathbb E^2(m) = B.m.\]
    One can decompose $\mathbb E = \mathbb E^0 + \mathbb E^1 + ...$ with
    \[ \mathbb E^i \colon E^\bullet \to A^i \otimes_{A^0} E^{\bullet+1-i}.\]
    The Leibniz rule is then equivalent to $\mathbb E^i$ being $A^0$-linear if $i \neq 1$ and fulfilling the Leibniz rule for $i = 1$. The curvature equation decomposes into
    \[\sum_{a+b = k} \mathbb E^a \mathbb E^b = 0 \quad \text{ for } k \neq 2 \]
    and 
    \[
    \mathbb E^0 \mathbb E^2 + \mathbb E^2 \mathbb E^0 + (\mathbb E^1)^2 = B.
    \]
    In particular, $(E^\bullet, \mathbb E^0)$ is a complex of finitely generated projective module $A^0$-modules. It turns out that the cohesive module $(A^\bullet \otimes_{A^0} E^\bullet, \mathbb E)$ is isomorphic to the zero-object in $\Ho (\mathcal P_A)$ if and only if the complex $(E^\bullet, \mathbb E^0)$ is acyclic.
\end{remark}
\begin{remark}
    The homotopy category of cohesive modules over a non-negatively graded curved algebra $A$ is equivalent to the compact objects in the so-called \emph{twisted derived category} of $A$ and conversely the twisted derived category can be recovered as the derived category of $\mathcal P_A$, see \cite{HH2} and references therein.
    If $A$ is a dg-algebra the twisted derived category is a refinement of the derived category.
\end{remark}
A fundamental example is the following. Let $X = (M,I)$ be a complex manifold. We obtain the Dolbeault dg-algebra:
\[ \mathbf A(X) := (\Omega^{0,\bullet}(X), \wedge, \bar \partial,0). \]
Block proved the following theorem:
\begin{theorem}[{\cite[Theorem 4.1.3]{BlockNC1}}]\label{th:block}
Let $X$ be a compact complex manifold, and consider its Dolbeault algebra $\mathbf A(X)$. Then the homotopy category of cohesive modules $\Ho \Pcal_{\mathbf{A}(X)}$ is equivalent to $\Dbcoh(X)$. 
\end{theorem}
This theorem has been strengthened to non-compact complex manifolds in \cite{RHMC}. Other natural examples of cohesive modules in topology and geometry include $\infty$-local systems, see \cite{BlockHigher}, \cite{RHMC} and  \cite{Antweiler24} (in the presence of curvature). 
\begin{example}
    The equivalence between ${\rm Ho} \Pcal_{\mathbf{A}}$ and $\Dbcoh(X)$ is a generalization of the classical Koszul-Malgrange theorem. It states that a holomorphic vector $E$ bundle is equivalent to a smooth vector bundle together with a partial flat connection on it. In our terms, the corresponding cohesive module is of the form $\Gamma^\infty(E \otimes_{\mathcal O_X} (\Omega^{0,\bullet}, \bar \partial)) =(\Omega^{0,\bullet}(X,E), \bar \partial_{E})$ where $\bar \partial_{E}$ is the unique extension of $\bar \partial_E \colon \Gamma^{\infty}(E) \to \Omega^{0,1}(M,E)$ compatible with the Leibniz rule.
\end{example}
\section{Deforming complex structures}\label{sect:deformations}

 \subsection{Deformation theory of Dolbeault algebras and generalized complex structures}
    Let $(M, \Jcal)$ be a GC manifold. According to \cite{Gualtieri:2003dx}, the moduli space of GC structures close to $\Jcal$ is controlled by the dg Lie algebra (DGLA)
    \[ (\Gamma(\Wedge^{\bullet+1} L_{\Jcal}), d_{CE}, [\cdot, \cdot]_{GH})\]
    where $L_{\Jcal}$ is the $+\ii$-eigenbundle of $\Jcal$, $[\cdot, \cdot]_{GH}$ is the Gerstenhaber bracket induced by the Courant bracket, and $d_{CE}$ is the Chevalley-Eilenberg differential under the identification $L_{\Jcal} \cong (\bar L_{\Jcal})^*$.
    In the special case of a complex structure $I$ and $\Jcal_I = \begin{pmatrix}
        I &0 \\ 0 &-I^*
    \end{pmatrix}$,
    the $+\ii$-eigenbundle reads 
    \[ L_{\Jcal_I} = (T^{0,1} M)^*  \oplus T^{1,0} M \]
    and the deformation DLGA takes the form
    \[ (\Gamma(\Wedge^{0,\bullet} T^*M \otimes \Wedge^{\bullet+1,0} TM), \bar \partial, [\cdot, \cdot]_{GH}), \]
    which serves as a model for the derived global sections $R(\Gamma_{\mathcal O_X}(\Wedge^{\bullet+1} T_{\mathrm{hol}} M), [\cdot, \cdot]_{GH})$. By the Hochschild-Kostant-Rosenberg Theorem, its cohomology is equal to 
    \[ H^\bullet(R\Gamma_{\mathcal O_X}(\Wedge^{\bullet} T_{\mathrm{hol}} M)) = \mathrm{Ext}^\bullet_{\mathcal O_{X \times X}}(\Delta_* \mathcal O_X, \Delta_* \mathcal O_X).\]
    The latter space is often taken as the definition of the Hochschild cohomology $HH^\bullet(X)$ of $X$. 
    Hence, first order deformations of $\Jcal_{I}$ are controlled by the second Hochschild cohomology group
    \begin{align}\label{eq:SecondHH} HH^2(X) = H^2(\mathcal O_X) \oplus H^1(TX) \oplus H^0(\Wedge^2 TX).\end{align}
    
    Moreover, if $X$ is algebraic, then the Hochschild cohomology of $X$ equals the Hochschild cohomology of $\Dbcoh(X)$
    and this is expected to control deformations of $\Dbcoh(X)$.
    
    In the non-algebraic case, $\Dbcoh(X)$ is known not to be a strong invariant of $X$ and computations of its Hochschild cohomology are not available, due to lack of representability of functors. 
    
    In recent work, the first author introduced continuous Hochschild complexes $HC_{\mathrm{cont}}(A)$ of complete locally convex (curved) dg-algebras $A$. Their cohomology $HH^\bullet(A)$ can be computed in terms of derived endomorphisms of $A$ in a version of the contraderived category. In the case of the Dolbeault dg-algebra $\mathbf A(X)$ endowed with its natural Fréchet topology, he obtained: 
\begin{theorem}[Theorem 2.2.3 in \cite{antweilerContinuous}]
	\label{CtrDolbeault}
	Let $X = (M,I)$ be a complex manifold.
	There is an $L_\infty$-quasi-isomorphism \[(\Gamma(\Wedge^{0,\bullet} T^*X \otimes \Wedge^{\bullet+1,0} TX), \bar \partial, [\cdot, \cdot]_{GH}) \xrightarrow{\sim} HC_{\mathrm{cont}}(\mathbf A(X)).\] In particular, 
	\[ HH^\bullet_{\mathrm{cont}}(\mathbf A(X)) = H^\bullet(X, \Wedge^\bullet T_{\mathrm{hol}} X) = HH^\bullet(X).\]
\end{theorem}
In particular, the formal deformation theories of $X = (M, I)$ as a GC manifold and of $\mathbf A(X)$ as a complete locally convex cdg-algebra are equivalent. 

\begin{example}
    In the case of a complex torus $V / \Gamma$ with constant complex structure $I \colon V \to V$ the space of first order deformations is equal to
\begin{align} HH^2(V/\Gamma,I) = \Wedge^{0,2} V_{\mathbb C}^* \oplus (V_{\mathbb C}^*)^{0,1} \otimes V_{\mathbb C}^{1,0} \oplus \Wedge^{2,0} V_{\mathbb C}.\hfill\qedhere
\end{align}
\end{example}
\subsection{First order deformation theory of categories}   
Let $X = (M,I)$ be a complex manifold. In this subsection we will discuss the first order deformations of the derived category of $X$ and of its Dolbeault algebra (which is identified with deformations of $X$ as a GC manifold by Theorem \ref{CtrDolbeault}). 

The formal deformation theory of dg-category or stable $\infty$-categories is closely related to their Hochschild cohomology, however due to the curvature problem this correspondence is not precise, see \cite{LowenCurvature} and \cite{Blanc_2018} for discussion of the problem and partial results.
In general, Maurer-Cartan elements in Hochschild cochains naturally give formal deformations of dg-categories as curved $A_\infty$-categories. 

Although the arguments in this section are expected to carry over to formal deformations (using the existence of a formal deformation quantization of any Poisson manifold by \cite{kon}),
 we shall restrict ourselves to the first order setting, i.e.\  deformations linear over the dual numbers 
 \[\mathbb C_{\epsilon} := \mathbb C[\epsilon]/\epsilon^2.\]
 Thus we can explicitly compare to the deformed categories of coherent sheaves considered in \cite{Toda} in the case of first order deformations. 
 To that end, we consider a class
\[ [B] + \beta \in H^2(\mathcal O_X) \oplus \Gamma_{\mathcal O_X}(\Wedge^{2,0} TM). \]
(we will ignore $H^1(TX)$ as it corresponds to deformations of the complex structure).
To $\beta$ we associate a new sheaf of algebras $\mathcal O_X^{\beta}$. As a sheaf of vector spaces, it has the form
\[ \mathcal O_X^\beta = \mathbb C_{\epsilon} \otimes \mathcal O_X = \mathcal O_X \oplus \epsilon \mathcal O_X\]
with multiplication determined by
\[ f \cdot^{\beta} g = fg + \epsilon \pi \ii  \beta(\partial f, \partial g) \]
for $f,g \in \mathcal O_X \subset \mathcal O_X^\beta$.

The class $[B]$ itself does not twist $\mathcal O_X^\beta$ but rather its category of representations. To describe it, we use the isomorphism $H^2(\mathcal O_X) \cong H^2(\Omega^{0,*}(X))$ under which we represent $[B]$ as a \v Cech 2-cocycle $\check{B} \in \check C_{cl}^2(\mathcal O_X, \{U_i\}_{i \in I})$ with respect to an open good cover $\mathcal U = \{U_i\}_{i\in I}$, so that $\check B$ corresponds to a collection  $(\check{B}_{ijk} \in \mathcal O(U_{ijk}))$. Here, and in what follows, we use the standard notation $U_{i_1\dots i_k} = U_{i_1}\cap\dots\cap U_{i_k}$ for $i_1,\dots,i_k\in I$. Consider the image $\check H$ of $\epsilon \check{B}$ under the exponential map $\exp(2\pi \ii (-))$, i.e., $\check{H}_{ijk} = 1 + 2\pi \ii \epsilon \check{B}_{ijk}$. This defines a \v Cech cocycle in $\check{C}^2( (\mathcal O_X^\beta)^\times, \{U_i\}_{i \in I})$ where we note that $(\mathcal O_X^\beta)^\times = \mathcal O_X^\times \oplus \epsilon \mathcal O_X$.
\begin{definition}
    An \emph{$\check{H}$-twisted sheaf $M$ over $\mathcal O_X^\beta$} is a collection $\{M_i\}_{i \in I}$ where each $M_i$ is a sheaf over $\mathcal O_{U_i}^\beta$ and a collection of isomorphisms $\phi_{ij} \colon M_{i}|_{U_{ij}}  \to M_{j}|_{U_{ij}}$ which are linear over $\mathcal O_{U_{ij}}^\beta$ and satisfy the twisted cocycle condition
    \[ \phi_{jk} \phi_{ij} = \check{H}_{ijk} \phi_{ik}.\]
    The notion of maps between $\check{H}$-twisted sheaves is defined in the same fashion. For consistency, it is important to note that $\check{H}$ is central. Denote by $D_{\mathrm{perf}}(\mathcal O_X^\beta, \check{H})$ the subcategory of the derived category of all $\check{H}$-twisted sheaves which are locally quasi-isomorphic to perfect complexes, which is independent of the choices involved. 
\end{definition}
On the other hand, we can define the following global $\CC_\epsilon$-linear cdg-algebra:
\[ \mathbf A(X,\epsilon \beta, \epsilon B) := (\mathbb C_{\epsilon} \otimes \Omega^{0,*}(X), \bar \partial, \wedge^{\beta}, 2\pi \ii \epsilon B),\]
where 
\[ \omega \wedge^\beta \eta := \omega \wedge \eta + \epsilon \pi \ii \iota_{\beta}( \partial \omega \wedge \partial \eta ) \]
 for $\omega, \eta \in \Omega^{0,*}(X)$. Clearly, $\mathbf A(X,\epsilon \beta, \epsilon B)$ can naturally be promoted to a sheaf of (soft) cdg-algebras on $X$, which we denote as $\underline{\mathbf A(X,\epsilon \beta, \epsilon B)}$.
When $B = 0$, there is a natural inclusion of sheaves of algebras $\mathcal O_X^\beta \to \underline{\mathbf A(X,\epsilon \beta, 0)}$ inducing quasi-isomorphisms on stalks by the Dolbeault--Grothendieck Lemma. In fact, we have the following:
\begin{proposition}
    There is an equivalence of $\mathbb C_\epsilon$-linear categories
    \[ \Ho (\mathcal P_{\mathbf A(X,\epsilon \beta, \epsilon B)}) \to D_{\mathrm{perf}}(\mathcal O_X^\beta, \check{H}).\]
    \begin{proof}[Proof sketch]
        The full proof is essentially an application of the techniques and results in \cite{RHMC} and \cite[Section 7]{Antweiler24}. We thus focus on defining the functor on objects. Fix the following data:
        \[ f_{ij} \in C^\infty(U_{ij}), \quad \omega_{i} \in \Omega^{0,1}(U_i) \] such that
        \[ \check B = \check d(f), \quad \bar \partial f = \check d(\omega),\quad  \bar \partial(\omega_i) = B|_{U_i}\]
        where $\check d$ denotes the \v Cech differential.
        
        Given a cohesive module $P \in \mathcal P_{\mathcal A(X,\epsilon \beta, \epsilon B)}$ we can make it into a sheaf by
        \[ \underline{P} = \underline{\mathbf A(X,\epsilon \beta, \epsilon B)} \otimes_{\mathcal A(X,\epsilon \beta, \epsilon B)} P. \]
        This tensor-product  will automatically result in a sheaf because of the  softness of $C^\infty(X)$ and (graded) projectivity of $P$. Define a complex of twisted sheaves linear over $\mathcal O_X^\beta$ as follows. For $i \in I$, set
        \[ P_i := (\underline{P}|_{U_i}, d - \epsilon 2\pi \ii \omega_i)\]
        and define isomorphisms 
        \[ \phi_{ij} = 1 + \epsilon 2 \pi \ii f_{ij} \colon P_i|_{U_{ij}} \to P_j|_{U_{ij}}. \]
        Then by construction 
        \[
        \check d \phi = \check H
        \]
        and $P_i$ is cohesive over $\underline{\mathbf A(U_i,\epsilon \beta, 0)}$ and hence also defines a module over $\mathcal O_{U_i}^\beta$. One can show that $P_i$ is locally quasi-isomorphic to a perfect complex of $\mathcal O_{U_i}^\beta$-modules.
    \end{proof}
\end{proposition}

\subsection{Twisted coherent sheaves}
In the previous section we studied twisted sheaves only as first order deformations. As we will see here, this restriction is not necessary.
Let us consider the first summand $H^2(\mathcal O_X)$ of the space of first order deformations \eqref{eq:SecondHH}. 
Choosing a real closed 2-form $B$ whose $(0,2)$-part represents a given class $[B^{0,2}] \in H^2(\mathcal O_X)$, the deformed GC structure is given by
\[ \Jcal_{I}^{B} = \btrans{B } \Jcal_{I} \btrans{-B }.\]
Note that $\Jcal_{I}^{B}$ only depends on the $(0,2)$-part of $B$. The $2\pi \ii$-exponential 
\[
\exp(2\pi \ii (-)) \colon \mathcal O_X \to \mathcal O_X^*
\]
maps $B^{0,2}$ to a class $\tB \in H^2(\mathcal O_X^*)$. We will review here again the notion of twisted sheaf, this time non-infinitesimally. We can view $\tB$ as twisting the cocycle-condition for modules over $\mathcal O_X$. More precisely, choose an open cover $\{U_i\}_{i \in I}$ of $X$ for which we can view $\tilde B$ as a \v Cech 2-cocycle  $\tB \in \check C_{cl}^2(\mathcal O_X^*, \{U_i\}_{i \in I})$. 
A $\tB$-\emph{twisted sheaf} over $\mathcal O_X$ is then defined to be a collection $\{M_i\}_{i\in I}$ of sheaves of $\mathcal O_{U_i}$-modules together with isomorphisms 
\[
\phi_{ij} \colon M_i|_{U_{ij}} \to M_j|_{U_{ij}} 
\]
which fulfill the twisted cocycle condition 
\[
\phi_{jk}\phi_{ij} = \tB_{ijk} \phi_{ik}.
\]
Many notions defined for sheaves or complexes of sheaves carry over to the case of twisted sheaves. For example, a twisted sheaf is called \emph{coherent} if every $M_i$ is a coherent $\mathcal O_{U_i}$-module. In particular, we can consider the derived category $\Dbcoh(X,\tB)$ of complexes $\tB$-twisted sheaves with coherent cohomology, see e.g. \cite{Caldararu2002Elliptic}.
We propose to define the derived category of the GC structure $\Jcal_{I}^{B}$ as
\[ \D(X,\Jcal_{I}^{B}) :=  \Dbcoh(X,\tB). \]
By definition, changing $B$ by an integral 2-form, does not change the latter category (as it does not change $\tB$). It is also easy to see that the latter category only depends on the class $\tilde H$ and not on the choice of cover. 
\begin{definition}\label{definition:BDolbeault}
    Let $B^{0,2} \in \Omega^{0,2}(X)$ be a $\bar \partial$-closed $(0,2)$-form. We define the \emph{generalized Dolbeault algebra} of the GC structure $\Jcal_{I}^{B}$ as
    \[ \mathbf A(I,0,B) := (\Omega^{0,\bullet}(X),\wedge,\bar \partial, 2\pi \ii {B^{0,2}}). \hfill\qedhere\]
\end{definition}
    Note that $\mathbf A(I,0,B)$ is indeed a cdg-algebra as
    \[
    \bar \partial^2(\alpha) = 0 = 2\pi i({B^{0,2}} \wedge \alpha - \alpha \wedge {B^{0,2}} )\text{ for all } \alpha \in \mathbf A(I,0,B^{0,2})
    \]
    since $\wedge$ is graded-commutative.

The reason for this definition is given by the following theorem, mentioned in the introduction:
\begin{theorem}[\cite{BlockNC1}, \cite{Antweiler24},\cite{BlockDaenzer}]\label{theorem:derivedcat}
The homotopy category ${\rm Ho} \Pcal_{\mathbf{A}(I,0,B)}$ is equivalent to the bounded derived category $\Dbcoh(X,\tB)$ of complexes of \emph{$\exp(2\pi \ii B^{0,2})$-twisted sheaves} of $\Ocal_X$-modules with coherent cohomology.
\end{theorem}

\subsection{Non-commutative tori}
Let us now consider the last summand $H^0(\Wedge^2 TX)$ of the space of first order deformations \eqref{eq:SecondHH}.
By Theorem \ref{CtrDolbeault} any holomorphic Poisson bivector (corresponding to the last piece) thus yields a formal deformation quantization of $\mathcal A(X)$. We are interested in the case where the quantization is strict, rather than formal. This is possible for a torus $X = (V/\Gamma, I \colon V \to V)$. In that case we have 
\[ H^0(\Wedge^2 TX) = \Wedge^{2,0} V_{\mathbb C}.\]
To define a non-formal deformation quantization associated to a complex bivector, choose a real bivector $\beta \in \Wedge^2 V$ whose $(0,2)$-component equals the given complex bivector. 

Recall that Fourier duality induces an isomorphism of Fréchet spaces
\[ C^\infty(V/\Gamma) \cong \mathcal S(\Gamma^*) \]
where $\mathcal S(\Gamma^*)$ is the Schwartz space on $\Gamma^*$, i.e., the space of series $\sum_{\gamma \in \Gamma^*} \lambda_\gamma u_\gamma$ for which the coefficients $\lambda_\gamma$ decrease quicker than any polynomial. The isomorphism $\mathcal S(\Gamma^*) \to C^\infty(V/\Gamma)$ sends $u_\gamma$ to the function $x \mapsto u_{\gamma}(x) := e^{2\pi i \gamma(x)}$. Our constructions will be given in terms of the Fourier modes $u_\gamma$.

\begin{definition}
    We define the $\beta$-deformed Dolbeault algebra as 
\begin{equation*}
    \mathbf{A}(I, \beta, 0) = (C^\infty(V / \Gamma) \otimes \Wedge^{(0,\bullet)} V_{\mathbb C}^*, \star^{\beta}, \bar \partial_{I}, 0),
\end{equation*}
where 
\begin{equation} \label{eq:starformula} (u_{\gamma_1} \otimes \omega_1) \star^{\beta} (u_{\gamma_2} \otimes \omega_2) := e^{\pi \ii \beta(\gamma_1,\gamma_2)} u_{\gamma_1 + \gamma_2} \otimes \omega_1 \wedge \omega_2\end{equation}
    for $u_{\gamma}(x) := e^{2 \pi \ii \gamma(x)}$, $\gamma \in \Gamma^*$ and $\omega_i \in \Wedge^{(0,\bullet)} V_{\mathbb C}^*$.
\end{definition}
The algebras $\mathbf A(I,\beta,0)$ and $\mathbf A(I,0,B)$ were first considered in \cite{Block2006Mukai}.
We will also use the notation $C^\infty(V/\Gamma,\beta)$ for the smooth functions on $V/\Gamma$ whose product is deformed as
\begin{equation}
\label{eq:deformedproduct}
    u_{\gamma_1} \star^{\beta} u_{\gamma_2} := e^{\pi \ii \beta(\gamma_1,\gamma_2)} u_{\gamma_1 + \gamma_2}.
\end{equation}

We remark that $\beta$ in the formula for $\star^{\beta}$ needs to be real, otherwise the product would not be well-defined.

\begin{lemma}
    \label{lemmma:integralBeta}
    Up to isomorphism of algebras, $C^\infty(V/\Gamma,\beta)$ and $\mathbf{A}(I, \beta, 0)$ only depend on $[\beta] \in \Wedge^2 V / \Wedge^2 \Gamma$.
    \begin{proof}
        Let $\beta \in \Wedge^2 V$ and $\sigma \in \Wedge^2 \Gamma$. To construct an isomorphism 
        \[ \Phi \colon C^\infty(V/\Gamma,\beta) \to C^\infty(V/\Gamma, {\beta + \sigma}) \]
        we choose a $\mathbb Z$-basis of $\Gamma^*$, denoted $\alpha_1,...,\alpha_n$. Then there is a unique continuous algebra map $\Phi$ satisfying
        \[ \Phi(u_{\alpha_i}) = u_{\alpha_i} \quad \forall i = 1,...,n. \]
        More specifically it is given by
        \[ \Phi(u_{\gamma}) = \exp\left(\pi \ii \sum_{i < j} \sigma(\gamma_i \alpha_i, \gamma_j \alpha_j)\right) u_{\gamma} \] 
        where $\gamma = \sum_{i} \gamma_i \alpha_i$.
    \end{proof}
\end{lemma}

\subsection{Generalized Dolbeault algebra}
We now combine the two constructions in one definition:
\begin{definition}
\label{definition:deformeddolbeault}
Let $V$ be a real vector space together with a linear complex structure $I \colon V \to V$ and let $\Gamma \subset V$ be a lattice. Let
$\beta \in \Wedge^2 V$ be a real bivector and let $B \in \Wedge^2 V^*$ be a real 2-form. Then the $(I, \beta, B)$-\emph{generalized Dolbeault algebra} is defined to be
\begin{equation*}
    \mathbf{A}(I, \beta, B) = (C^\infty(V/\Gamma) \otimes \Wedge^\bullet (V_{\mathbb C}^*)^{0,1}, \star^\beta, \bar \partial_{I}, 2 \pi \ii B^{0,2}). \hfill\qedhere
\end{equation*}
\end{definition}
The cdg-algebras $\mathbf{A}(I,\beta,B)$ (and $\mathbf A(X,0,B)$ for arbitrary compact $X$ and $B$-field $B$) come equipped with certain functional-analytic packages, namely they are \emph{Fr\'echet-Fredholm}, see \cite[Definition 2.1.10]{NCGrauert}. 
We summarize some of the main results of \cite{NCGrauert} below. For the purpose of the present paper, the reader may skip the technical details. 
The main point is that the categories $\mathcal P_{\mathbf{A}(I,\beta,B)}$ are locally proper, i.e.\ the hom complexes are perfect over the ground field, 
and that there exists a well-defined defined Fourier-Mukai calculus:
\begin{theorem}[{\cite[Theorem 2.5.3]{NCGrauert}}]\label{thm:grauert}
    Let $A$ be a Fréchet-Fredholm cdg-algebra. Then the morphism complexes in $\mathcal P_{A}$ have finite-dimensional cohomology. More generally, let $A$, $B$, $C$ be Fréchet-Fredholm cdg-algebras. Then the tensor product over $B$ induces an $A_\infty$-functor 
     \[ \mathcal P_{A \hat \otimes B^{op}} \times \mathcal P_{B \hat \otimes C^{op}} \to \mathcal P_{A \hat \otimes C^{op}}.\]
\end{theorem}
\begin{proposition}[{\cite[Lemma 4.1.1]{NCGrauert}}]\label{prop:locallyproper}
    The cdg-algebra $\mathbf{A}(I, \beta, B)$ is Fréchet-Fredholm.
\end{proposition}
\begin{proposition}[{\cite[Lemma 4.1.2]{NCGrauert}}]
    The cdg-algebra $\mathbf{A}(I, \beta, B)$ is Fréchet twist-smooth. In particular, there exists a canonical cohesive module 
    \[
    \Delta \in \mathcal P_{\mathbf{A}(I, \beta, B) \hat \otimes \mathbf{A}(I, \beta, B)^{op}}
    \] 
    representing the identity functor 
    \[
    \mathcal P_{\mathbf{A}(I, \beta, B)} \to \mathcal P_{\mathbf{A}(I, \beta, B)}.
    \]
\end{proposition}

\begin{definition}
    Let $A$ and $B$ be Fréchet-Fredholm cdg-algebras which are also Fréchet twist-smooth. 
    We call $A$ and $B$ \emph{cohesive Morita equivalent} if there exist cohesive bimodules $R \in \mathcal P_{A \hat \otimes B^{op}},S \in \mathcal P_{B \hat \otimes A^{op}}$ such that
    \[
    R \hat \otimes_B S \simeq \Delta_A
    \]
    and 
    \[
    S \hat \otimes_A R \simeq \Delta_B,
    \]
    where $\simeq$ denotes continuous homotopy equivalence.
\end{definition}
From Theorem \ref{thm:grauert} we obtain the following Corollary.

\begin{corollary}[\cite{NCGrauert}]
    Let $A$ and $B$ be Fréchet twist-smooth, Fréchet-Fredholm cdg-algebras. Suppose that $A$ and $B$ are cohesive Morita equivalent. Then their dg-categories of cohesive modules are $A_\infty$-quasi-equivalent. 
\end{corollary}
\begin{remark}
    For the reader uncomfortable with $A_\infty$-functors we mention that it is possible to induce an equivalence of categories by a dg functor at the price of enlarging the category of cohesive modules.
    Any bimodule induces a natural functor between the larger contraderived categories of $A$ and $B$ as in the corollary.
    Under the assumptions of the corollary it restricts to the subcategories of modules weakly equivalent to cohesive modules.
    We have to replace it  by an $A_\infty$-functor to consider a functor between the categories of honest cohesive modules.  
\end{remark}

\begin{example}
    Let $\mathbf A(X), \mathbf A(Y)$ be the Dolbeault algebras of compact complex manifolds $X$ and $Y$. 
    Then we may translate from cohesive modules to elements of the bounded derived category. $\mathbf A(X)$ and $\mathbf A(Y)$ are cohesive Morita equivalent if and only if they are related by invertible Fourier-Mukai transforms, i.e.\ if there exist
    \[ R \in \Dbcoh(X \times Y),\quad  S \in \Dbcoh(Y \times X) \]
    such that 
    \[ \Delta_* (\mathcal O_X) \simeq R(p_{X \times X})_*( p_{X \times Y}^* R \otimes^L p_{Y \times X}^* S) \in \Dbcoh(X \times X),\]
    where 
    \begin{align*}
        p_{X \times X}\colon X \times Y \times X &\to X \times X\\ p_{X \times Y}\colon X \times Y \times X &\to X \times Y\\
        p_{Y \times X}\colon X \times Y \times X &\to Y \times X. 
    \end{align*} are the projections and
    \[ \Delta_* (\mathcal O_Y) \simeq R(q_{Y \times Y})_*( q_{Y \times X}^* S \otimes^L q_{X \times Y}^* R) \in \Dbcoh(Y \times Y).\]
    Here the projections $q_{(-)}$ are defined dually.
\end{example}

\section{Non-commutative Fourier-Mukai kernels}\label{section:Heisenberg}
\subsection{Heisenberg modules for smooth tori}
Let $V/\Gamma$ now be a real torus of even dimension and $\sigma \in  \Wedge^2 V$ a real translation invariant \emph{non-degenerate} bivector. We also let $\omega  := \sigma^{-1}$ be the corresponding translation invariant symplectic form. In this case, we will construct a canonical module over the real non-commutative torus $C^{\infty}(V/\Gamma,{\sigma})$ which is known as the \emph{Heisenberg module}. It will be finitely generated and projective. 

The Heisenberg Lie algebra $\mathfrak h$ is defined as the central extension
\[ 0 \to \mathbb C \to \mathfrak h \to V_{\mathbb C}^* \to 0 \] 
corresponding to $2\pi \ii\sigma$ with corresponding splitting $r \colon V^*_{\mathbb C} \to \mathfrak h$. 
For $v,w \in V^*$ we have by definition
\begin{align} \label{equation:cr} [r(v),r(w)] = 2\pi \ii \sigma(v,w).\end{align}
We choose complementary Lagrangian subspaces $L,L' \subset (V^*, \sigma)$ which yields the isomorphism 
\[
\rho \colon L' \to L^*,\quad  \rho(l') = \sigma (l')|_{L} \in L^*.
\]
Hence, we obtain an identification of symplectic spaces
\[ \Phi \colon (L \oplus L^*, \omega_{\langle \cdot, \cdot \rangle}) \cong (V^*, \sigma), \quad (l,\alpha) \mapsto l + \rho^{-1}(\alpha).\quad \]
where $\omega_{\langle \cdot, \cdot \rangle}((l,\alpha),(l',\alpha')) = \alpha(l') - \alpha'(l)$. The action of the Heisenberg Lie algebra $\mathfrak h$ on $\mathcal S(L)$ (the Schwartz space on $L$) is determined by
\begin{equation}\label{eq:action} r(\Phi(l,\alpha)).f(x) := 2\pi \ii \alpha(x) f(x) - \partial_{l} f(x) \end{equation}
and we will write 
\[(v^*).f := r(v^*).f\quad \text{ for } v^* \in V^*_{\mathbb C}.\]
\begin{definition}
    Define an action of the non-commutative torus $C^{\infty}(V/\Gamma,\sigma)$ on $\mathcal S(L)$ by setting
\[ (u_{\gamma} . g)(x) = e^{\gamma}g(x) = e^{\pi \ii \alpha(2x-l)} g(x-l) ,\]
for $L^* \oplus L \ni (\alpha,l ) = \gamma \in \Gamma^*$ and $u_{\gamma}(x) := e^{2\pi \ii \gamma(x)}$. See \cite[Remark 4.2.1]{NCGrauert} for an interpretation of the notation $e^{\gamma}g(x) = e^{2\pi \ii \alpha(2x-v)} g(x-v)$. 
It is not difficult to check that the formula is well-defined (and continuous). The resulting module is finitely generated projective over $C^{\infty}(V/\Gamma,\sigma)$ and we denote it by $H(\sigma)$.
\end{definition}

\subsubsection{Connections}
We define the (complex) non-commutative de-Rham algebra of $(V/\Gamma,\sigma)$ as the dg-algebra
\[ \Omega(V/\Gamma,\sigma) := (C^{\infty}(V/\Gamma) \otimes \Wedge^\bullet V_{\mathbb C}^*, \star^{\sigma}, d_{DR}, 0),\] 
where $\star^\sigma$ is given by the formula \ref{eq:starformula}.
Given a module $M$ over $C^{\infty}(V/\Gamma,\sigma)$, a \emph{connection} is given by a map of degree $1$
\[ \nabla \colon M \to \Omega^1(V/\Gamma,\sigma) \hat \otimes_{C^\infty(V/\Gamma,\sigma)} M \]
which fulfills the Leibniz rule. Note that $\nabla$ is flat, i.e. $\nabla^2 = 0 \colon  M \to \Omega^2(V/\Gamma,\sigma) \hat \otimes_{C^\infty(V/\Gamma,\sigma)} M$, if and only if $\Omega^\bullet(V/\Gamma,\sigma) \hat \otimes_{C^\infty(V/\Gamma,\sigma)} M$ defines a dg-module over $\Omega(V/\Gamma,\sigma)$.
More generally, we can consider the following curved algebra:
\begin{definition}
    The $(\sigma,B)$ deformed de-Rham algebra is the curved dg-algebra 
    \[ \Omega(V/\Gamma,\sigma,B) := (C^{\infty}(V/\Gamma) \otimes \Wedge^\bullet V_{\mathbb C}^*, \star^{\sigma}, d_{DR}, 2\pi \ii B). \qedhere \]
\end{definition}
We call a connection $\nabla$ on $M$ \emph{projectively flat} with curvature $B \in \Wedge^2 V^*_{\mathbb C}$ if $\nabla^2$ has the form
\[ \nabla^2(m) = 2\pi \ii B \otimes m.\]
In that case, $(\Omega^\bullet(V/\Gamma,\sigma) \hat \otimes_{C^\infty(V/\Gamma,\sigma)} M, \nabla)$ defines a cohesive module over the cdg-algebra $\Omega(V/\Gamma,\sigma,B)$.

\begin{lemma}
    The Heisenberg module $H(\sigma)$ carries a canonical projectively flat connection $\nabla$ with curvature 
    \[ \nabla^2 = -2\pi \ii \omega\]
    and hence $(\Omega^\bullet(V/\Gamma,\sigma) \hat \otimes_{C^\infty(V/\Gamma,\sigma)} H(\sigma), \nabla)$ is a cohesive module over the cdg-algebra $\Omega(V/\Gamma,\sigma,-\omega)$. 
    \begin{proof}
        We use the notation $\nabla_v(g) := \iota_v ( \nabla(g) )$ for $v \in V$. Then $\nabla$ is uniquely determined by defining
        \[ \nabla_v(g) := \omega (v).g. \]
        We verify the Leibniz rule using \eqref{eq:action}: 
        \begin{align*}
            \nabla_v(u_{\gamma}.g) &= \omega(v).(e^{\gamma}g) \\
        &= e^{\gamma}  \omega(v).g + [r(\omega(v)),r(\gamma)]. e^{\gamma}g 
        \\ 
        &= u_\gamma.(\nabla_v(g)) + 2\pi \ii \sigma( \omega(v), \gamma) u_\gamma.g\\
        &= u_\gamma.(\nabla_v(g)) + 2\pi \ii \gamma(v) u_\gamma.g \\
        &= u_\gamma.(\nabla_v(g)) + d_v(u_\gamma).g.
        \end{align*} 
        To get from the first to the second line one expands $\omega(v).(e^{\gamma}g)$ and uses the commutation relations between $r(\omega(v))$ and $r(\gamma)$.
        
        Now we compute the curvature:
        \begin{align*}
            [\nabla_v, \nabla_w](g) &= \omega (v).\omega (w).g - \omega (w).\omega (v).g \\
            &= [r(\omega (v)),r(\omega (w))].g\\
            &= 2\pi \ii \sigma( \omega (v), \omega (w) ).g\\
            &= -2\pi \ii \omega( v, w ).g. \qedhere
        \end{align*}
    \end{proof}
\end{lemma}
The constructed Heisenberg module is defined for $C^{\infty}(V/\Gamma,\sigma)$ where $\sigma$ is non-degenerate. However, in view of Lemma \ref{lemmma:integralBeta}, it also defines a module over $C^{\infty}(V/\Gamma,\beta)$  if $\beta\in \Wedge^2 V$ satisfies \[\beta \equiv \sigma \mod \Wedge^2 \Gamma.\]
\begin{example}
\label{example:commutativePC}
    Let $V/\Gamma$ be a torus and $V^* / \Gamma^*$ its dual torus. Smooth complex line bundles on $V/\Gamma \times V^* / \Gamma^*$ are parametrized by
    \[ H^2(V/\Gamma \times V^* / \Gamma^*, \mathbb Z) = \Wedge^2 (\Gamma^* \oplus \Gamma) \cong  (\Gamma^* \otimes \Gamma) \oplus \Wedge^2 \Gamma \oplus \Wedge^2 \Gamma^*.\]
    The canonical pairing between $\Gamma^*$ and $\Gamma$ yields an element 
    \[ -\omega \in \Gamma^* \otimes \Gamma\]
    so that $-\iota_\gamma(\omega) = \gamma$ for $\gamma \in \Gamma$.
    The corresponding line bundle with curvature $-2 \pi \ii \omega$ is known as the \emph{Poincaré line bundle}. As a map $\Gamma \oplus \Gamma^* \to (\Gamma \oplus \Gamma^*)^* = \Gamma^* \oplus \Gamma$, $\omega$ has the form
    \[  -\omega  = \begin{pmatrix}
 0 & -1 \\ 1 & 0
     \end{pmatrix}\]
     and so the corresponding bivector $\sigma := \omega^{-1}$ is of the form
     \[
     \sigma  = \begin{pmatrix} 0 & -1 \\ 1 & 0
     \end{pmatrix},
     \]
     which lies on the lattice $\Wedge^2 (\Gamma \oplus \Gamma^*)$, hence $C^{\infty}(V/\Gamma \times V^* / \Gamma^*,\sigma) \cong C^{\infty}(V/\Gamma \times V^* / \Gamma^*)$ and the Heisenberg module $H(\sigma)$ is isomorphic to (the global sections of) the  Poincaré line bundle.
\end{example}
\subsubsection{Bimodules} 
Suppose we are given a bivector $\beta \in \Wedge^2 V$ for a torus $V/\Gamma$ and a bivector $\tbeta \in \Wedge^2 V^*$ for its dual torus $V^* / \Gamma^*$. We are interested in left modules over 
\[ C^\infty(V/\Gamma,\beta) \hat \otimes C^\infty(V^*/\Gamma^*,{\tbeta})^{op}\cong  C^\infty(V/\Gamma \times V^*/\Gamma^*, \beta \otimes 1 - 1 \otimes \tbeta) \cong C^\infty((V\oplus V^*)/(\Gamma \oplus\Gamma^*),\sigma_0).\] 
The symbol $\hat \otimes$ denotes the completed projective tensor product. The bivector $\sigma_0 = \beta \otimes 1 - 1 \otimes \tbeta$  may then be written in the form
\[ \sigma_0 = \begin{pmatrix}
 \beta & 0 \\ 0 & -\tbeta
     \end{pmatrix}\colon (V\oplus V^*)^*= V^*\oplus V\to V\oplus V^*.\]

Attempting to generalize commutative Poincaré line bundles as in Example \ref{example:commutativePC}, we modify $\sigma_0$ by an integral bivector to obtain
\[ \sigma := \begin{pmatrix}
 \beta & -1 \\ 1 & -\tbeta
     \end{pmatrix}. \]
If $\sigma$ is invertible, we obtain the Heisenberg module $H(\sigma)$ for the algebra $C^\infty(V/\Gamma \times V^*/\Gamma^*,{\sigma}) \cong C^\infty(V/\Gamma \times V^*/\Gamma^*, \beta \otimes 1 - 1 \otimes \tbeta)$
with curvature $-2\pi \ii \sigma^{-1}$ hence defining a cdg-module over  $\Omega(V/\Gamma \times V^* / \Gamma^*,\sigma, -\sigma^{-1})$. The inverse of $\sigma$ is given by
\begin{align} \label{equation:sigma} -\sigma^{-1} = \begin{pmatrix}
                \tbeta(1-\beta \tbeta)^{-1}  & -(1-\tbeta \beta)^{-1}
                \\ (1-\beta \tbeta)^{-1} & -\beta(1-\tbeta \beta)^{-1}
            \end{pmatrix} \end{align}
and in particular, invertibility of $\sigma$ is equivalent to $1-\beta \tbeta \colon V \to V$ being invertible.
\subsection{Non-commutative Poincaré line bundles}
Recall we constructed the Heisenberg module $H(\sigma)$ for $\sigma$ a given non-degenerate bivector and constructed a connection making it into a cohesive module over the twisted de-Rham algebra
\[ \Omega(V/\Gamma,\sigma,B) := (C^{\infty}(V/\Gamma) \otimes \Wedge^\bullet V_{\mathbb C}^*, \star^\beta, d_{DR}, 2\pi \ii B).\]
If $I$ is any complex structure on $V$, the projection $V_{\mathbb C}^* \to (V^*_{\mathbb C})^{0,1}$ induces a map of cdg-algebras 
\begin{align} \label{equation:extensions}\Omega(V/\Gamma,\sigma,B) \to \mathbf{A}(I, \sigma, B) = (C^{\infty}(V/\Gamma) \otimes \Wedge^{(0,\bullet)} V_{\mathbb C}^* ,\star^{\sigma}, \bar \partial, 2\pi \ii B^{0,2}) \end{align}
and hence an extension of scalars functor between cohesive modules.
\begin{example}
    Let us explain the role of the de-Rham algebra in the untwisted, commutative case. For any smooth manifold, the Higher Riemann-Hilbert correspondence \cite{RHMC} establishes an equivalence
    \[ D_{\mathrm{fclc}}(\underline{\mathbb C}) \cong H^0( \mathcal P_{\Omega(M)} ) \]
    where $D_{\mathrm{fclc}}(\underline{\mathbb C})$ denotes the subcategory of the derived category of complexes of sheaves of vector spaces whose cohomology sheaves are locally finitely generated and locally constant. If $X = (M,I)$ is a complex manifold then the unit $\underline{\mathbb C} \to \mathcal O_X$ induces a functor $D_{\mathrm{fclc}}(\underline{\mathbb C}) \to \Dbcoh(\mathcal O_X)$ which agrees with the functor on the level of cohesive modules induced by the map of dg-algebras  $\Omega(M) \to \mathbf A(X)$. In the curved and non-commutative case, $\mathcal P_{\Omega(V/\Gamma,\sigma,B)}$ should thus be thought of as a higher-categorical version of twisted non-commutative locally constant sheaves. 
\end{example}

We would like to use Heisenberg modules to define non-commutative Fourier--Mukai kernels between a torus $V/\Gamma$ and its dual torus $V^*/\Gamma^*$ equipped with complex structures $I\in\C_V, \tilde I\in\C_{V^*}$, 2-forms $B\in\Wedge^2V^*,\tB\in\Wedge^2V$ and bivectors $\beta\in\Wedge^2 V,\tbeta\in\Wedge^2V^*$, respectively. . Thus, we consider the case where the RHS of \eqref{equation:extensions} has the product form:
\begin{align} &\mathbf{A}(I, \beta, B) \hat \otimes \mathbf{A}(\tilde I, \tbeta, \tB)^{op}.
\end{align}
We see that $\sigma_0 \in \Wedge^2(V\oplus V^*)$ and $\underline{B} \in \Wedge^2(V\oplus V^*)^* = \Wedge^2(V^*\oplus V)$ as on the left side of \eqref{equation:extensions} must take the form 
\[ \sigma_0  = \begin{pmatrix}
                \beta   & 0
                \\ 0 & -\tbeta 
            \end{pmatrix}\colon (V\oplus V^*)^*= V^*\oplus V\to V\oplus V^*\]
and 
        \[ \underline{B} = \begin{pmatrix}
                B  + B_{\Delta}   &  -\alpha^*
                \\ \alpha & -\tilde{B} - \tB_{\Delta} 
            \end{pmatrix} \colon  V\oplus V^*\to  (V\oplus V^*)^*= V^*\oplus V \]
with $\alpha \colon V\to V^*$ being of type $(1,1)$ for $I \times \tilde{I}$, i.e.,
\begin{align} \alpha I = -\tilde{I}^* \alpha, \end{align}
$B_{\Delta} \in\Wedge^2V^*$ of type $(1,1)$ for $I$, and $\tB_{\Delta} \in\Wedge^2V$ of type $(1,1)$ for $\tilde I$. If $(1-\beta\tbeta)$ is invertible, we may proceed as in \S 6.1.2, and obtain the Heisenberg module $H(\sigma)$ associated with the non-degenerate bivector 
\[
\sigma = \begin{pmatrix}
    \beta & -1 \\ 1 & \tbeta
\end{pmatrix}.
\] 
Comparison with \eqref{equation:sigma} shows that $\underline{B}$ must then be given by
\begin{align*} \underline{B} = - (\sigma )^{-1} = \begin{pmatrix}
                \tbeta (1-\beta  \tbeta )^{-1}  & (1-\tbeta  \beta )^{-1}
                \\ (1-\beta  \tbeta )^{-1} & -\beta (1-\tbeta  \beta )^{-1}
            \end{pmatrix}.\end{align*}
In particular, comparing the lower left entries in these formulas for $\underline{B}$, we find that $\alpha = (1-\beta  \tbeta)^{-1}$ is invertible and hence the complex structure $\tilde{I}$ is determined by
\begin{align}
    \tilde{I} = -(\alpha^*)^{-1} I^* \alpha^* = -(1-\tbeta \beta ) I^* (1-\tbeta \beta )^{-1} \colon V^* \to V^*.
\end{align}
Note that if $\beta = 0$ or $\tbeta = 0$, one obtains $\tilde{I} = -I^*$, making $(V^* / \Gamma^*, \tilde{I})$ into the dual complex torus of $(V / \Gamma, I)$.
We summarise our construction in the following Proposition:

\begin{proposition}\label{proposition:bimodules}
    Let $\beta \in \Wedge^2 V$, $\tbeta \in \Wedge^2 V^*$ such that $1 - \beta  \tbeta  \colon V \to V$ is invertible. Let $I\colon V\to V$ be a complex structure. Then there exists a cohesive Heisenberg bimodule $P$, which we call the \emph{non-commutative Poincaré line bundle}, over 
    \[ \mathbf{A}(I, \beta, B) \hat \otimes \mathbf{A}(\tilde I, \tbeta, \tB)^{op}\]
    if 
    \[
    \betatrans{\beta } \btrans{B } L_{\Jcal_I}
    \]
    is T-dual to 
    \[
    \betatrans{\tbeta } \btrans{\tB } L_{\Jcal_{\tilde I^*}}
    \]
    via $\varphi$ (see \ref{sec:Tdualitypoint}), i.e., if 
    \[
    \varphi\left(\betatrans{\beta } \btrans{B } L_{\Jcal_I} \right) = \betatrans{\tbeta } \btrans{\tB } L_{\Jcal_{\tilde I^*}} 
    \]
    holds.
    \begin{proof}
        From the computation of the curvature, we see that $P$ defines a cohesive module over \\$\mathbf{A}(I, \beta, B) \hat \otimes \mathbf{A}(\tilde I, \tbeta, \tB)^{op}$ if the equations
        \begin{align*}
            &\tilde{I} = -(1-\tbeta  \beta ) I^* (1-\tbeta  \beta )^{-1} \\
            &B^{(0,2)_I} = (\tbeta  (1-\beta  \tbeta )^{-1})^{(0,2)_I} \\
            &\tilde{B}^{(0,2)_{\tilde{I}}} = (\beta  (1-\tbeta  \beta )^{-1})^{(0,2)_{\tilde{I}}}
        \end{align*}
        hold. These are equivalent to 
        \[
            \varphi\left(\betatrans{\beta } \btrans{B } L_{\Jcal_I} \right) = \betatrans{\tbeta } \btrans{\tB } L_{\Jcal_{\tilde I^*}}
        \]
        by Lemma \ref{lemma:Tdualflip}.
    \end{proof}
\end{proposition}

The non-commutative Poincaré line bundles are essential for us, as they induce equivalences:
\begin{theorem}[\cite{NCGrauert}]\label{theorem:inducedequivalences}
Let $(V/\Gamma, I, \beta, B)$ and $(V^*/\Gamma^*, \tilde I, \tbeta, \tB)$ be such that 
\[ \betatrans{\beta } \btrans{B } L_{\Jcal_I}\] and \[\betatrans{\tbeta } \btrans{\tB } L_{\Jcal_{\tilde I}}\] are T-dual via $\varphi$. Moreover, suppose that 
\[ 
1-\tbeta  \beta  \colon V^* \to V^*
\]
is invertible.
Then $\mathbf{A}(I, \beta, B)$ is cohesive Morita equivalent to $\mathbf{A}(\tilde I, \tbeta, \tB)$.
In particular, the non-commutative Poincaré line bundle induces an $A_\infty$-quasi-equivalence
\[ \mathcal P_{\mathbf{A}(I, \beta, B)} \xrightarrow{\sim} \mathcal P_{\mathbf{A}(\tilde I, \tbeta, \tB)} .\]
\begin{proof}
    In Proposition \ref{proposition:bimodules}, we showed that the existence of the non-commutative Poincaré line bundles is equivalent to T-duality 
    \[
    \varphi\left(\betatrans{\beta } \btrans{B } L_{\Jcal_I} \right) = \betatrans{\tbeta } \btrans{\tB } L_{\Jcal_{\tilde I^*}}.
    \]
    In \cite[Theorem 4.3.11]{NCGrauert} it is proved that the non-commutative Poincaré line bundles induce equivalences.
\end{proof}
\end{theorem}
\begin{remark}
    As a consequence of Theorem \ref{theorem:welldefined}, we will see that the statement of the theorem also holds if $1-\beta  \tbeta $ is singular but the equivalence bimodule may be more complicated.
\end{remark}
\begin{corollary}[{\cite[Corollary 4.3.13]{NCGrauert}}]
    Let $(V/\Gamma,I)$ be a complex torus and let $\sigma \in \Wedge^{2,0} V$. Choose a real bivector $\beta \in \Wedge^2 V$ such that $(\beta)^{2,0} = \sigma$. Then $\mathcal P_{\mathbf A(I,\beta,0)}$ does not depend on the choice of $\beta$ up to $A_\infty$-quasi-equivalence.
\end{corollary}

\section{The derived category of a generalized complex torus}\label{sec:gencat}
Recall that we defined the generalized Dolbeault algebra $\mathbf{A}(I,\beta,B)$ in Definition \ref{definition:deformeddolbeault}. 

\begin{definition}\label{definition:gencategory}
Let $\Jcal$ be a translation-invariant GC structure on $V/\Gamma$ of even cotype. Choose $I$, $\beta$ and $B$ such that 
\[
    L_{\Jcal} =\betatrans{\beta }\btrans{B } L_{\Jcal_I}.
\]
The \emph{derived category of $\Jcal$} is defined as 
\[ \Dcal(\Jcal) = \Pcal_{\mathbf{A}(I,\beta,B)}.\]
If $\Jcal$ is of odd cotype we let $\Jcal'$ be the GC structure obtained by performing integral T-duality relative to any circle and define 
\[
    \D(\Jcal) := \Dcal(\Jcal'). \qedhere
\]
\end{definition}
\begin{remark}
By construction $\Dcal(\Jcal)$ is a pre-triangulated dg-category, or equivalently a stable $\infty$-category linear over $\CC$. We may disregard some structure and consider its homotopy category as a triangulated category.
\end{remark}
We note that such $\beta, B, I$ always exist by Proposition \ref{proposition:B_then_beta_eventype}.
The goal of this section is to show that the definition does not depend on the choices. Moreover, we show that $\D(\Jcal)$ satisfies the desiderata described in the introduction.
\begin{example}\label{eg:dolbeault}
    Let $\Jcal$ be the GC structure associated to the complex torus $(X,J)$. Then 
    \[
    \Ho \D(\Jcal) = \Ho \mathcal P_{\mathbf A(J)} \cong \Dbcoh(X)
    \] by Theorem \ref{theorem:derivedcat}.
\end{example}
\begin{example}\label{eg:dolbeaultcurved}
    Let $\Jcal$ be the GC structure associated to the complex torus $(X,J)$ equipped with a gerbe $B$. Then 
    \[
    \Ho \D(\Jcal) = \Ho \mathcal P_{\mathbf A(J,0,B)} \cong \Dbcoh(X, B)
    \] by Theorem \ref{theorem:derivedcat}.
\end{example}
\begin{example}\label{eg:fukaya} 
Let $\Jcal_\omega$ be the GC structure associated to a symplectic manifold $(M,\omega)$.
As in Example \ref{example:Hyperkähler}
we choose an auxiliary Hyperkähler structure
and find 
  \begin{equation}
            \Jcal_{\omega_J} = \betatrans{-\omega_K^{-1}}\btrans{\frac{\omega _K}{2}}\Jcal_I\btrans{\frac{-\omega _K}{2}}\betatrans{\omega_K^{-1}}.\hfill
        \end{equation}
Thus the generalized Dolbeault algebra of $\Jcal_\omega$ is
    \[
    \mathbf{A}(\omega) := \mathbf{A}\left(J,-\omega_K^{-1}, \frac{1}{2}\omega_K\right)
    \]
    We will show in section \ref{sect:fukaya} that as a consequence of homological mirror symmetry $\D(\Jcal_\omega)$ is equivalent to a suitable Fukaya category of $(M, \omega)$. 
    \end{example}
\begin{remark}
    By Proposition \ref{prop:locallyproper}, $\Dcal(\Jcal)$ is a locally proper category but it does in general not have a compact generator. This is already evident for the derived category of a generic complex torus.
\end{remark}    
\subsection{Well-definedness}\label{sect:welldefined}
Let $\Jcal$ be an invariant GC structure on $V/\Gamma$, and let $\beta_1,B_1,I_1$ and $\beta_2,B_2,I_2$  be such that 
\[L_{\Jcal} = \betatrans{\beta_1 }\btrans{B_1 }L_{\Jcal_{I_1}} = \betatrans{\beta_2 }\btrans{B_2 }L_{\Jcal_{I_2}}.\] 
To show that $\Dcal(\Jcal)$ is well-defined, we will prove that there exists a presentation 
\[ \betatrans{\tbeta } \btrans{\tB } L_{\Jcal_{\tilde I}} = \varphi(L_{\Jcal}) \]
of the T-dual of $\Jcal$, such that for both $(I_1,\beta_1,B_1)$ and $(I_2,\beta_2,B_2)$ the invertibility condition of Theorem \ref{theorem:inducedequivalences} is met. Thus by composing $A_\infty$-functors, we obtain an $A_\infty$-quasi-equivalence between $\mathcal P_{\mathbf{A}(I_1,\beta_1,B_1)}$ and $\mathcal P_{\mathbf{A}(I_2,\beta_2,B_2)}$.
\begin{theorem}\label{theorem:welldefined}
    If 
    \[
    L_{\Jcal} = \betatrans{\beta_1 } \btrans{B_1 } L_{\Jcal_{I_1}} = \betatrans{\beta_2 } \btrans{B_2 } L_{\Jcal_{I_2}}
    \]
    then the algebras 
    $\mathbf{A}(I_1,\beta_1,B_1)$ and $\mathbf{A}(I_2,\beta_2,B_2)$ are cohesive Morita equivalent. In particular, $\D(\Jcal)$ is well-defined up to $A_\infty$-quasi-equivalence if $\Jcal$ is of even cotype.
    \begin{proof}
        Suppose that \[ L_{\Jcal} = \betatrans{\beta_1 }\btrans{B_1 }L_{\Jcal_{I_1}} = \betatrans{\beta_2 }\btrans{B_2 }L_{\Jcal_{I_2}}.\]
        We want to use Theorem \ref{theorem:inducedequivalences} to compare with the T-dual picture, but as we need  invertiblity assumptions we may have to perturb the presentation.
        So choose $\Delta$ as in Lemma \ref{lemma:pert} below. Since $\Delta$ is of type $(1,1)$ with respect to $I_1$, we have 
        \[L_{\Jcal} = \betatrans{\beta_1 }\btrans{B_1 + \Delta}L_{\Jcal_{I_1}}\]
        and
        \[\mathbf{A}(I_1,\beta_1,B_1) = \mathbf{A}(I_1,\beta_1,B_1 + \Delta).\]
      We write down a presentation of the T-dual of $L_{\Jcal}$ on $V^* / \Gamma^*$ 
        \[ \varphi(L_{\Jcal}) = \betatrans{\tbeta } \btrans{\tB } L_{\Jcal_{\tilde I}},\]
        and by Lemma \ref{lemma:Tdualflip} we have
        \begin{align*}
            &\tbeta  = (1+(B_1  + \Delta ) \beta_1 )^{-1} (B_1  + \Delta) \\
            &\tB  = \beta_1  (1+(B_1  + \Delta ) \beta_1 )^{-1}\\
            &\tilde I = -(1+(B_1  + \Delta ) \beta_1 )^{-1} I^* (1+(B_1  + \Delta ) \beta_1 ).
        \end{align*}
        These equations are a particular instance of the equations in Lemma \ref{lemma:Tdualflip} applied to $(I,\beta,B) := (I_1,\beta_1,B_1+\Delta)$
        noting $(1+(B_1+\Delta)\beta_1)^{-1} = 1 - \tilde \beta \beta_1$ with invertibility provided by Lemma \ref{lemma:pert}.
        Moreover, as $\Delta$ is of type $(1,1)$, the same equations hold for $(I_1,\beta_1,B_1)$, and thus they hold for the second presentation $(I_2,\beta_2,B_2)$.
        Since $\Delta$ was chosen exactly so that $1 - \beta_2  \tbeta $ and $1-\beta_1  \tbeta $ are invertible, Theorem \ref{theorem:inducedequivalences} yields the $A_\infty$-quasi-equivalences
        \[ \mathcal P_{\mathbf{A}(I_1,\beta_1,B_1+\Delta)} \simeq \mathcal P_{\mathbf{A}(\tilde I, \tbeta, \tB)} \simeq \mathcal P_{\mathbf{A}(I_2,\beta_2, B_2)}. \qedhere\]
    \end{proof}
\end{theorem}
Well-definedness for GC structures of odd cotype will be deduced in Subsection \ref{sect:integral}.

\begin{lemma}
    \label{lemma:pert}
    Let $B_1,B_2$ be two forms on $V$, $\beta_1, \beta_2$ bivectors on $V$ and $I_1,I_2$ complex structures such that 
    \[ L_{\Jcal} = \betatrans{\beta_1 }\btrans{B_1 }L_{\Jcal_{I_1}} = \betatrans{\beta_2 }\btrans{B_2 }L_{\Jcal_{I_2}}.\]
    Then there is a real two form $\Delta \in \Wedge^{(1,1)_{I_1}} V^*_{\mathbb C} \cap \Wedge^2 V^*$ of type $(1,1)$ for $I_1$, such that 
    the following endomorphisms are invertible
    \begin{align*}
        &\alpha := 1 + \beta_1  (B_1 +\Delta ) \colon V \to V\\
        &1 - \beta_2  (B_1 +\Delta ) \alpha^{-1} \colon V \to V.
    \end{align*}
    \end{lemma}
    \begin{proof}
        The invertibility of these endomorphisms is equivalent to the invertibility of $\alpha$ and
         \begin{align*}
        &Q:= 1 + (\beta_1  - \beta_2 )(B_1 +\Delta ) \colon V \to V.
        \end{align*}
        Let us focus on the invertibility of $\alpha$ first. If $\beta_1=0$ there is nothing to prove, so suppose $\beta_1\neq 0$ and define the real subspace $E := \mathrm{Im}(\beta_1 )$. Since $E$ is even-dimensional there exists $\omega \in \Wedge^{(1,1)_{I_1}} V^*_{\mathbb C} \cap \Wedge^2 V^*$ such that $\omega|_{E \times E}$ is non-degenerate by Lemma \ref{lemma:11forms}. Consider the polynomial
        \[ p(t) := \det(1 + \beta_1 (B_1  + t \omega )) = \det(1_E + \beta_1 (B_1  + t \omega )|_{E}). \]
        The leading coefficient of $p(t)$ in $t$ is
        \[  \det(\beta_1 ( \omega )|_{E}) \neq 0 \]
        by construction of $\omega$. Hence, $p(t)$ does not vanish identically. In fact $\{ t \in \mathbb R, p(t) \neq 0\} \subset \mathbb R$ is cofinite. Choose $t_0 \in \mathbb R$ with $p(t_0) \neq 0$ so that 
        \[ 1+\beta_1  (B_1  + t_0 \omega ) \] is invertible. Analogously, we can then solve for $t_1$ and $\omega_1$ such that 
        \[ 1 + (\beta_1  - \beta_2 )(B_1 + t_0 \omega  + t_1 \omega_1 )\]
        is invertible.
        As $t_1$ can be chosen arbitrarily close to $0$, we can arrange that
        \[ 1 + \beta_1  (B_1 +t_0 \omega  + t_1 \omega_1 ) \]
        is also invertible. Hence, $\Delta :=t_0 \omega + t_1 \omega_1$ is the desired two form. 
    \end{proof}
\begin{lemma}
    \label{lemma:11forms}
    Let $V$ be a real vector space with a complex structure $J \colon V \to V$ and let $E \subset V$ be a real even-dimensional subspace of $V$. Then there exists a real 2-form of type $(1,1)$, $\omega \in \Wedge^{(1,1)} V^*_{\mathbb C} \cap \Wedge^2 V^*$ 
    such that $\omega|_{E \times E}$ is non-degenerate.
    \begin{proof} 
        Let 
        \[E_J := E \cap JE\] 
        be the maximal $J$-invariant subspace of $E$. Choose a complement $C \subset E$, so that $E_J \oplus C = E$. Then $JC \cap E = 0$ and $C$ is even-dimensional. Thus, \[V = E_J \oplus C \oplus JC \oplus T\] where $T$ is a complex subspace. 
        Choose a non-degenerate real 2-form $\omega_{E_J}$ on $E_J$ of type $(1,1)$ and a real non-degenerate 2-form $\omega_C$ on $C$. Then for $(e,c,Jc',t), (\tilde e, \tilde c,J \tilde c',\tilde t) \in V$ we define
        \[ \omega((e,c,Jc',t), (\tilde e, \tilde c,J \tilde c',\tilde t)) := \omega_E(e,\tilde e) + \omega_C(c,\tilde c) + \omega_C(c',\tilde c').\]

        $\omega$ is real of type $(1,1)$ and its restriction to $E = E_J \oplus C$ is non-degenerate.
    \end{proof}
\end{lemma}

\subsection{Integral invariance}\label{sect:integral}
\begin{theorem}\label{thm:zinvariance}
   Let $\Jcal$ be a translation invariant GC structure on $T^{2n}$ of even cotype. 
   Then for any $g\in SO(2n, 2n; \Z)$ there is an $A_\infty$-quasi-equivalence 
   \[\D(\Jcal) \simeq \D(g\Jcal g^{-1}).\] 
\end{theorem}
\begin{proof}
We first note that $SO(2n, 2n; \Z)$ is generated by the matrices of the following form:
$$
\begin{pmatrix}
    P & 0 \\ 0 & P^{-T}
\end{pmatrix},
\begin{pmatrix}
    \id & \beta \\ 0 & \id
\end{pmatrix},
\begin{pmatrix}
    0 & \id \\ \id & 0
\end{pmatrix},\quad $$
where $P \in GL(2n, \Z)$ and $\beta$ is an antisymmetric integral $2n\times 2n$-matrix. This follows for example from \cite[Theorem 9.2.8]{hahn2013classical}:
These matrices generate all elementary matrices, thus generating the spinoral kernel, and 
\[P = \operatorname{diag}
(-1, 1, \dots, 1)\] 
provides an element with non-trivial spinor norm.

Thus the theorem  follows in three steps.
The first set of generators of $SO(2n,2n;\ZZ)$ correspond to base change by $P$, which induces a natural isomorphism of generalized Dolbeault algebras.
The second set of generators correspond to integral $\beta$-transforms, which do not change the generalized Dolbeault algebra by Lemma \ref{lemmma:integralBeta}.
The last generator corresponds to T-duality relative to a point, which leaves the generalized derived category invariant by Theorem \ref{theorem:inducedequivalences}. 
\end{proof}

\begin{corollary}\label{corollary:welldefinedandinvariant}
    $\D(\Jcal)$ is well-defined up to $A_\infty$-quasi-equivalence and invariant under $O(2n,2n;\Z)$.
\end{corollary}
\begin{proof}
    If $\Jcal$ is of even cotype this is Theorem \ref{theorem:welldefined}. If $\Jcal$ is of odd cotype we have to show that the choice of T-duality did not matter. But the composition of any two type-changing T-dualities is conjugation by an element of $SO(2n,2n;\Z)$, so the first result follows from Theorem \ref{thm:zinvariance}.

    Invariance under $O(2n, 2n; \Z)$ follows as T-duality relative to a circle is represented by a matrix with negative determinant, see Remark \ref{remark:ts1matrix}, generating $O(2n,2n; \Z)$.  

    Note that we formulated Theorem \ref{thm:zinvariance} for automorphisms while T-duality involves a change of vector space. 
    We will be lax about this point as it does not affect the proofs but would complicate notation.
\end{proof}
\begin{corollary}\label{cor:integralconsequences}
The generalized derived category $\Dcal(\Jcal)$ is invariant under integral $B$-transforms and integral T-duality.
\end{corollary} 
\begin{proof}
    The first statement is clear from the theorem.
    The fact that T-duality corresponds to conjugation with an element in $SO(2n,2n;\ZZ)$ is Lemma \ref{lem:Tdualintegral}.
\end{proof}

  \begin{remark}
      Note that this should not be confused with the action of $SO(2n,2n;\ZZ)$ on non-commutative tori by fractional linear transforms by Morita equivalences \cite{rieffel1999morita}.
  \end{remark}
\section{Homological mirror symmetry and relation to the Fukaya category}\label{sect:fukaya}
 
\subsection{Noncommutative mirror symmetry}
In this subsection we explore mirror symmetry and the relationship of our construction with the Fukaya category.

Let $X= (V/\Gamma, \omega + \ii B)$ a $B$-symplectic torus und let us denote by $\Fuk'(X)$ its Fukaya category over $\C$. 

There are three types of desirable statements in this context, any two of which imply the third. (Their precise meaning of course depend on what exactly one means by the Fukaya category and by homological mirror symmetry. We will return to this shortly.)

\begin{enumerate}
    \item $\Dcal(\Jcal_{\omega+\ii B})$ is equivalent to $\Fuk'(X)$
    \item $\Fuk'(X)$ satisfies homological mirror symmetry.
    \item $\Dcal(\Jcal)$ satisfies homological mirror symmetry.
\end{enumerate}

In the following subsections we will directly compare objects and hom spaces in $\Dcal(\Jcal_\omega)$ and in the Fukaya category, giving strong indication of (1).

In this subsection we prove (3), which will show (1) assuming homological mirror symmetry.

To be precise we prove that $\Dcal(\Jcal)$ satisfies a noncommutative version of homological mirror symmetry. It has long been expected that the mirror of a $B$-symplectic torus is sometimes a noncommutative complex torus. In fact it follows from a simple dimension count that the space of complex tori is not large enough to contain mirror partners for all $B$-symplectic manifolds.

By the SYZ-conjecture, mirror symmetry corresponds to T-duality with Lagrangian fibers, and by Lemma \ref{lem:Tdualintegral} this is given by integral conjugation. 
In fact, choose any splitting $\Gamma = \Gamma_1 \times \Gamma_2$. Then letting $V_i := \Gamma_i \otimes_{\mathbb Z} \mathbb R$, we may express T-duality by conjugating the GC structure $\Jcal$ with
$$S = \begin{pmatrix}
        0 & 0 & \id & 0 \\ 0 & \id & 0 & 0 \\ \id & 0 & 0 & 0 \\ 0 & 0 & 0 & \id
    \end{pmatrix} \colon (V_1 \oplus V_2) \oplus (V_1 \oplus V_2)^* \to (V_1^* \oplus V_2) \oplus (V_1^* \oplus V_2)^*$$
where the resulting generalized complex structure is defined on $\tilde V / \tilde \Gamma = V_1^*/ \Gamma_1^* \times V_2 / \Gamma_2$.

    This mirror transform for GC tori already appears in \cite[Definition 3.14]{kajiura2008noncommutative}. We now have a generalized derived category to formulate and prove a homological mirror statement:

    \begin{theorem}\label{thm:nchms}
  Given any translation invariant GC structure $\mathcal J$ on the torus let $\hat \Jcal = S\Jcal S^{-1}$ be its mirror partner.
Then \[\D(\Jcal) \simeq \Dcal(\hat \Jcal).\]
\end{theorem}
\begin{proof}
This is a direct corollary of Corollary \ref{cor:integralconsequences} as $S \in SO(2n,2n; \Z)$.
\end{proof}

Let us now consider the meaning of this theorem.

  As above $X$ is a $2n$-torus equipped with a (translation invariant) $B$-symplectic structure $\omega + \ii B$. 
  There is an $A_\infty$-category $\Fuk_{\mathrm{aff}}(X)$ (over $\CC$) whose objects are built out of twisted complexes on affine Lagrangians $L$ equipped with unitary connections with curvature $2\pi \ii B|_L$. 
  The hom spaces are given by Floer cohomology, for further details, in particular the composition rules, we refer to  \cite{Fukaya02}.

  If now $V$ contains a pair of $L, L'$ of Lagrangians (i.e. $\omega + \ii B$ vanishes on $L$ and $L'$) s.t. $\Gamma = (\Gamma \cap L) \oplus (\Gamma \cap L')$ (see \cite[Assumption 1.11]{Fukaya02}) we may write $V/\Gamma$ as an affine Lagrangian fibration with an affine Lagrangian section, with base $L/(\Gamma \cap L)$ and fiber $L'/(\Gamma \cap L')$.
  Then dualizing the fibers we obtain a complex torus $\check X$, which is the mirror of $X$. 
  This is the most basic case of the SYZ-conjecture \cite{Strominger}.
  Fukaya proves that $\Fuk_{\mathrm{aff}}(X)$ fully faithfully embeds in the bounded derived category of coherent sheaves on $\check X$ \cite{Fukaya02}.
  
However, the Fukaya category over $\CC$ may contain additional objects, so-called \emph{co-isotropic branes}. The most concrete example are space-filling branes, which are expected to correspond to certain coherent sheaves on special abelian varieties that do not appear to be in the image of Fukaya's functor \cite{Kapustin03}. (The prime example is the graph of complex multiplication on the square of an elliptic curve with complex multiplication.) 
As far as we understand it remains possible that they appear in a non-obvious way in the split-closure of Fukaya's category of affine Lagrangians. 

The hom spaces between co-isotropic branes are not well-understood in general.
 \begin{remark}\label{rem:novikov}
      The reason to restrict attention to affine Lagrangians is that the infinite series needed to define the Fukaya category do not in general converge over $\CC$. 
      We may thus think of $\Fuk_{\mathrm{aff}}(X)$ as a \emph{convergent Fukaya category}.
      The true Fukaya category of $X$ should be defined over the Novikov ring, where its objects are generated by all Lagrangians.    
      Then the mirror is a rigid analytic variety over the Novikov ring, and the full statement of homological mirror symmetry was proven in the case of the standard symplectic 4-torus \cite{Abouzaid10a} and is expected to hold in general -- without any addition of co-isotropic branes.
  \end{remark}

\begin{theorem}\label{thm:hms}
  Let  $X=(V/\Gamma,\omega + \ii B)$ be a $B$-symplectic torus admitting the structure of an affine Lagrangian fibration with an affine Lagrangian section and let $\Jcal_{\omega + \ii B} = \btrans{B}\Jcal_{\omega}\btrans{-B}$ be the associated GC structure. Let $\check X$ be the complex torus mirror to $X$. Then \[\Ho \D(\Jcal_{\omega+\ii B}) \cong \Dbcoh(\check X).\]
\end{theorem}
Note that we are not assuming $\check X$ is  a principally polarized abelian variety here.
\begin{proof}
The existence of the affine Lagrangian fibration with an affine Lagrangian section may be re-written as the existence of a Lagrangian splitting $V = L \oplus M$ with $L$ and $M$ rational. Then $\omega + \ii B$ is block off-diagonal and conjugation with $S$ turns $\Jcal_{\omega + \ii B}$ into a purely complex GC structure $\Jcal_J$ (on $L^* / \Gamma^*_L \times B / \Gamma_B$).
The mirror $\check X$ of $(V/\Gamma,\omega + \ii B)$ is given by equipping $L^* / \Gamma^*_L \times B / \Gamma_B$ with the complex structure $J$.

By Theorem \ref{thm:nchms} we have $\D(\Jcal_{\omega + \ii B}) \cong \D(\Jcal_J)$  and Example \ref{eg:dolbeault} shows that $\D(\Jcal_J) \cong \Dbcoh(\check X)$.
\end{proof}

Let us now denote by $\Fuk'(X)$ a putative convergent Fukaya category containing Fukayas $\Fuk_{\mathrm{aff}}(X)$ as well as any co-isotropic branes needed to make homological mirror symmetry true.

\begin{conjecture}[Homological mirror symmetry]\label{conj:hms}
    Let $X$ be the $2n$-torus equipped with a translation-invariant $B$-symplectic structure $\omega + \ii B$ admitting the structure of a Lagrangian fibration together with a Lagrangian splitting.
    Then \[\Fuk'(X) \cong \Dbcoh(\check X).\]
\end{conjecture}

\begin{corollary}\label{cor:fukaya}
    Let $X$ be the $2n$-torus equipped with a translation-invariant $B$-symplectic structure $\omega + \ii B$ admitting the structure of a Lagrangian fibration with a Lagrangian splitting.
    If homological mirror symmetry holds then $\Dcal(\Jcal_{\omega + \ii B}) \simeq \Fuk'(X)$. 
\end{corollary}

If no Lagrangian fibration and splitting exist, the generalized complex mirror is no longer purely complex.
\begin{corollary}\label{cor:nchms}
  Let  $X$ be a $B$-symplectic torus admitting an affine Lagrangian and let $\Jcal_{\omega + \ii B}$ be the associated GC structure. 
  Then the mirror is a noncommutative torus.
  Moreover there is a complex torus $\check X$ with a 2-form $B'$ such that  $\D(\Jcal_{\omega+\ii B})$ is equivalent to the derived category of coherent $B'$-twisted sheaves on $\check X$.
\end{corollary}
\begin{proof}
This is another corollary of Theorem \ref{thm:nchms}.
If there is any affine Lagrangian in $X$ we may choose a Lagrangian lift to $V$ and a possibly non-Lagrangian complement $M$ with $\Gamma = (\Gamma \cap L) \oplus (\Gamma \cap M)$. 
With respect to the splitting $V = L \oplus M$, $S$ sends $\Jcal_{\omega + \ii B}$ to the GC structure of a non-commutative torus as a mirror. 

 The computation is that if $\omega + \ii B$ admits an affine Lagrangian it may be written so that the top left quadrant of its matrix representation is $0$

and the schematic computation
$$
S \Jcal_{\omega+\ii B}S^{-1} = S \begin{pmatrix}
    0 & 0 & 0 & * \\
    0 & 0 & * & * \\
    * & * & 0 & 0 \\
    *& 0 & 0 & 0
\end{pmatrix} S^{-1} =
\begin{pmatrix}
    0 & * & * & 0 \\
    * & 0 & 0 & * \\
    0 & 0 & 0 & * \\
    0& 0 & * & 0
\end{pmatrix} 
$$
gives the $\beta$-transform of a purely complex torus, i.e.\ a noncommutative torus. 

Switching the order of the basis leads to a $B'$-transform of a complex torus, thus the Fukaya category agrees with the derived category of twisted sheaves, see Example \ref{eg:dolbeaultcurved}.
\end{proof}

If there is no affine Lagrangian at all then regardless of choice of basis any $S$-conjugate of the associated GC structure is also $B$-symplectic, i.e.\ it is a noncommutative deformation of a curved deformation of a complex torus.
In this case there are no affine Lagrangians and we do not expect any cohesive modules over a curved noncommutative Dolbeault algebra, thus we expect that this is just an equivalence of trivial categories.

We give a precise formulation of the noncommutatiev homological mirror symmetry conjecture:
\begin{conjecture}[Noncommutative homological mirror symmetry]\label{conj:nchms}
    Let $X$ be the $2n$-torus equipped with a translation-invariant $B$-symplectic structure $\omega + \ii B$
    Then we have \[\Fuk'(X) \cong \D(S \Jcal_{\omega + \ii B} S^{-1}).\]
\end{conjecture}
     We now turn to explicit objects in $\D(\Jcal_{\omega})$ (setting $B = 0$ for simplicity). 
   Before investigating Lagrangians in the next section we consider space-filling branes.
     
   \begin{example}\label{eg:sfb}
   Let  $(V/\Gamma, \omega)$ have a space-filling brane, i.e.\ a line bundle equipped with a connection whose curvature $F$ satisfies 
   \[F\omega^{-1}F + \omega  = 0,\] see \cite{Kapustin03} and \cite{kapustin2005branes}.
    Then one computes 
    \[\btrans{F}L_{\Jcal_\omega} = \betatrans{\frac{1}{2} (F)^{-1}}L_{\Jcal_J}\] 
    with 
    \[J = \omega^{-1}F\] 
    and by integral invariance (Theorem \ref{thm:zinvariance}) we have an equivalence 
    \[ \mathcal D(\mathcal J_{\omega}) \cong \mathcal P_{\mathbf A(J, \frac{1}{2} F^{-1}, 0)}.\]

    As this construction removed the curvature there is a free rank 1 module over ${\mathbf A(J, \frac{1}{2} F^{-1}, 0)}$, which we can think of as representing the space-filling brane as an object in $\Fuk(X, \omega)$.
    The derived endomorphism space is ${\mathbf A(J, \frac{1}{2} F^{-1}, 0)}$.

    This agrees with the predictions in the mathematical physics literature.
    There the endomorphism space represents the algebra $\mathcal A$ of $(\mathcal B_{cc}, \mathcal B_{cc})$-strings, with $\mathcal B_{cc}$ standing for the space-filling brane, see \cite[Section 1.5]{Gukov:2008ve}.
    With notation as above there is a holomorphic symplectic structure $\Omega = F + \ii \omega$ and $\mathcal A$ is expected to be the deformation quantization of the Dolbeault algebra controlled by the inverse holomorphic bivector $\Omega^{-1}$ (where we view $\Omega$ as a map from holomorphic tangent to holomorphic cotangent space), see \cite[Section 2.2]{Gukov:2008ve} and \cite[Section 11.1]{kapustin2007electric}.
    But as in \cite[Section 3.2]{kapustin2005branes} one finds 
    $\Omega^{-1} = \frac 1 4(F^{-1} - \ii \omega^{-1})$ and this is exactly the $(2,0)$-part of the real bivector  $\frac 1 2 F^{-1}$. 
    Thus the deformation on holomorphic functions induced by $\frac 1 2 F^{-1}$ is the one determined by $\Omega^{-1}$, as expected.
   \end{example}

\subsection{Lagrangians and strict deformation quantization}  
\label{section:floerdef}
In this section we will relate the Fukaya category of a symplectic torus $(V/\Gamma,\omega)$ of real dimension $4n$ with our generalized derived category.
This gives a description of the Fukaya category in terms of strict deformation quantization, without assuming the homological mirror symmetry conjecture. 
Hence we establish an alternative approach to proving mirror symmetry for tori in view of Theorem \ref{thm:zinvariance}. 
In fact, our computations are more general than those contained in \cite{Fukaya02} since we do not assume that the mirror is complex (it could be non-commutative or gerby).
Concretely, we show explicitly how Lagrangians give rise to objects in the deformed categories and show that Floer cohomology agrees with the spaces of homomorphisms in the generalized derived category, in particular recovering the Maslov index. 
Higher compositions will be considered in future work. For the definition of the Fukaya category (over $\CC$) we refer to \cite{Fukaya02}.

 We begin by constructing cohesive modules associated to Lagrangians.
\begin{definition}
    An \emph{affine Lagrangian object} in $(V/\Gamma, \omega)$ is an embedded affine Lagrangian torus in $V$ equipped with a unitary local system of rank 1. It is thus determined by a triple $(L,b,\alpha)$ where 
    \begin{enumerate}
        \item $L \subset V$ is a linear Lagrangian subspace of $V$ which is rational with respect to $\Gamma$, i.e.,
        \[ \Gamma_L := L \cap \Gamma \cong \mathbb Z^{2n}. \]
        \item $b \in V/L$ is a base vector,
        \item $\alpha \in L^*$ is a real 1-form along $L$.
    \end{enumerate}
    The image of $L+b$ in $V/\Gamma$ is an embedded Lagrangian torus which will be denoted as $L_b$. 
\end{definition}
Any constant Hyperkähler structure $g,I,J,K$ on $V/\Gamma$ with $\omega$ given by $\omega  = \omega_J  = g  J$ yields a model for the derived category $\D(\Jcal_{\omega})$. As in Example \ref{eg:fukaya}
\[
\D(\Jcal_{\omega}) = \mathcal P_{\mathbf A(I,-\omega_K^{-1},\omega_K/2)}.
\]
Let $(L,b,\alpha)$ be an affine Lagrangian object with respect to $\omega = \omega_J$ which is totally real for $I$. 
We have the following relations:
\begin{align}
    V = L \oplus IL,\quad \omega_K  = \begin{pmatrix}
        \omega_K|_{L \times L} & 0 \\ 0 & \omega_K|_{IL \times IL}
    \end{pmatrix} \colon L \oplus IL \to L^* \oplus (IL)^*,
\end{align}
In particular, $\omega_K|_{IL \times IL}$ is a symplectic form on $IL$. We will also use the notation $\omega_K$ for the restriction to $IL$. We choose complementary $\omega_K$-Lagrangian subspaces of $IL$:
\[ IL =  L_K' \oplus L_K \]
(these need not be rational, in fact $IL$ may not even be rational).
Define the following finitely-generated projective $\mathbf A(I,-\omega_K^{-1},\omega_K/2)^0 = C^\infty(V/\Gamma, -\omega_K^{-1})$-module 
\[ N^0(L) := C^\infty(L / \Gamma_L) \hat \otimes \mathcal 
S(L_K),\]
where $\Gamma_L = \Gamma \cap L$. The action is given by
\begin{align} \label{equation:LagAct} (u_{\gamma}.g)(l,v) = u_{\gamma}(l + v - \frac{1}{2} \omega_K^{-1}(\gamma|_{L_K'} )) g(l - \frac{1}{2} \omega_K^{-1}(\gamma|_L), v - \omega_K^{-1} (\gamma|_{L_K'}) ).\end{align}
(Note that the action does neither depend on $b$ nor on $\alpha$). 
It is finitely generated projective by \cite[Proposition 3.2]{Rieffel98}. The module $N^0(L)$, up to isomorphism, does not depend on the choice of splitting $IL = L_K' \oplus L_K$.

\begin{lemma}
    \label{lemma:projFlat}
    $b$ and $\alpha$ determine a projectively flat connection of curvature
    \[ 2\pi \ii \omega_K(L) := 2\pi \ii \begin{pmatrix}
        0 & 0 \\ 0 & \omega_K|_{IL \times IL}
    \end{pmatrix} \colon L \oplus IL \to L^* \oplus (IL)^* \]
    on $N^0(L)$ and $\omega_K(L)^{(2,0)_I + (0,2)_I} = \omega_K / 2$. 
\end{lemma}
    \begin{proof}
        Define $\nabla \colon (L \oplus L_K' \oplus L_{K}) \times  N^0(L) \to N^0(L) $ as 
        \[ \nabla_{(t,l_k',l_k)} g(l,v) := \partial_{(t,l_k')}g(l,v) + 2 \pi \ii [\omega_K(l_k+l_k',v+b)+\alpha(t)] g(l,v).\]
        One verifies the Leibniz rule and that the curvature is indeed given by $2\pi \ii \omega_K(L)$. For the final statement we note that $\omega_K(L)$ differs from $\frac{1}{2}\omega_K $ by $\frac{1}{2}\begin{pmatrix}
        -\omega_K|_{L \times L}& 0 \\ 0 & \omega_K|_{IL \times IL}
    \end{pmatrix}$
    which has type $(1,1)$ with respect to $I$.
    \end{proof}

\begin{definition}
    We call the cohesive module, associated to an affine Lagrangian object $(L,b,\alpha)$ obtained by endowing $N^0(L)$ with the $(0,1)$-part of the connection constructed in the proof of Lemma \ref{lemma:projFlat}, the \emph{Lagrangian module} and denote it by 
    \[ N(L,b,\alpha) = (N^0(L) \otimes \wedge^{(0,\bullet)} V_{\mathbb C}^*, \nabla^{(0,1)} ).\]
    Note that $\nabla$ depends on $b$ and $\alpha$.
\end{definition}
We already note that the definition of the Lagrangian module is only canonical up to a shift and we will consider suitable shifts for comparison with the literature later.

\begin{remark}
     One can also explicitly construct modules associated to Lagrangians which are complex with respect to $I$. 
     We will focus on the totally real Lagrangians as this condition is open and leads to more accessible computations.

     However, as the construction is of some geometric interest we now describe modules associated to complex Lagrangians. 

     Assume we have $(X, \omega)$ with $\omega = \omega_J$ as part of a Hyperkähler structure $(g, I, J, K)$.
    Take an $I$-complex Lagrangian $L$ on $(X, \omega)$. 
    It is also $\omega_K$-Lagrangian and thus functions on $L$ can be used as a module over $(C^\infty(T^{2n},{\omega_K^{-1}})$ 
    As $L$ is $I$-complex we obtain an action of $(T^*)^{0,1}$
    and as $\omega_K$ vanishes this gives rise to a curved module $N^0(L)$ for $\mathbf{A}(\omega)$.
    This module is not cohesive, but it admits a cohesive resolution.

    The intuition here is that a 
    non-commutative module generically needs to be coisotropic and a curved module generically needs to be isotropic, so a Lagrangian can be both.
\end{remark}
\subsection{Floer cohomology}
Consider now two affine Lagrangian objects $(L,b,\alpha), (\tilde L, \tilde b, \tilde \alpha)$.  
We remark that one can always find a Hyperkähler structure $g,I,J,K$ for which both $L$ and $\tilde L$ are totally real and we fix one.
Hence we obtain the cohesive modules $N(L,b,\alpha)$ and $N(\tilde L, \tilde b, \tilde \alpha)$ and we wish to compute the cohomology of 
\[ Q^\bullet := \Hom^\bullet_{\mathbf A(I,-\omega_K^{-1},\omega_K/2)}(N(L,b,\alpha),N(\tilde L, \tilde b, \tilde \alpha))\]
and compare it with the Floer complex
\[ CF^\bullet( (L,b,\alpha), (\tilde L, \tilde b, \tilde \alpha) ).\]
By the work of Fukaya \cite{Fukaya02}, the Floer cohomology is given by
\begin{equation} \label{eq:fukfloer} HF^\bullet((L,b,\alpha),(\tilde L, \tilde b, \tilde \alpha)) = H^{\bullet-\eta_{\mathrm{Fuk}}(L,\tilde L)}(L_b \cap \tilde L_{\tilde b},\mathcal L_{\tilde \alpha - \alpha}),\end{equation}
where $\eta_{\mathrm{Fuk}}(L, \tilde L)$ is the Maslov index, as used in \cite{FukayaPre}. We review it in Definition \ref{def:fukmaslov}.
\begin{remark}
    \label{remark:Maslovchoice}
    Recall that the choice of a Maslov index $\eta(L_1,L_2)$ for a pair of Lagrangians depends on additional choices. Typically one chooses a lift of $L$ to the universal cover of the Lagrangian Grassmannian. Another way of obtaining a $\mathbb Z$-grading is explained in \cite{FukayaPre} which requires the choice of an auxiliary Lagrangian subspace $L_{\mathrm{pt}}$ (see \ref{def:fukmaslov}). It turns out that the Hyperkähler structure chosen here also yields a notion of Maslov index $\eta_{\mathrm{HK}}(L_0,L_1)$, see Definition \ref{Def:HK2}. We will show in Corollary \ref{cor:compMaslov} that there is a map $L \mapsto \tau(L) \in \mathbb Z$ such that 
    \[ \eta_{\mathrm{HK}}(L_1,L_2) = \eta_{\mathrm{Fuk}}(L_1,L_2) + \tau(L_1) - \tau(L_2). \]
    This relationship means that the choices $\eta_{\mathrm{Fuk}}$ and $\eta_{\mathrm{HK}}$ are compatible: If we shift each Lagrangian $L$ in the derived category by $\tau(L)$ (see Definition \ref{def:correctedLagrange}), then the gradings induced by $\eta_{\mathrm{HK}}$ and $\eta_{\mathrm{Fuk}}$ coincide. 
\end{remark}

As $N(L,b,\alpha)$ is cohesive, $Q^\bullet$ is isomorphic to
\[ N(L,b,\alpha)^\vee \hat \otimes_{\mathbf A(I,-\omega_K^{-1},\omega_K/2)} N(\tilde L, \tilde b, \tilde \alpha) \]
where $N(L,b,\alpha)^\vee$ is the $\mathbf A(I,-\omega_K^{-1},\omega_K/2)^{op} = \mathbf A(I,\omega_K^{-1},-\omega_K/2)$-module 
    \[\Hom_{\mathbf A(I,-\omega_K^{-1},\omega_K/2)}(N(L,b,\alpha), \mathbf A(I,-\omega_K^{-1},\omega_K/2)). \]
In fact, it is not hard to see that
\begin{proposition}
    There is an isomorphism
    \[ N(L,b,\alpha)^\vee \cong N(L,-b,-\alpha) \]
    of $\mathbf A(I,\omega_K^{-1},-\omega_K/2)$-modules.
    \begin{proof}[Proof sketch]
        Consider the subalgebra $\mathcal S(\Lambda_L^*, -\omega_K^{-1}) \subset \mathcal S(\Gamma^*, -\omega_K^{-1}) = C^\infty( V/\Gamma, -\omega_K^{-1})$. We have that
        \[ S(\Gamma^*, -\omega_K^{-1}) \hat \otimes_{\mathcal S(\Lambda_L^*, -\omega_K^{-1})} H(-\omega_K^{-1}) = N(L) \quad( = (N(L,b,\alpha))^0 ) \]
        where $H(-\omega_K^{-1})$ is the Heisenberg-module constructed in \ref{section:Heisenberg}. The statement of the lemma then essentially follows from the following claim: \emph{For any non-commutative torus $V/\Gamma$ on an invertible bivector $\sigma$, there is an isomorphism}
        \[  \Hom_{C^\infty(V/\Gamma,\sigma)}(H(\sigma), C^\infty(V/\Gamma,\sigma)) = H(\sigma)^\vee \cong H(-\sigma) \]
        \emph{of $C^\infty(V/\Gamma,-\sigma)$-modules.}
        To prove the claim, choose a Lagrangian splitting $V = L \oplus L'$. Then
            \begin{align*}
                &H(\sigma) = \mathcal S(L), \quad (u_\gamma.f)(l) = u_\gamma(l + \frac{1}{2} \sigma(\gamma|_{L'})) f(l+\sigma(\gamma|_{L'}))
                \\&H(-\sigma) = \mathcal S(L), \quad (u_\gamma.g)(l) = u_\gamma(l - \frac{1}{2} \sigma(\gamma|_{L'})) g(l-\sigma(\gamma|_{L'})).
            \end{align*}
            We define the isomorphism $\Phi \colon H(-\sigma) \to H(\sigma)^\vee$ as:
            \[ \Phi \colon f \mapsto \left(g \mapsto \sum_{\gamma \in \Gamma^*} \int_L f(l) (u_{-\gamma}.g)(l)  \cdot u_\gamma\right) \]
            (which depends on the choice of a volume form on $L$).
            For every $\gamma$ we have the following identity:
            \[ \int_L (u_{\gamma}.f)(l) g(l) = \int_L f(l) (u_{\gamma}.g)(l). \]
            It follows from this formula that $\Phi$ is $C^\infty(V/\Gamma,-\sigma)$-linear in $f$ and $C^\infty(V/\Gamma,\sigma)$-linear in $g$.
  
        The proof of the lemma is completed by comparing connections.
    \end{proof}
\end{proposition}
Thus,
\begin{align} \nonumber Q^\bullet &\cong [N(L,-b,-\alpha) \hat \otimes_{\mathbb C} N(\tilde L, \tilde b, \tilde \alpha) ] / \sim \\
&\label{eq:quotient} \cong N(L \times \tilde L, (-b,\tilde b), -\alpha \otimes 1 + 1 \otimes \tilde \alpha) / \sim
\end{align}
where $\sim$ is generated by the relations
\[f \otimes g \sim u_{-\gamma}.f \otimes u_{\gamma}.g \text{ and } f \otimes \alpha \wedge g \sim f \wedge \alpha \otimes g, \alpha \in (V^{(0,1)})^{*} \]
and the quotient is understood to be completed.

To understand better the quotient \eqref{eq:quotient}, it is convenient to choose a particular realization of the cohesive module $N(L \times \tilde L, (-b,\tilde b), -\alpha \otimes 1 + 1 \otimes \tilde \alpha)$. According to the previous subsection, we have the natural splitting
\begin{equation}\label{eq:VVsplitting} V \oplus V = (L \oplus \tilde L) \oplus (IL \oplus I\tilde L) \end{equation}
into $(I \times I)$-totally real, $\omega_J$-Lagrangian subspaces. On $V \times V$ we need to consider the symplectic form \[ \omega' := -\omega_K \otimes 1 + 1 \otimes \omega_K\]
which is block-diagonal with respect to splitting \eqref{eq:VVsplitting}.
Let 
\[(IL \cap I \tilde L)^+ := \{ (Il,Il) \in IL \oplus I \tilde L ~:~ l \in L \cap \tilde L \}.\] Clearly, \[(IL \cap I \tilde L)^+ \subset (IL \oplus I \tilde L, \omega') \]
is isotropic. The symplectic form $\omega'$ descends to a symplectic form on the subquotient \[ Z := ((IL \cap I\tilde L)^+)^\perp / (IL \cap I \tilde L)^+. \] 
Choose a symplectic basis $e_1,...,e_{2n},f_1,...,f_{2n}$ of $(IL \oplus I \tilde L,\omega')$ such that $(IL \cap I\tilde L)^+ = \langle e_1,...,e_r \rangle$. Then define
\[C_1 := \langle e_{r+1},...,e_{2n} \rangle, \quad C_2 := \langle f_{r+1},...,f_{2n} \rangle \quad  B := \langle f_1,...,f_r \rangle.
\]
Thus \[(C_1 \oplus C_2, \omega'|_{C_1 \oplus C_2}) \cong Z, \quad B \cong (IL \oplus I \tilde L)/(IL \cap I\tilde L)^\perp\] and we have the Lagrangian decomposition 
\[ (IL \oplus I\tilde L, \omega') = [ (IL \cap I \tilde L)^+ \oplus C_1 ] \oplus [ B \oplus C_2].\]
We can therefore choose $N(L \times \tilde L, (-b,\tilde b), -\alpha \otimes 1 + 1 \otimes \tilde \alpha)$ to have the form
\begin{align*} N^0(L \times \tilde L) &= \mathcal S(\Gamma^*_L \times \Gamma^*_{\tilde L}) \hat \otimes_{\mathbb C} \mathcal S( (IL \cap I \tilde L)^+ \oplus C_1) \\
u_{(\gamma,\tilde \gamma)}. (u_{\gamma_L}  \otimes u_{\gamma_{\tilde L}} \otimes g)(v) &= \exp{(\pi \ii (\omega')^{-1}[(\gamma, \tilde \gamma),(\gamma,\tilde \gamma)|_{B \oplus C_2}-(\gamma_L, \gamma_{\tilde L})  ])} \cdot \\
&\quad\quad \cdot u_{\gamma_L + \gamma|_{\Gamma_L}} \otimes u_{\gamma_{\tilde L} + \tilde \gamma|_{\Gamma_{\tilde L}}} \cdot g(v-(\omega')^{-1}((\gamma,\tilde \gamma)|_{B \oplus C_2}))\\
 N(L \times \tilde L, (-b,\tilde b), -\alpha \otimes 1 + 1 \otimes \tilde \alpha) &= (N^0(L \times \tilde L) \otimes \Wedge^{(0,\bullet)} (V^* \oplus V^*), \nabla^{(0,1)})
 \end{align*}
with connection data
\begin{align}
 \label{eq:connection} \nabla \colon V \oplus V \times N^0(L \times \tilde L) &\to N^0(L \times \tilde L) \\\nonumber
 \nabla_{(u,\tilde u)} (u_{\gamma_L}  \otimes u_{\gamma_{\tilde L}} \otimes g)(v) &= \partial_{\pi_L(u)}(u_{\gamma_L}) \otimes u_{\gamma_{\tilde L}} \cdot  g(v) + u_{\gamma_L} \otimes \partial_{\pi_{\tilde L}(\tilde u)} (u_{\gamma_{\tilde L}}) \cdot g(v)  \\\nonumber
& \quad + u_{\gamma_L} \otimes u_{\gamma_{\tilde L}} \otimes [\partial_{\pi_{(IL \cap I \tilde L)^+ \oplus C_1 }(u,\tilde u) } g(v) + 2\pi \ii (\omega'((u,\tilde u),v+\tilde b - b) 
\\\nonumber&\quad + (\tilde \alpha - \alpha)(\pi_{L \oplus \tilde L}(u,\tilde u)) )g(v) ].
\end{align}
Consider the map 
\[\Gamma^* \to \Gamma^*_L \times \Gamma^*_{\tilde L},\quad\gamma \mapsto (-\gamma|_L, \gamma|_{\tilde L}).\] Its kernel is $\Lambda_L \cap \Lambda_{\tilde L}$, where $\Lambda_L$ denotes the kernel of the restriction map $\Gamma^* \to \Gamma^*_L$. Its cokernel fits into a short exact sequence
\[
\begin{tikzcd}
0 \arrow[r] & T \arrow[r] & (\Gamma^*_L \times \Gamma^*_{\tilde L})/ \Gamma^* \arrow[r] & (\Gamma_L \cap \Gamma_{\tilde L})^* \arrow[r] & 0,
\end{tikzcd} \]
which we can split as $\Gamma_L \cap \Gamma_{\tilde L}$ is free abelian so that
\begin{equation}
    \label{eq:splittorsion}
     (\Gamma^*_L \times \Gamma^*_{\tilde L})/ \Gamma^* \cong (\Gamma_L \cap \Gamma_{\tilde L})^* \times T, 
     \end{equation}
where $T$ is the finite torsion subgroup of $(\Gamma^*_L \times \Gamma^*_{\tilde L})/ \Gamma^*$. The surjection $\Gamma^*_L \times \Gamma^*_{\tilde L} \to (\Gamma_L \cap \Gamma_{\tilde L})^* \times T$ admits a \underline{set-theoretic} splitting $S$, linear only in the first argument: 
\[
\begin{tikzcd}
0 \arrow[r] & \Lambda_L \cap \Lambda_{\tilde L} \arrow[r] & \Gamma^* \arrow[r] & \Gamma_L^* \times \Gamma_{\tilde L}^* \arrow[r] & (\Gamma_L^* \times \Gamma_{\tilde L}^*) / \Gamma^* \arrow[r] \arrow[d, "\cong"] & 0 \\
            &                                             &                    &                                                 & (\Gamma_L \cap \Gamma_{\tilde L})^* \times T \arrow[lu, "S", dashed, red]        &  
\end{tikzcd} \]
The splitting $S$ determines a split inclusion between Schwartz spaces (as it is linear in the first argument)
\[ \mathcal S(S) \colon \mathcal S((\Gamma_L \cap \Gamma_{\tilde L})^* \times T) \hookrightarrow \mathcal S(\Gamma^*_L \times \Gamma^*_{\tilde L}), u_{\gamma} \mapsto u_{S(\gamma)}.\]

\begin{lemma}
    \label{lemma:LemmaIdent}
    We have an isomorphism
    \[ N^0(L \times \tilde L) / \Gamma^* \cong \mathcal S((\Gamma^*_L \times \Gamma^*_{\tilde L})/ \Gamma^* \times (\Lambda_L \cap \Lambda_{\tilde L})^* \times C_1). \]
    \begin{proof}
        Consider the intermediate quotient
        \[ \mathrm{id} \otimes p \otimes \mathrm{id} \colon  N^0(L \times \tilde L) \to N^0(L \times \tilde L) / (\Lambda_L \cap \Lambda_{\tilde L}) = \mathcal S(\Gamma^*_L \times \Gamma^*_{\tilde L}) \hat \otimes \mathcal S((IL \cap I\tilde L)^+) / (\Lambda_L \cap \Lambda_{\tilde L}) \hat \otimes \mathcal S(C_1), \]
        where 
        \[p\colon\mathcal S((IL \cap I\tilde L)^+) \to \mathcal S((IL \cap I\tilde L)^+) / (\Lambda_L \cap \Lambda_{\tilde L})\]  is the quotient map and  $\Lambda_L \cap \Lambda_{\tilde L}$ acts as
        \[ \gamma.g(Il,Il) := u_{\gamma}.g(Il,Il) = g(Il-\omega_K^{-1}(\gamma),Il-\omega_K^{-1}(\gamma)). \]
        We identify
        \[ \mathcal S((IL \cap I\tilde L)^+) / (\Lambda_L \cap \Lambda_{\tilde L}) \cong C^\infty(IL \cap I\tilde L / \omega_K^{-1}(\Lambda_L \cap \Lambda_{\tilde L})) \cong \mathcal S((\Lambda_L \cap \Lambda_{\tilde L})^*) \]
        by $g \mapsto (Il \mapsto \sum_{\lambda \in \omega_K^{-1}(\Lambda_L \cap \Lambda_{\tilde L})} g((Il+\lambda),(Il+\lambda)))$.
        We have \[N^0(L \times \tilde L) / \Gamma^* = [N^0(L \times \tilde L) / (\Lambda_L \cap \Lambda_{\tilde L})] / (\Gamma^* / \Lambda_L \cap \Lambda_{\tilde L}).\]
        Consider the subspace 
        \[  \mathcal S(S) \otimes \mathrm{id} \colon \mathcal S((\Gamma^*_L \times \Gamma^*_{\tilde L})/\Gamma^*) \hat \otimes \mathcal S((\Lambda_L \cap \Lambda_{\tilde L})^*) \hat \otimes \mathcal S(C_1) \hookrightarrow N^0(L \times \tilde L) / (\Lambda_L \cap \Lambda_{\tilde L}).\]
        The composition of $\mathcal S(S) \otimes \mathrm{id}$ with the projection onto $N^0(L \times \tilde L) / \Gamma^*$ is the desired isomorphism.
    \end{proof}
\end{lemma}

\begin{lemma}
    \label{lemma:transport}
    There is a canonical isomorphism
    \[ \Psi \colon V/(L \cap \tilde L \oplus IL \cap I \tilde L) \xrightarrow{\cong} Z = ((IL \cap I\tilde L)^+)^\perp / (IL \cap I \tilde L)^+,\quad v \mapsto (\pi_{IL}(v),\pi_{I \tilde L}(v)). \]
    Hence, $Z$ naturally carries a complex structure induced by $I$. Moreover, the symplectic form $\omega'$ on $Z$ is of type $(1,1)$ for $I$.
    \begin{proof}
        The map $V \ni v \mapsto (\pi_{IL}(v), \pi_{I\tilde L}(v)) \in IL \oplus I\tilde L$ factors through the subspace $((IL \cap I\tilde L)^+)^\perp$ because $\omega_K|_{IL \times L} = \omega_K|_{I\tilde L \times \tilde L} = 0$. Its kernel is $L \cap \tilde L$, and the preimage of $(IL \cap I\tilde L)^+$ is $IL \cap I \tilde L$. This shows that $\Psi$ is well-defined and injective. For dimensional reasons it is an isomorphism. The form $\Psi^* \omega'$ is given by
        \[ (\Psi^* \omega')(v,v') = \omega_K((\pi_{I \tilde L}-\pi_{IL})(v),v'). \]
        Thus,
        \begin{align*}
            (\Psi^* \omega')(Iv,Iv') &= \omega_K((\pi_{I \tilde L}-\pi_{IL})(Iv),Iv')\\
            &= \omega_K(I(\pi_{\tilde L}-\pi_{L})(v),Iv') \\
            &= -\omega_K(I(\pi_{I\tilde L}-\pi_{IL})(v),Iv') \\
            &= \omega_K((\pi_{I\tilde L}-\pi_{IL})(v),v') =  (\Psi^* \omega')(v,v')
        \end{align*}
        as $\pi_{I\tilde L}-\pi_{IL} = \pi_{L} - \pi_{\tilde L}$ and since $\omega_K$ is of type $(2,0)+(0,2)$ with respect to $I$.
    \end{proof}
\end{lemma}

\begin{definition}
    \label{Def:HK2}
    Let $L_1,L_2$ be $\omega_J$ Lagrangian subspaces which are totally real with respect to $I$. Hence $V = L_1 \oplus IL_1 = L_2 \oplus IL_2$. We define on $V$ the quadratic form 
    \[ t(v) = \omega_J( (\pi_{L_2} - \pi_{L_1})(v),v).\]
    Define the \emph{Hyperkähler Maslov index}
    \[ \eta_{\mathrm{HK}}(L_1,L_2) = \frac{1}{2} n^-(t)  \]
    where $n^-(t)$ denotes the negative index of $t$, i.e., the number of negative eigenvalues.
\end{definition}
\begin{lemma}
    \label{lemma:Maslov}
    Let $\omega'$ and $I$ be as in Lemma \ref{lemma:transport}. $\omega'(I(-),(-))$ defines a symmetric form on $Z$ and we have the following:
    \[ \frac{1}{2} n^+(\omega'(I(-),(-))) = \eta_{\mathrm{HK}}(L,\tilde L) \]
    where $n^+$ denotes the positive index.
    \begin{proof}
        This follows from the proof of Lemma \ref{lemma:transport}. 
    \end{proof}
\end{lemma}

So far we computed the $0$-degree \[Q^0 = \Hom^0_{\mathbf A(I,-\omega_K^{-1},\omega_K/2)}(N(L,b,\alpha),N(\tilde L, \tilde b, \tilde \alpha)) \cong \mathcal S((\Gamma^*_L \times \Gamma^*_{\tilde L})/ \Gamma^* \times (\Lambda_L \cap \Lambda_{\tilde L})^* \times C_1).\] The full space has the form
\[ Q^\bullet := (\mathcal S((\Gamma^*_L \times \Gamma^*_{\tilde L})/ \Gamma^* \times (\Lambda_L \cap \Lambda_{\tilde L})^* \times C_1) \otimes \Wedge^{(0,\bullet)} V^*, \nabla^{(0,1)}) \]
by \eqref{eq:quotient}, where $\nabla$ is obtained by descending the connection \eqref{eq:connection} on $N(L \times \tilde L, (-b,\tilde b), -\alpha \otimes 1 + 1 \otimes \tilde \alpha)$ to the quotient and transporting through the isomorphism given in Lemma \ref{lemma:LemmaIdent}. We record the following two facts about $\nabla$:
\begin{enumerate}
    \item \label{formula:res1} The restriction $\nabla \colon [L \cap \tilde L \oplus IL \cap I \tilde L] \times Q^0 \to Q^0$ is given by
    \begin{align*} \nabla_{(l,Il')}( u_{\gamma} \otimes u_{\gamma'} \otimes g ) &= \partial_l u_{\gamma} \otimes u_{\gamma'} \otimes g + u_{\gamma} \otimes \partial_{\omega_K(Il')} u_{\gamma'} \otimes g   \\ &\quad + 2\pi \ii [ \omega_K(Il',\tilde b - b) + (\tilde \alpha - \alpha)(l)] u_{\gamma} \otimes u_{\gamma'} \otimes g 
    \end{align*}
    \item \label{formula:res2}The full map $\nabla \colon V \times Q^0 \to Q^0$ has the form
    \[ \nabla_{v} (u_{\gamma} \otimes u_{\gamma'} \otimes g) = \nabla_v(u_{\gamma} \otimes u_{\gamma'}) \otimes g + u_{\gamma} \otimes u_{\gamma'} \otimes \tilde \nabla_{[v]} g\]
    where $\tilde \nabla \colon V/(L \cap \tilde L \oplus IL \cap I \tilde L) \times \mathcal S(C_1) \to \mathcal S(C_1)$ is given by 
    \[ \tilde \nabla_{[u]}(g)(v) =  \partial_{\pi_{C_1}(\Psi(u))}g(v) + 2\pi \ii \omega'(\Psi(u),v) g(v).\]
\end{enumerate}
\begin{definition}\label{def:sdrw}
    Let $C$ be a real vector space.
    Let $\alpha$ be a polynomial closed $1$-form on $C$. The \emph{Schwartz-de-Rham-Witten} complex of $\alpha$ is the real cochain complex
    \[ \mathcal S_{DRW}^\bullet(C, \alpha) := (\mathcal S_{\mathbb R}(C) \otimes \Wedge^\bullet C^*, d_{DR} + \alpha \wedge (-) ). \qedhere\]
\end{definition}
Note that such a complex has previously appeared in a different context in \cite{Prokhorenkov1999}.
\begin{definition}
    \label{def:correctedLagrange}
    Fix an $I$-totally real, $\omega_J$-Lagrangian subspace $L_{\mathrm{pt}} \subset V$. Given an affine Lagrangian $(L,b,\alpha)$ we define the \emph{corrected Lagrangian module}
    \[ \tilde N(L,b,\alpha) := N(L,b,\alpha)[-\tau(L)] \]
    where $\tau(L) = \eta_{\mathrm{HK}}(L,L_{\mathrm{pt}})$.
\end{definition}

\begin{proposition}
    \label{prop:SpectralFloer}
    There is a spectral sequence $E^{pq}_*$ converging to the cohomology \[ E^{pq}_* \Rightarrow H^\bullet(\Hom^\bullet_{\mathbf A(I,-\omega_K^{-1},\omega_K/2)}(N(L,b,\alpha),N(\tilde L, \tilde b, \tilde \alpha))).\]
    Its first page is 
    \[ E_1^{pq} = H^q(L_b \cap \tilde L_{\tilde b}, \mathcal L_{\tilde \alpha - \alpha}) \otimes_{\mathbb R} \mathcal S^p_{DRW}(C_1, dh) \]
    where $h$ is a quadratic function on $C_1$ whose negative index is equal to the Hyperkähler Maslov index $\eta_{\mathrm{HK}}(L, \tilde L)$.
    It collapses on the second page to
    \[ E_2^{pq} = \begin{cases}
        &H^q(L_b \cap \tilde L_{\tilde b},\mathcal L_{\tilde \alpha - \alpha}) \text{, if } p = \eta_{HK}(L,\tilde L) \\
        &0 \text{, else. }
    \end{cases}\]
    \begin{proof}
        On the exterior algebra $\Wedge^{(0,\bullet)} V^*$ we introduce the decreasing filtration
        \[ F^i [\Wedge^{(0,\bullet)} V^*] = \Wedge^{(0,\geq i)} (V/(L \cap \tilde L \oplus IL \cap I \tilde L)^* \wedge \Wedge^{(0,\bullet)} V^*.\]
        Its filtered pieces are \[F^i / F^{i+1} = \Wedge^{(0,i)} (V/(L \cap \tilde L \oplus IL \cap I \tilde L)^* \otimes \Wedge^{(0,\bullet)} (L \cap \tilde L \oplus IL \cap I \tilde L)^*.\]
        $F$ induces a (bounded) filtration also on $Q^\bullet$:
    \[ F^i Q^\bullet = \mathcal S((\Gamma^*_L \times \Gamma^*_{\tilde L})/ \Gamma^* \times (\Lambda_L \cap \Lambda_{\tilde L})^* \times C_1) \otimes F^i [\Wedge^{(0,\bullet)} V^*] \]
        The differential $d = \nabla^{(0,1)}$ on $Q^\bullet$ is compatible with it. The associated spectral sequence has terms
        \begin{align*}
            E_0^{pq} &= [\mathcal S((\Gamma^*_L \times \Gamma^*_{\tilde L})/ \Gamma^* \times (\Lambda_L \cap \Lambda_{\tilde L})^*) \otimes \Wedge^{(0,q)} (L \cap \tilde L \oplus IL \cap I \tilde L)^*] \\
            & \quad \hat \otimes [\mathcal S(C_1) \otimes  \Wedge^{(0,p)} (V/(L \cap \tilde L \oplus IL \cap I \tilde L)^*]
        \end{align*}
        with differential $d_0 \colon E_0^{pq} \to E_0^{p(q+1)}$ given by \ref{formula:res1}, in particular acting only on the first factor. Now \[\mathcal S((\Gamma^*_L \times \Gamma^*_{\tilde L})/ \Gamma^* \times (\Lambda_L \cap \Lambda_{\tilde L})^*) \otimes \Wedge^{(0,\bullet)} (L \cap \tilde L \oplus IL \cap I \tilde L)^*\] is isomorphic to \[\oplus_{t \in T} (\mathcal S((\Gamma_L \cap \Gamma_{\tilde L})^* \times (\Lambda_L \cap \Lambda_{\tilde L})^*) \otimes \Wedge^{(0,\bullet)} (L \cap \tilde L \oplus IL \cap I \tilde L)^*, \nabla^{(0,1)}|_{L \cap \tilde L \oplus IL \cap I \tilde L})\]
        by equation \ref{eq:splittorsion}.
        The summands are isomorphic to Dolbeault complexes with twisted differentials:
    \[ \Omega^{(0,\bullet)}( (L \cap \tilde L) / (\Gamma_L \cap \Gamma_{\tilde L}) \times (IL \cap I \tilde L) / \omega_K^{-1}(\Lambda_L \cap \Lambda_{\tilde L})), \bar \partial + 2\pi \ii (\tilde \alpha - \alpha + \omega_K(b-\tilde b))^{(0,1)} ). \]
    Its cohomology is hence
    \begin{equation} = \begin{cases}
        &\Wedge^\bullet  (L \cap \tilde L)_{\mathbb C} \text{, if } (\tilde \alpha - \alpha)|_{L \cap \tilde L} \in (\Gamma_L \cap \Gamma_{\tilde L})^* \text{ and } L_b \cap \tilde L_{\tilde b} \neq \emptyset \\
        & 0 \text{, else. }
    \end{cases}
    \end{equation}
    Moreover, as $T$ parametrises the connected components of $L_b \cap \tilde L_{\tilde b}$, we get 
    \[ H^\bullet(L_b \cap \tilde L_{\tilde b}, \mathcal L_{\tilde \alpha - \alpha}) \]
    after summing over $T$ (here $\mathcal L_{\tilde \alpha - \alpha}$ denotes the local system associated to the $1$-form $\tilde \alpha - \alpha$). Hence, the first page of the spectral sequence reads
    \[ E_1^{pq} = H^q(L_b \cap \tilde L_{\tilde b}, \mathcal L_{\tilde \alpha - \alpha}) \otimes [\mathcal S(C_1) \otimes  \Wedge^{(0,p)} (V/(L \cap \tilde L \oplus IL \cap I \tilde L)^*] \]
    and the differential $E_1^{pq} \to E_1^{(p+1)q}$ is induced by $\tilde \nabla$, see \ref{formula:res2}, and acts on the latter factor only. Recall that 
    \[ C_1 \oplus C_2 = Z \cong V/(L \cap \tilde L \oplus IL \cap I \tilde L) \]
    is a Lagrangian splitting for $\omega'$ and $\omega'$ is of type $(1,1)$ for $I$. 
    We can choose $C_1$ to be totally real for $I$ and $C_2 := IC_1$.
    Then 
    \[ \mathcal S(C_1) \otimes  \Wedge^{(0,\bullet)} (V/(L \cap \tilde L \oplus IL \cap I \tilde L)^* \cong (\mathcal S_{\mathbb R}(C_1) \otimes_{\mathbb R} \Wedge^\bullet  C_1) \otimes_{\mathbb R} \mathbb C \]
    where the differential on $\mathcal S_{\mathbb R}(C_1) \otimes_{\mathbb R} \Wedge^\bullet  C_1$ reads
    \[ d \colon C_1 \times \mathcal S_{\mathbb R}(C_1) \to \mathcal S_{\mathbb R}(C_1), \quad d_c(g)(v) = \partial_c g(v) - 2\pi \omega'(Ic,v) g(v).\]
    Note that $h := \omega'(I(-),(-))$ is a symmetric non-degenerate bilinear form on $C_1$. Diagonalise $h$ as
    \[ A h A^{*} = \mathrm{diag}(-1,...,-1,1,...,1). \]
    Let $r$ be the number of negative eigenvalues and $s$ the number of positive eigenvalues. By Lemma \ref{lemma:Maslov}, $s$ coincides with the Hyperkähler Maslov index $\eta_{\mathrm{HK}}(L_1,L_2)$. Define the complexes
    \[ S^- := \mathcal S_{\mathbb R}(\mathbb R) \xrightarrow{\partial_x + 2\pi x } \mathcal S_{\mathbb R}(\mathbb R),\quad S^+ := \mathcal S_{\mathbb R}(\mathbb R) \xrightarrow{\partial_x - 2\pi x } \mathcal S_{\mathbb R}(\mathbb R).\]
    Using $A$, we obtain an isomorphism of Fréchet complexes $\mathcal S_{\mathbb R}(C_1) \otimes_{\mathbb R} \Wedge^\bullet  C_1 \cong (S^-)^{\hat \otimes r} \hat \otimes (S^+)^{\hat \otimes s} $.
    There are continuous homotopy equivalences
    \[ S^- \simeq (\mathbb R \to 0) = \mathbb R,\quad S^+ \simeq (0 \to \mathbb R) =\mathbb R[1].\]
    We conclude that 
    \[ \mathcal S(C_1) \otimes  \Wedge^{(0,\bullet)} (V/(L \cap \tilde L \oplus IL \cap I \tilde L)^* \simeq \mathbb C[1]^{\otimes s} =  \mathbb C[s]\]
    so
    \[ E_2^{pq} = \begin{cases}
        &H^q(L_b \cap \tilde L_{\tilde b},\mathcal L_{\tilde \alpha - \alpha}), \text{ if } p = s \\
        &0, \text{ else. }
    \end{cases} \qedhere\]
    \end{proof}
    
\end{proposition}
\begin{theorem}
    \label{thm:FloerCohomology}
    Let $(L,b,\alpha)$ and $(\tilde L, \tilde b, \tilde \alpha)$ be affine Lagrangian objects. Then the derived Hom between $\tilde N(L,b,\alpha)$ and $\tilde N(\tilde L,\tilde b, \tilde \alpha)$ coincides with Floer cohomology:
    \[ H^\bullet(\Hom^\bullet_{\mathbf A(I,-\omega_K^{-1},\omega_K/2)}(\tilde N(L,b,\alpha),\tilde N(\tilde L, \tilde b, \tilde \alpha))) \cong HF^{\bullet}( (L,b,\alpha), (\tilde L, \tilde b, \tilde \alpha) ).\]

    \begin{proof}
        From the spectral sequence in Proposition \ref{prop:SpectralFloer} we find that
        \[
        H^\bullet(\Hom^\bullet_{\mathbf A(I,-\omega_K^{-1},\omega_K/2)}( N(L,b,\alpha),N(\tilde L, \tilde b, \tilde \alpha))) \cong H^{\bullet-\eta_{\mathrm{HK}}(L,\tilde L)}(L_b \cap \tilde L_{\tilde b},\mathcal L_{\tilde \alpha - \alpha})\]
        whereas 
        \begin{equation*} HF^\bullet((L,b,\alpha),(\tilde L, \tilde b, \tilde \alpha)) = H^{\bullet-\eta_{\mathrm{Fuk}}(L,\tilde L)}(L_b \cap \tilde L_{\tilde b},\mathcal L_{\tilde \alpha - \alpha}),\end{equation*}
        according to equation \ref{eq:fukfloer}.
        Hence by Definition \ref{def:correctedLagrange} and Corollary \ref{cor:compMaslov} below it follows that
        \[ H^\bullet(\Hom^\bullet_{\mathbf A(I,-\omega_K^{-1},\omega_K/2)}(\tilde N(L,b,\alpha),\tilde N(\tilde L, \tilde b, \tilde \alpha))) \cong HF^{\bullet}( (L,b,\alpha), (\tilde L, \tilde b, \tilde \alpha) ). \qedhere\]
        
    \end{proof}
\end{theorem}

\subsection{Maslov index}
Here we will compare the grading of Floer cohomology arising in the (affine) Fukaya category of a symplectic torus $(V/\Gamma,\omega)$ as used in \cite{FukayaPre}, and the grading arising in Section \ref{section:floerdef}. 
\begin{definition}[{\cite[Appendix]{KashiwaraSchapira}}]
    Let $L_0,L_1,L_2 \subset V$ be Lagrangian subspaces. Define the quadratic form
    \[ q \colon L_1 \oplus L_2 \oplus L_3 \to \mathbb R, \quad q(x_1,x_2,x_3) = \omega(x_1,x_2) + \omega(x_2,x_3) + \omega(x_3,x_1). \]
    The \emph{negative Kashiwara index} is defined as
    \[ \tau^-(L_1,L_2,L_3) := n^-(q) \]
    where $n^-$ is the negative index of $q$ (i.e., the number of negative eigenvalues).
\end{definition}
We fix a reference Lagrangian subspace $L_{\mathrm{pt}} \subset V$. 
\begin{definition}[{\cite[Lemma-Definition 2.12]{FukayaPre}}]
\label{def:fukmaslov}
    Let $L_1,L_2 \subset V$ be Lagrangian subspaces. Define \emph{Fukaya's Maslov index} (relative to $L_{\mathrm{pt}}$) as 
    \[ \eta_{\mathrm{Fuk}}(L_1,L_2) := 2n - \tau^-(L_1,L_2,L_{\mathrm{pt}}) - \dim(L_2 \cap L_{\mathrm{pt}}) - \dim(L_1 \cap L_{2}). \]
    where $2n = \dim_{\mathbb R}(V)$.
\end{definition}
Note that this is the definition from the preprint \cite{FukayaPre}, not the published version \cite{Fukaya02}.
Let $g,I,J,K$ be a Hyperkähler structure on $V$ with $\omega_J = \omega$ and suppose that $L_{\mathrm{pt}}$ is totally real for $I$.
Recall that we defined the Hyperkähler Maslov index $\eta_{\mathrm{HK}}$ in this situation, given in Definition \ref{Def:HK2}.
\begin{proposition}
        Let $\eta \in \{\eta_{\mathrm{HK}},\eta_{\mathrm{Fuk}} \}$. Then
        \begin{enumerate}
            \item $\tau^-(L_1,L_2,L_3) = \eta(L_2,L_1) + \eta(L_3,L_2) + \eta(L_1,L_3)$,
            \item $\eta(L_1,L_2) + \eta(L_2,L_1) = n - \dim(L_1 \cap L_2)$.
        \end{enumerate}
        \begin{proof}
            The statements for $\eta = \eta_{\mathrm{Fuk}}$ follow from \cite[Lemma-Definition 2.12]{FukayaPre}. Let us now show the identities for $\eta = \eta_{\mathrm{HK}}$. The equality $\eta_{\mathrm{HK}}(L_1,L_2) + \eta_{\mathrm{HK}}(L_2,L_1) = n - \dim(L_1 \cap L_2)$ follows from the definition of $\eta_{\mathrm{HK}}$ (as the kernel of the quadratic form is $L_1 \cap L_2 \oplus IL_1 \cap IL_2$). It remains to show that
            \[\tau^-(L_1,L_2,L_3) = \eta_{\mathrm{HK}}(L_2,L_1) + \eta_{\mathrm{HK}}(L_3,L_2) + \eta_{\mathrm{HK}}(L_1,L_3).\]
            We define on $V \oplus V \oplus V$ the quadratic form 
            \[ t(v_1,v_2,v_3) = \omega((\pi_{L_1}-\pi_{L_2})(v_1),v_2) + \omega((\pi_{L_2}-\pi_{L_3})(v_2),v_3) + \omega((\pi_{L_3}-\pi_{L_1})(v_3),v_1).\] We decompose 
            \[ V \oplus V \oplus V = (L_1 \oplus IL_1) \oplus (L_2 \oplus IL_2) \oplus (L_3 \oplus IL_3). \]
            We have 
            \begin{align*}
                t((l_1+il_1'),(l_2+il_2'),(l_1+il_2') = &\omega(l_1,l_2) + \omega(l_2,l_3) + \omega(l_3,l_1) \\
                &+ \omega(l_1',l_2') + \omega(l_2',l_3') + \omega(l_3',l_1').
            \end{align*}
            Hence the negative index of $t$ is 
            \[ n^-(t) = 2 \tau^-(L_1,L_2,L_3). \]
            Now define an isomorphism $\Psi \colon V  \oplus V \oplus V \to V \oplus V \oplus V $
            \begin{align*}
                \Psi = \frac{1}{2} \begin{pmatrix}
                    1 & 1 & -1 \\
                    1 & -1 & -1 \\
                    1 & -1 & 1
                \end{pmatrix} = \begin{pmatrix} 1 & 0 & 1 \\
                1 & -1 & 0 \\
                0 & -1 & 1\end{pmatrix}^{-1}.
            \end{align*}
            Writing $\Psi(v_1,v_2,v_3) = (u_1,u_2,u_3)$ we find that 
            \begin{align*}
                t(v_1,v_2,v_3) = \omega((\pi_{L_1}-\pi_{L_2})(u_1),u_1) + \omega((\pi_{L_2}-\pi_{L_3})(u_2),u_2) + \omega((\pi_{L_3}-\pi_{L_1})(u_3),u_3)
            \end{align*}
            and so 
            \[ 2 \tau^-(L_1,L_2,L_3) = n^-(t) = 2 (\eta_{\mathrm{HK}}(L_2,L_1) + \eta_{\mathrm{HK}}(L_3,L_2) + \eta_{\mathrm{HK}}(L_1,L_3)). \qedhere \]
        \end{proof}
\end{proposition}
\begin{corollary}
\label{cor:compMaslov}
        The indices $\eta_{\mathrm{HK}}$ and $\eta_{\mathrm{Fuk}}$ are cohomologous, i.e., there exists for every $L$ a number $\tau(L) \in \mathbb Z$ such that $\eta_{\mathrm{HK}}(L_1,L_2) = \eta_{\mathrm{Fuk}}(L_1,L_2) + \tau(L_1) - \tau(L_2)$.
        \begin{proof}
            By the first part of the proposition we have
            \[ \eta_{\mathrm{HK}}(L_1,L_2) + \eta_{\mathrm{HK}}(L_2,L_3) + \eta_{\mathrm{HK}}(L_3,L_1) = \eta_{\mathrm{Fuk}}(L_1,L_2) + \eta_{\mathrm{Fuk}}(L_2,L_3) + \eta_{\mathrm{Fuk}}(L_3,L_1).\]
            Apply this to $L_3 = L_{\mathrm{pt}}$ to get
            \[ \eta_{\mathrm{HK}}(L_1,L_2) = \eta_{\mathrm{Fuk}}(L_1,L_2) + \eta_{\mathrm{Fuk}}(L_2,L_{\mathrm{pt}}) + \eta_{\mathrm{Fuk}}(L_{\mathrm{pt}},L_1) - \eta_{\mathrm{HK}}(L_2,L_{\mathrm{pt}}) - \eta_{\mathrm{HK}}(L_{\mathrm{pt}},L_1).\]
            By the second part of the proposition we obtain 
            \[ \eta_{\mathrm{HK}}(L_1,L_2) = \eta_{\mathrm{Fuk}}(L_1,L_2) + \eta_{\mathrm{Fuk}}(L_2,L_{\mathrm{pt}})-\eta_{\mathrm{HK}}(L_2,L_{\mathrm{pt}}) -\eta_{\mathrm{Fuk}}(L_1,L_{\mathrm{pt}})  + \eta_{\mathrm{HK}}(L_1,L_{\mathrm{pt}})\]
            which has the desired form. In fact, $\eta_{\mathrm{Fuk}}(L,L_{\mathrm{pt}}) = 0$ so it suffices to let $\tau(L) = \eta_{\mathrm{HK}}(L,L_{\mathrm{pt}})$.
    \end{proof}
\end{corollary}

\bibliographystyle{alphaurl}
\bibliography{b1refsupdated}

\setlength{\parskip}{0mm plus0.4mm minus0.4mm}

\end{document}